\documentclass{article}
\usepackage{inputenc,amsfonts,amsmath,amssymb}
\usepackage[all]{xy}
\usepackage{amsthm}
\usepackage[colorlinks=true, linkcolor=blue, citecolor=black!50!green, urlcolor=blue]{hyperref}
\usepackage[left=2.8cm, right=2.8cm, top=2cm]{geometry}
\usepackage{extarrows}
\usepackage{graphicx,subcaption}
\usepackage{wrapfig}
\usepackage{multirow}
\usepackage{tikz-cd}
\usepackage{sectsty}
\usepackage{authblk}

\numberwithin{equation}{section}
\sectionfont{\normalfont\fontsize{15}{15}\selectfont}

\usepackage{aliascnt}
\newcommand{\thdef}[2]{
	\newaliascnt{#1}{theorem}  
	\newtheorem{#1}[#1]{#2}
	\aliascntresetthe{#1}  
	\newtheorem*{#1*}{#2}
	
    \expandafter\newcommand\expandafter{\csname #1autorefname\endcsname}{#2}
    
}

\newcommand{\edge}{\,
	\begin{tikzpicture}
		\draw (0,.5ex) -- (2ex,.5ex);
		\draw[fill] (0,.5ex) circle (.2ex);
		\draw[fill] (2ex,.5ex) circle (.2ex);
		\draw[fill] (0,0) circle (0);
	\end{tikzpicture}\,
}

\def \RR {{\mathbb R}}         
\def \CC {{\mathbb C}}
\def \ZZ {{\mathbb Z}}

\def \TT {{\mathbb T}}

\def \QQ {{\mathbb Q}}

\def \calC  {{\mathcal{C}}}

\def \calE  {{\mathcal{E}}}

\def \calG  {{\mathcal{G}}}
\def \calO  {{\mathcal{O}}}
\def \calH  {{\mathcal{H}}}

\def \calQ  {{\mathcal{Q}}}

\def \calS {{\mathcal{S}}}

\def \calK  {{\mathcal{K}}} 
\def \calW  {{\mathcal{W}}} 
\def \sL {{\scriptscriptstyle L}}
\def \g  {\mathfrak{g}}   
\def \h  {\mathfrak{h}}

\def \n  {\mathfrak{n}}
\def \m  {\mathfrak{m}}

\def \t  {\mathfrak{t}}
\def \l  {\mathfrak{l}}

\def \hs {\hspace{.2in}}
\def \hand {\hs \mbox{and} \hs}

\def \btheta {{\boldsymbol{\theta}}}
\def \bbeta {{\boldsymbol{\beta}}}

\def \blambda {\boldsymbol{\lambda}}

\def \TT {{\mathbb{T}}}

\def \mX {\mathsf{X}}
\def \mfX {\mathfrak{X}}
\def \CCsn {{(\CC^\times)^n}}
\def \bfv {{\bf v}}
\def \bfw {{\bf w}}
\def \bfh {{\bf h}}
\def \dpi {{\mathsf{d}_{\pi_0}}}
\def \XX {\mfX(\mX)}
\def \Jw {J_{\bfw}}
\def \ep {\varepsilon}
\def \bfu {{\bf u}}
\def \mH {{\mathsf{H}}}
\def \lara {\langle \,, \, \rangle}
\def \bnu {\boldsymbol{\nu}}
\def \bmu {\boldsymbol{\mu}}
\def \bsigma {{\boldsymbol{\sigma}}}
\def \bgamma {{\boldsymbol{\gamma}}}

\def \la {\langle}
\def \ra {\rangle}
\def \pist {\pi_{\rm st}}

\def \calOs {\calO^{(s_{i_1}, \ldots, s_{i_n})}}

\def \bsigma {{\boldsymbol{\sigma}}}
\def \bmu {{\boldsymbol{\nu}}}
\def \blam {{\boldsymbol{\lambda}}}
\def\calJ {{\mathcal{J}}}
\def \bbmu {[\![\mu]\!]}
\def \bbc {[\![c]\!]}
\def \mW {\mathsf{W}}

\def \Jt {J_{\btheta}}
\def \bth {{\btheta}}

\def \calM {\mathcal{M}}
\def \lb {{\rm lb}}
\def \rb {{\rm rb}}

\def \ZZone {(\ZZ_{\geq -1})^n}
\usepackage{tikz-cd,pgfplots}
\usetikzlibrary{angles,quotes}
\usetikzlibrary{calc}

\newtheorem{theorem}{Theorem}[subsection]
\thdef{lemma}{Lemma}
\thdef{proposition}{Proposition}
\thdef{corollary}{Corollary}
\thdef{conjecture}{Conjecture}
\theoremstyle{definition}
\thdef{definition}{Definition}
\thdef{example}{Example}
\thdef{remark}{Remark}
\thdef{lemma-definition}{Lemma-Definition}
\thdef{lemma-notation}{Lemma-Notation}
\thdef{notation}{Notation}
\thdef{assumption-notation}{Assumption-Notation}
\thdef{notation-lemma}{Notation-Lemma}
\thdef{definition-lemma}{Definition-Lemma}
\thdef{notation-remark}{Notation-Remark}
\thdef{definition-remark}{Definition-Remark}
\thdef{definition-notation}{Definition-Notation}
\thdef{remark-definition}{Remark-Definition}
\thdef{notation-definition}{Notation-Definition}
\usepackage[textwidth=2.0cm]{todonotes}

\title{Deformations of $\TT$-log-symplectic log-canonical Poisson structures and symmetric Poisson CGL extensions}

\author{Jiang-Hua Lu \, and \, Mykola Matviichuk}

\AtEndDocument{%
  \par
  \medskip
  \begin{tabular}{@{}l@{}}%
    \textsc{Department of Mathematics, the Hong Kong University, Pokfulam Road, Hong Kong}\\
    \textit{E-mail address}: \texttt{\href{jhlu@maths.hku.hk}{jhlu@maths.hku.hk}}
  \end{tabular}
    \par
  \medskip
  \begin{tabular}{@{}l@{}}%
    \textsc{The Chinese University of Hong Kong}\\
    \textit{E-mail address}: \texttt{\href{mykola.matviichuk@gmail.com}{mykola.matviichuk@gmail.com}}
  \end{tabular}
  }
\date{}

\begin{document}

\maketitle
\begin{abstract}
For a complex algebraic torus $\TT$, we study $\TT$-invariant Poisson deformations of a $\TT$-log-symplectic log-canonical Poisson structure $\pi_0$ on $\CC^n$. 
We show that every $\TT$-invariant first-order deformation of $\pi_0$ with linearly independent
$(\mathbb{C}^\times)^n$-weights is unobstructed.
For a special class of $\pi_0$ defined by the so-called 
{\it symmetric $\TT$-action data}, 
we show that $\pi_0$ can be canonically deformed to  symmetric 
$\TT$-Poisson CGL extensions (of $\CC$) as defined by K. Goodearl and M. Yakimov. 
As a consequence, we classify all symmetric $\TT$-Poisson CGL extensions 
in terms of their log-canonical terms
$\pi_0$  and the second $\TT$-invariant Poisson cohomology of $\pi_0$. 
We further characterize, among all symmetric Poisson CGL extensions, those {\it of Cartan type}, i.e., those 
associated to sequences of simple roots in the root systems of 
symmetrizable generalized Cartan matrices. In particular, we prove that the standard Poisson structures on Bott-Samelson cells and generalized Schubert cells for semi-simple complex Lie groups
are the (uniquely determined) maximal normalized admissible 
deformations of their log-canonical terms.
Finally, for any symmetric $\TT$-Poisson CGL extension $\pi$ with log-canonical term $\pi_0$, 
we present an explicit formula expressing the initial mutation matrix in the Goodearl-Yakimov theory on cluster algebras associated to $\pi$ in terms of the $\CCsn$-weights of the second $\TT$-invariant Poisson cohomology of  $\pi_0$.

\end{abstract}

\tableofcontents
\addtocontents{toc}{\protect\setcounter{tocdepth}{1}}

\section{Introduction and main results}

\subsection{Introduction}
For a log-canonical Poisson structure $\pi_0$ on $\CC^n$ given in the standard linear coordinates 
$(x_1, \ldots, x_n)$ by
\begin{equation}\label{eq:logcan}
\{x_j,\, x_k\}_{\pi_0} = \lambda_{j, k} x_j x_k, \hs j,k \in [1, n],  
\end{equation}
where $\lambda_{j,k} \in \CC$ for $j,k \in [1, n]$, one natural problem is to 
describe all of  its deformations into algebraic Poisson structures on $\CC^n$.
To expect a reasonable answer to this question, one clearly needs to impose some conditions on the skew-symmetric 
Poisson coefficient matrix $\blambda=(\lambda_{j,k})$. For instance, if $\blambda=0$, describing all 
the Poisson deformations of 
$\pi_0=0$ amounts to describing \textit{all} algebraic Poisson structures on $\CC^n$. One way to make the deformation problem tractable is to require that $\det(\blambda)\not=0$, in which case  $\pi_0$ is 
said to be \textit{log-symplectic}. Formal Poisson deformations of log-symplectic
log-canonical Poisson structures have been thoroughly studied in \cite{Matviichuk2020}. In 
the current paper, we relax the log-symplecticity condition on $\pi_0$ to the so-called {\it $\TT$-log-symplecticity}, where $\TT$ is a complex algebraic torus, acting on $\CC^n$ by rescaling the coordinates $(x_1, \ldots, x_n)$. 
As we show in this paper, allowing $\TT$-log-symplecticity opens a pathway to much wider classes of examples, especially those  from Lie theory.
Furthermore, while we prove an unobstructedness theorem for $\TT$-invariant 
formal Poisson deformations of such a  $\pi_0$, 
our main results are on $\TT$-invariant {\it algebraic Poisson deformations} of $\pi_0$, and the tori $\TT$ in our
main applications are non-trivial.

More specifically, let $\TT$ be a complex algebraic torus acting on $\CC^n$ by rescaling the
standard linear coordinates $(x_1, \ldots, x_n)$ on $\CC^n$. We say that
a log-canonical Poisson structure $\pi_0$
as in \eqref{eq:logcan} is {\it $\TT$-log-symplectic} if $(\CC^\times)^n \subset \CC^n$ is a $\TT$-leaf of $\pi_0$,
i.e., the union of all the $\TT$-translations of a single symplectic leaf of $\pi_0$
(see $\S$\ref{ss:intro-1} or \autoref{de:T-log-symp} for an algebraic definition via the Poisson coefficient 
matrix $\blambda$ and the $\TT$-weights of the 
coordinates $(x_1, \ldots, x_n)$). When the torus $\TT$ is trivial, the notion of $\TT$-log-symplecticity agrees with that of log-symplecticity. 

For any $\TT$-log-symplectic log-canonical Poisson structure $\pi_0$ on 
$\CC^n$, by  a {\it first order $\TT$-invariant Poisson deformation of $\pi_0$} we mean an algebraic
bi-vector field on $\CC^n$ of the form  $\pi_0 + \pi_1$,
where $\pi_1$ represents a class in the (algebraic) $\TT$-invariant
second Poisson cohomology $\mH^2_{\pi_0}(\CC^n)^\TT$ of $\pi_0$. 
The first part of the paper consists of $\S$\ref{s:forml-deform} and $\S$\ref{s:alg-deform}: in  
$\S$\ref{s:forml-deform} we establish some basic properties of 
$\mH^2_{\pi_0}(\CC^n)^\TT$, and in $\S$\ref{s:alg-deform} we prove 
existence and uniqueness of {\it admissible} 
$\TT$-invariant formal, as well as algebraic, Poisson deformations 
of $\pi_0$ along certain $\pi_1$.

In the second part of the paper ($\S$\ref{s:action-data} - $\S$\ref{s:Cartan})
we apply the general theory to a special class of $\TT$-log-symplectic log-canonical Poisson structures 
$\pi_0$, namely those  {\it of $\TT$-action type}, for which $\pi_0$ is defined using the $\TT$-action 
(see \autoref{def:action-data-0}). Examples of such log-canonical Poisson structures are abundant, 
and their admissible Poisson deformations are all algebraic. When $\pi_0$ is {\it 
of symmetric $\TT$-action type} (see \autoref{def:sym-action-data}), 
we show that the admissible Poisson deformations of $\pi_0$ 
are precisely all the symmetric $\TT$-Poisson CGL extensions, as defined by  
K. Goodearl and M. Yakimov in \cite{GY:Poi-CGL}, that have  $\pi_0$ as their log-canonical term. As a consequence, 
we prove a classification 
theorem on symmetric $\TT$-Poisson CGL extensions in terms of their log-canonical term $\pi_0$ and 
the Poisson
cohomology space $\mH^2_{\pi_0}(\CC^n)^\TT$.

 The theory of symmetric Poisson CGL extensions
was developed by K. Goodearl and M. Yakimov \cite{GY:PNAS, GY:quantum-CGL, GY:Poi-CGL, Y:maximalGreen} as a systematic framework
for studying cluster structures on varieties coming from Lie theory, such as double Bruhat cells 
and Schubert cells \cite{FZ:double, GLS:partial, GY:JEMS, GY:integral}. 
Symmetric Poisson CGL extensions 
associated to double Bruhat cells and Schubert cells
of symmetrizable Kac-Moody groups, equipped with the so-called {\it standard Poisson 
structures} \cite{GY09, Harold-W:Kac-Moody},
are implied by the results in \cite{GY:JEMS, GY:integral} on the corresponding (quantum) CGL extensions.
For a complex semi-simple Lie group $G$,  using the Poisson geometry of Bott-Samelson varieties and the theory of Poisson Lie groups, it is proved in \cite{EL:BS} 
that the standard Poisson structure on every 
{\it generalized Schubert cell} of $G$ is a symmetric Poisson CGL extension in the naturally defined 
{\it Bott-Samelson coordinates} (see $\S$\ref{ss:BS-cells} for detail). 
With the viewpoint of our classification result on all symmetric Poisson CGL extensions, 
we prove in \autoref{thm:BS-cells} that the standard Poisson structure on any
generalized Schubert cell of a complex semi-simple Lie group 
is the unique {\it maximal normalized admissible} Poisson deformation of its log-canonical term
(again presented in the Bott-Samelson coordinates).

Let $A$ be any $r \times r$ symmetrizable generalized Cartan  with a given symmetrizer. Then every sequence
$\alpha_{{\bf i}} = (\alpha_{i_1},\ldots, \alpha_{i_n})$ of simple roots associated to $A$ gives rise to a
log-canonical Poisson structure $\pi_0^{(A, {\bf i})}$ of symmetric action type  (see \eqref{eq:pi0-alpha}). 
In particular, we obtain the unique maximal normalized admissible Poisson deformation $\pi^{(A, {\bf i})}$ of 
$\pi_0^{(A, {\bf i})}$  (see \autoref{nota:pi-bfi}) as a symmetric Poisson CGL extension. 
When $A$ is of finite type, we prove that the Poisson structure
$\pi^{(A, {\bf i})}$ recovers the standard Poisson structure on the generalized Schubert cell 
defined by $\alpha_{\bf i}$ studied in 
\cite{EL:BS}. 
We remark that the deformation theoretical approach in this paper is very 
elementary and relies neither on the theory of Kac–Moody groups  nor on that of Poisson–Lie groups.

Motivated by the Lie theoretical examples, we define a {\it symmetric Poisson CGL extension  
of Cartan type}  to be any (not necessarily maximal or normalized) 
admissible algebraic Poisson deformation of the $\pi_0^{(A, {\bf i})}$
for any  symmetrizable generalized Cartan matrix $A$ and 
any sequence  $\alpha_{{\bf i}} = (\alpha_{i_1},\ldots, \alpha_{i_n})$ of simple roots associated to $A$
(see $\S$\ref{ss:Cartan-intro} and $\S$\ref{s:Cartan}). 
In \autoref{pr:Cartan-type},  we present a criterion for testing
whether or not a given
symmetric Poisson CGL extension is  of Cartan type.

Recall that the main result of the Goodearl-Yakimov 
theory \cite[Theorem 11.1]{GY:Poi-CGL} says that  a (normalized) symmetric
$\TT$-Poisson CGL extension $\pi$ on $\CC^n$ 
gives rise to a cluster structure on $\CC[x_1 \ldots, x_n]$  compatible with both the
$\TT$-action and the Poisson structure $\pi$, in the sense \cite{GSV:book} that every extended cluster consists of $\TT$-weight vectors and is log-canonical with respect to $\pi$. For a symmetric $\TT$-Poisson CGL extension $\pi$ on $\CC^n$ with log-canonical term $\pi_0$, we show in this paper that there is a collection of
non-negative integers, which we call {\it Cartan integers} associated to $\pi$, that are
extracted from the non-zero $\CCsn$-weights in the Poisson cohomology classes in $\mH^2_{\pi_0}(\CC^n)^\TT$ that appear in $\pi$.  
We then give an explicit formula expressing the initial mutation matrix $M$ in the Goodearl-Yakimov theory in terms of such Cartan integers. 
 For $\pi = \pi^{(A, {\bf i})}$ of Cartan type,
the Goodearl-Yakimov initial 
mutation matrix recovers  the one constructed by L. Shen and D. Weng in \cite{ShenWeng:DBS} using string diagrams.

In separate works, we will further explore the (somewhat mysterious) connections between Poisson 
cohomology and cluster structures through deformations of Poisson structures, $g$-vectors of cluster variables, and principal coefficient extensions.

We now set up some notation in $\S$\ref{ss:nota-intro}  which will be used throughout the paper, and we give  more details on our main results in 
$\S$\ref{ss:intro-1} - $\S$\ref{ss:app-intro}.

\subsection{Notation}\label{ss:nota-intro}
Throughout the paper, for integers $j,k\in \ZZ$, we write $[j,k]$ for the set of $i\in \ZZ$ such that 
$j\leq i \leq k$.
Elements in $\ZZ^n \subset \CC^n$, for $n
\geq 1$, 
are understood to be column vectors unless otherwise specified, 
and the standard basis of $\CC^n$ is denoted by $(e_1, \ldots, e_n)$.
If an element in $\ZZ^n \subset \CC^n$ is denoted by a letter, say $\bfw$,  the 
components of $\bfw$ will be denoted by $\bfw_i$, $i \in [1,n]$, unless otherwise indicated.  
For $m \in \ZZ$, we write $(\ZZ_{\geq m})^n = \{\bfw \in \ZZ^n: \bfw_i \geq m\, \forall i \in [1, n]\}$. 
The transpose of a matrix $M$ is denoted by $M^t$.

We identify  the character lattice of the complex torus $\CCsn$ with $\ZZ^n$ in such a way that the
group homomorphism $\CCsn \to \CC^\times$  corresponding to 
$\bfw  \in \ZZ^n$ is  given by 
$\lambda \mapsto \lambda^\bfw :=\lambda_1^{\bfw_1} \cdots \lambda_n^{\bfw_n} \in \CC^\times$.
If $V$ is a complex vector space with a  $\CCsn$-action by linear isomorphisms, for $\bfw \in \ZZ^n$ we set 
\[
V^\bfw = \{v \in V: \lambda \cdot v = \lambda^\bfw v, \, \forall \, \lambda \in \CCsn\}
\]
and call $\bfw$ a $\CCsn$-weight in $V$ if $V^{\bfw} \neq 0$.
For  $\mW \subset \ZZ^n$,  set $V^\mW = 0$ if $\mW = \emptyset$, and otherwise
\begin{equation}\label{eq:VW}
V^\mW = \sum_{\bfw \in \mW} V^{\bfw}.
\end{equation}
For any algebraic torus $\TT$ acting on $V$ by linear isomorphisms, let $V^\TT$ be the $\TT$-invariant subspace of $V$.

\subsection{Poisson cohomology and formal Poisson deformations}\label{ss:intro-1}
Throughout the paper, let $n \geq 2$ and set $\mX = \CC^n$ with the standard linear coordinates 
$(x_1, \ldots, x_n)$ and
\[
\mfX(\mX) = \bigoplus_{p=0}^n \mfX^p(\mX),
\]
where  $\mfX^p(\mX)$ for $p \in [0, n]$ is the space of all algebraic $p$-vector fields on $\mX$ (see $\S$\ref{ss:cohom-log}). One of our main tools for studying formal and algebraic Poisson structures
on $\mX$ is the $\CCsn$-action on $\mX$ via
\[
\lambda \cdot (x_1, \ldots, x_n) = (\lambda_1 x_1, \ldots, \lambda_n x_n), \hs \lambda = (\lambda_1, \ldots, \lambda_n)^t \in \CCsn,
\]
and the induced decomposition of $\mfX(\mX)$ into $\CCsn$-weight spaces. It is easy to see that
for $\bfw \in \ZZ^n$ one has $\mfX(\mX)^\bfw \neq 0$ if and only if $\bfw \in \ZZone$.
Observe also that elements in $\mfX^2(\mX)$ with $\CCsn$-weight $0 \in \ZZ^n$ correspond to 
log-canonical Poisson structures on $\mX$. Let $[\, , \, ]_{\rm Sch}$ be the Schouten bracket on $\mfX(\mX)$.

Let $\TT$ be a complex algebraic torus with character lattice $X^*(\TT)$, and let $\TT$ act on
$\mX = \CC^n$ via
\begin{equation}\label{eq:t-action-intro}
t \cdot (x_1, \ldots, x_n) = (t^{\beta_1}x_1, \, \ldots, \, t^{\beta_n}x_n),
\end{equation}
where $\bbeta = (\beta_1, \ldots, \beta_n) \in X^*(\TT)^n$ 
(see $\S$\ref{ss:T-cohom-log} for notation). We also identify $X^*(\TT)$ as a lattice in $\t^*$, where 
$\t$ is the Lie algebra of $\TT$
(see again $\S$\ref{ss:T-cohom-log} for detail). 
A log-canonical Poisson structure $\pi_0$ on $\mX$, with Poisson coefficient matrix 
$\blambda = (\lambda_{j, k})$ as in \eqref{eq:logcan}, is said to be 
{\it $\TT$-log-symplectic} if the equations 
\[
\blambda \bfw = 0\in \CC^n \hs \mbox{and} \hs \bbeta \bfw = 0 \in \t^*
\]
for (a column vector) $\bfw \in \CC^n$ have only the zero solution (see \autoref{de:T-log-symp}).

Given a $\TT$-log-symplectic log-canonical Poisson structure $\pi_0$ on $\mX = \CC^n$ 
as in \eqref{eq:logcan},  
we consider in this paper  both formal and algebraic Poisson structures 
$\pi$ on $\mX$ of the form 
\[
\pi = \pi_0 + \pi_{\rm tail},
\]
where $\pi_{\rm tail}$ is $\TT$-invariant and contains no log-canonical terms. 
The first step in constructing such a $\pi$ is to consider all $\pi_0 + \pi_1 \in \mfX^2(\mX)$, where
$\pi_1\in \mfX^2(\mX)^\TT$ contains no log-canonical terms and satisfies $[\pi_0, \pi_1]_{\rm Sch} = 0$.
In other words, we consider 
$\mH_{\pi_0}^2(\mX)^\TT$, the $\TT$-invariant part of the second cohomology of  co-chain complex
$(\mfX(\mX), [\pi_0, -])$ (see \eqref{eq:H-X}) and the induced $\CCsn$-weight space decomposition
\[
\mH_{\pi_0}^2(\mX)^\TT = \mH_{\pi_0}^2(\mX)^0 + \bigoplus_{\bth \in \calS(\pi_0)} \mH_{\pi_0}^2(\mX)^\bth,
\]
where $\calS(\pi_0) \subset \ZZone\backslash\{0\}$ is defined by
\begin{equation}\label{eq:S-pi0-intro}
\calS(\pi_0) = \mbox{the set of all the non-zero} \; (\CC^\times)^n\mbox{-weights in}\; \mH_{\pi_0}^2(\mX)^\TT.
\end{equation}
 We first prove the following results  on $\calS(\pi_0)$ in \autoref{lm:HP2decomp} and \autoref{lm:ij-unique}.

\medskip
\noindent
{\bf Proposition.} 
{\it Let $\pi_0$ be any $\TT$-log-symplectic log-canonical
Poisson structure  on $\mX=\CC^n$.

1) The subset $\calS(\pi_0)$ of $(\ZZ_{\geq -1})^n$ is finite, and each $\bth \in \calS(\pi_0)$ has exactly two entries of $-1$;

2) For each $\bth  \in {\calS}(\pi_0)$ the space
$\mH_{\pi_0}^2(\mX)^\bth$ is $1$-dimensional with a basis vector uniquely represented by 
\[
V_\bth = \left(\prod_{i \neq j,k} x_i^{\bth_i}\right) \frac{\partial}{\partial x_j} \wedge \frac{\partial}{\partial x_k}
\in \mfX^2(\mX)^{\bth},
\]
where $\bth=(\bth_1, \ldots, \bth_n)^t$ and 
$1\leq j < k\leq n $ are such that $\bth_j = \bth_k = -1$.
}

\begin{notation}\label{nota:Wm}
{\rm For a finite subset $\mW$ of $(\ZZ_{\geq -1})^n$ and $m \geq 1$, let $\mW_{m}$ denote the set of all elements in $\ZZ^n$ that  can be written
as sums of $m$  elements in $\mW$, and let 
\[
\mW_{\geq m} = \bigcup_{m'\geq m} \mW_{m'}.
\]
}
\end{notation}

We prove the following result on formal Poisson deformations of $\pi_0$.

\medskip
\noindent
{\bf Theorem A} (\autoref{thm:mainDef}.) {\it Let $\pi_0$ be any $\TT$-log-symplectic log-canonical Poisson structure on $\mX = \CC^n$, and assume that 
 $\calS\subset \calS(\pi_0)$ is $\CC$-linearly independent (see 
\autoref{ld:indep} for a number of equivalent assumptions).
Let $c = (c_{\bth})_{\bth \in \calS}$ be a set of formal commuting parameters,
and let $\pi_1^\calS(c) = \sum_{\bth \in \calS} c_{\bth} V_{\bth}$. Then there exists 
$\pi_{\geq 2}^\calS(c) \in \mfX^2(\mX)^{\calS_{\geq 2}}\bbc$ (where $\calS_{\ge2}$ is the set of sums of two or more elements of $\calS$)
such that
\begin{equation}\label{eq:pic-power}
\pi^\calS(c) = \pi_0 + \pi_1^\calS(c) + \pi_{\geq 2}^\calS(c) \in \mfX^2(\mX)\bbc
\end{equation}
satisfies $[\pi^\calS(c), \pi^\calS(c)]_{\rm Sch} = 0$. Moreover, all such $\pi^\calS(c)$'s are gauge equivalent.
}

\medskip
For any linearly independent subset $\calS$ of $\calS(\pi_0)$, 
we call $\pi^\calS(c)$ in Theorem A an {\it $\calS$-admissible formal Poisson deformation}
of $\pi_0$. When $\calS(\pi_0)$ is itself $\CC$-linearly independent, we show that 
each $\pi^{\calS(\pi_0)}(c)$ in Theorem A is {\it versal} in a sense explained in 
\autoref{thm:maximal-formal}.
We remark that the miracle that Poisson deformations of $\pi_0$ along $\pi_1^\calS(c)$ as in Theorem A are unobstructed 
has been observed in \cite{Matviichuk2020} in the context of complex log-symplectic manifolds. While the methods of \cite{Matviichuk2020} heavily use techniques from homotopy algebra such $L_\infty$-algebras, our tools in
proving Theorem A are elementary and our 
strategies are adopted from \cite{MPS:quantum} where deformations of $q$-symmetric algebras are studied. 

\subsection{Algebraic Poisson deformations under Properties \eqref{eq:W0} and \eqref{eq:W1}}\label{ss:intro-2}
Turning to algebraic Poisson deformations of $\pi_0$, for
$\calS \subset \calS(\pi_0)$ and $\pi_1 \in \mfX^2(\mX)^\calS$ let 
\[
{\rm Poi}^\calS(\pi_0; \pi_1) \subset \mfX^2(\mX)^\TT
\]
be the set of all 
algebraic Poisson structures $\pi$ on $\mX = \CC^n$ of the form 
\[
\pi = \pi_0 +\pi_1 + \cdots + \pi_N,
\]
where $N \geq 1$ and $\pi_m \in \mfX^2(\mX)^{\calS_{m}}$ 
for every $m \in [1, N]$. We call any  $\pi \in {\rm Poi}^\calS(\pi_0; \pi_1)$ 
 an {\it $\calS$-admissible algebraic Poisson deformation of $\pi_0$ along $\pi_1$}. 
We also say that $\pi \in {\rm Poi}^\calS(\pi_0; \pi_1)$ is  {\it maximal} if 
$\pi_1 =\sum_{\bth \in \calS} c_\bth V_\bth$ with $c_\bth \in \CC^\times$ for  every $\bth \in \calS$.
To give sufficient conditions for the existence and uniqueness of admissible deformations of $\pi_0$, we further make the following definitions.

1) A finite subset $\mW$ of $\ZZone$ is said to have  Property \eqref{eq:W0} if 
$\mfX^0(\mX)^{\mW_{\geq 1}}=0$, equivalently, if every
element in $\mW_{\geq 1}\cap \ZZone$ has at least one $-1$ entry
(see also \autoref{def:W0} and \autoref{lm:three-conditions});

2) A finite subset $\mW$ of $\ZZone$ is said to have  Property \eqref{eq:W1} if 
$\mfX^1(\mX)^{\mW_{\geq 1}}=0$, equivalently, if every element in $\mW_{\geq 1}\cap \ZZone$ has at least two $-1$ entries.

3) An algebraic Poisson structure $\pi$ on $\mX = \CC^n$ is said to have Property \eqref{eq:W0}, respectively Property \eqref{eq:W1}, if
$\mW(\pi)$ has Property \eqref{eq:W0}, respectively Property \eqref{eq:W1}, where 
$\mW(\pi)\subset \ZZone$ denotes the set of all
non-zero $\CCsn$-weights of the monomial terms in $\pi$.

Denote by ${\rm Aut}(\mX)^\TT$ the group of all $\TT$-equivariant bi-regular automorphisms of $\mX = \CC^n$.
Our main general results on algebraic Poisson deformations of $\pi_0$ are summarized as follows.

\medskip
\noindent
{\bf Theorem B.} {\it Let $\pi_0$ be any $\TT$-log-symplectic log-canonical Poisson structure on $\mX = \CC^n$.

1) (\autoref{thm:4.8-W0} and \autoref{thm:4.9-W0})  If $\calS \subset \calS(\pi_0)$
has  Property \eqref{eq:W0}, then for any $\pi_1 \in \mfX^2(\mX)^\calS$, the set 
${\rm Poi}^\calS(\pi_0; \pi_1)$ is non-empty and is a single orbit 
under the action of a finite dimensional unipotent subgroup of ${\rm Aut}(\mX)^\TT$. 
Moreover, every $\TT$-invariant algebraic Poisson structure on $\mX$ with log-canonical term $\pi_0$ and  
Property \eqref{eq:W0} is isomorphic, via  an element in 
${\rm Aut}^\TT(\mX)$, to an element in ${\rm Poi}^\calS(\pi_0; \pi_1)$ for some 
$\calS \subset \calS(\pi_0)$ with Property \eqref{eq:W0} and some $\pi_1 \in \mfX^2(\mX)^\calS$.

2) (\autoref{thm:4.8-W1} and \autoref{thm:4.9-W1}) If $\calS \subset \calS(\pi_0)$
has  Property \eqref{eq:W1}, then  ${\rm Poi}^\calS(\pi_0; \pi_1)$ 
for each $\pi_1 \in \mfX^2(\mX)^\calS$ has exactly one element. Moreover, for any 
$\TT$-invariant algebraic Poisson structure $\pi$
on $\mX$ with log-canonical term $\pi_0$ and 
Property \eqref{eq:W1}, the set $\calS_\pi = \mW(\pi)\backslash \mW(\pi)_2$ is a subset of $\calS(\pi_0)$ with
Property \eqref{eq:W1}, and $\pi$ is the unique element in ${\rm Poi}^{\calS_\pi}(\pi_0; \pi_1)$, where 
$\pi_1$ is the $\mfX^2(\mX)^{\calS_\pi}$-component of $\pi \in \mfX^2(\mX)$ in the decomposition of $\mfX^2(\mX)$ into $\CCsn$-weight spaces.
}

\medskip

\medskip
When $\calS(\pi_0)$ has Property \eqref{eq:W1}, by writing an 
arbitrary element in $\mfX^2(\mX)^{\calS(\pi_0)}$ as
\[
\pi_1^{\calS(\pi_0)}(c) = \sum_{\bth \in \calS(\pi_0)} c_\bth V_\bth, \hs
\hs c = (c_\bth)_{\bth \in \calS(\pi_0)} \in \CC^{\calS(\pi_0)},
\]
we obtain (see \autoref{thm:max-family-W1}) the 
{\it maximal family of admissible algebraic Poisson 
deformations}
\begin{equation}\label{eq:pic-intro}
\pi^{\calS(\pi_0)}(c) = \pi_0 + \pi_1^{\calS(\pi_0)}(c) + \cdots + \pi_N^{\calS(\pi_0)}(c), 
\hs
\hs c = (c_\bth)_{\bth \in \calS(\pi_0)} \in \CC^{\calS(\pi_0)},
\end{equation}
of $\pi_0$, where $N \geq 1$ and for $m \in [2, N]$,  
$\pi_m^{\calS(\pi_0)}(c) \in \mfX^2(\mX)^{\calS(\pi_0)_m}$ and is 
homogeneous polynomial in $c$ of homogeneous degree $m$.
We further show in \autoref{cor:member-summand}
that for any $\calS \subset \calS(\pi_0)$, every $\calS$-admissible
algebraic Poisson deformation of $\pi_0$ is equal to $\pi^{\calS(\pi_0)}(c)$ for some
$c \in \CC^{\calS(\pi_0)}$ and is also a {\it direct summand} of $\pi^{\calS(\pi_0)}(c')$ for some $c' \in (\CC^\times)^{\calS(\pi_0)}$.

\subsection{Applications to symmetric Poisson CGL extensions}\label{ss:app-intro}
In the second part of the paper, we apply Theorem B to special classes of examples.
Let again $\TT$ be any complex algebraic torus with Lie algebra $\t$ and character lattice 
$X^*(\TT)\subset \t^*$, acting  on $\mX = \CC^n$ via 
$\bbeta = (\beta_1,  \ldots,  \beta_n) \in X^*(\TT)^n$ as in \eqref{eq:t-action-intro}. 

\subsubsection{Classifications of symmetric Poisson CGL extensions}\label{ss:intro-3}
We say that
a log-canonical Poisson structure $\pi_0$ on $\CC^n$ is  {\it
of $\TT$-action type} if\footnote{While 
there is no harm in removing the minus sign in the definition of $\pi_0$ and adjusting the results in the paper accordingly,
the minus sign is there to align with the special example of the standard Poisson structure on generalized Schubert cells which 
has appeared in a series of papers \cite{LM:mixed, LMi:Kostant, EL:BS} and has been implemented into a computer program.} 
\begin{equation}\label{eq:pi0-intro-0}
\pi_0 =-\sum_{j < k} \beta_j(h_k) x_jx_k \frac{\partial}{\partial x_j} \wedge \frac{\partial}{\partial x_k}
\end{equation}
for some  ${\bf h} = (h_1, \ldots, h_n) \in \t^n$ such that 
$\beta_j(h_j) \neq 0$ for every $j \in [1, n]$. 
We prove in \autoref{cor:action-pi0-log-sym} that every $\pi_0$ of $\TT$-action type
is $\TT$-log-symplectic, and we provide a method for computing $\calS(\pi_0)$ in \autoref{pr:Spi0-action}. 
In particular, it follows from \autoref{pr:Spi0-action} that $\calS(\pi_0)$ has Property
\eqref{eq:W0}. Applying 1) of Theorem B, we obtain a large class  of $\TT$-invariant
algebraic Poisson structures with log-canonical terms defined by $\TT$-action data (see \autoref{rk:new-class}).

We say that a log-canonical Poisson structure $\pi_0$ on $\mX =\CC^n$ is of {\it symmetric $\TT$-action type} if, in
addition to be given as in \eqref{eq:pi0-intro-0}, 
it also takes the form
\[
\pi_0 =\sum_{j < k} \beta_k(h_j^\prime) x_jx_k \frac{\partial}{\partial x_j} \wedge \frac{\partial}{\partial x_k}
\]
for some  ${\bf h}^\prime = (h_1^\prime, \ldots, h_n^\prime) \in \t^n$ such that 
$\beta_j(h_j^\prime) \neq 0$ for every $j \in [1, n]$. We prove in
\autoref{pr:Spi0-action-symmetric} that when $\pi_0$ is of symmetric $\TT$-action type, 
every $\bth \in \calS(\pi_0)$ is {\it negatively bordered on both sides}
in the sense that 
\begin{equation}\label{eq:bth-intro-0}
\bth =(0, \ldots, 0, -1, \bth_{j+1}, \ldots, \, \bth_{k-1}, -1, 0, \ldots, 0)^t
\end{equation}
for a unique pair $j < k$ and with $\bth_i \in \ZZ_{\geq 0}$ for all $i \in [j+1,k-1]$. 
Consequently, 
$\calS(\pi_0)$ has Property \eqref{eq:W1}.
 By
2) of Theorem B, for every $\calS \subset \calS(\pi_0)$ and every $\pi_1 \in 
\mfX^2(\mX)^{\calS}$ we have a unique $\calS$-admissible algebraic Poisson deformation of $\pi_0$ along 
$\pi_1$. The set of all such algebraic Poisson deformations of $\pi_0$ is the maximal family in
\eqref{eq:pic-intro}.

By \cite[Definition 6.1]{GY:Poi-CGL} (recalled in  \autoref{def:Poi-CGL}),  a
symmetric $\TT$-Poisson CGL extension of dimension $n$ is a special type of $\TT$-invariant 
algebraic Poisson structure on $\CC^n$ whose log-canonical term is of symmetric $\TT$-action type. 
As a main application of Theorem B, we 
prove the following classification result of symmetric Poisson CGL extensions.

\medskip
\noindent
{\bf Theorem C.} (\autoref{thm:sym-CGL} and \autoref{cor:classification-CGL})
{\it Let $\TT$ be any algebraic torus over $\CC$.

1) For any log-canonical Poisson structure $\pi_0$ on $\mX =\CC^n$ of 
symmetric $\TT$-action type and for every
$\calS \subset \calS(\pi_0)$, all $\calS$-admissible algebraic Poisson deformations of
$\pi_0$  are symmetric $\TT$-Poisson CGL extensions;

2) Every symmetric $\TT$-Poisson CGL extension $\pi$ of dimension $n$ is the unique  $\calS_\pi$-admissible algebraic Poisson deformation of
its log-canonical term 
$\pi_0$ along $\pi_1 \in \mfX^2(\mX)^{\calS_\pi}$, where
\[
\calS_\pi = \mW(\pi)\backslash \mW(\pi)_{2} \subset \calS(\pi_0), 
\]
and $\pi_1$ is the $\mfX^2(\mX)^{\calS_\pi}$-component of $\pi$ with respect to the
$\CCsn$-weight space decomposition of $\mfX^2(\mX)$;

3) Up to rescaling of the CGL generators, the assignment
$\pi \mapsto \calS_\pi$
is a one-to-one correspondence between 
symmetric $\TT$-Poisson CGL extensions  with log-canonical term
$\pi_0$ and subsets of $\calS(\pi_0)$.
}

\subsubsection{Mutation matrix and Poisson cohomology}
Let $\pi_0$ be a log-canonical Poisson structure on $\CC^n$ of symmetric $\TT$-action type. 
For  $\bth \in \calS(\pi_0)$ as \eqref{eq:bth-intro-0}, 
 write $j = {\rm lb}(\bth)$ and $k = {\rm rb}(\bth)$ and call them 
the {\it left border} and the {\it right border} of $\bth$, respectively. We define the 
{\it oriented smoothing graph} $\Gamma^+(\pi_0)$ of $\pi_0$ to be the oriented graph with vertex set $[1, n]$ and an
oriented edge $j \rightarrow k$ for each pair $(j, k)$ such that $j = {\rm lb}(\bth)$ and
$k = {\rm rb}(\bth)$ for some $\bth \in \calS(\pi_0)$ (see \autoref{de:Delta-n}). The vertex set of a connected component of $\Gamma^+(\pi_0)$
is called a {\it level set} in $[1, n]$ (associated to $\pi_0$). 
We show in \autoref{lm:bth-Cartan-integers} that for each $\bth \in \calS(\pi_0)$ as in \eqref{eq:bth-intro-0}
and every $i \in [j+1, \, k-1]$, the  integer $\bth_i\geq 0$ depends 
only on the level set $L(i)$ of $i$ and the level set $L(j)=L(k)$ of $j$ and $k$, and we call 
$-\bth_i$ the 
{\it Cartan integer} associated to the {\it interlacing levels} $L(i)$ and $L(j)=L(k)$
 (see \autoref{def:a-LL}).

Let now $\pi$ be a symmetric $\TT$-Poisson CGL extension on $\CC^n$ with log-canonical term $\pi_0$, and let $\calS_\pi \subset \calS(\pi_0)$ be as in 
Theorem C. 
 We prove in \autoref{lm:s-p-same} that the {\it successor map $s_\pi$} associated to $\pi$
defined in \cite{GY:Poi-CGL} is given by $s_\pi({\rm lb}(\bth)) = {\rm rb}(\bth)$ for $\bth \in \calS_\pi$ 
and $s_\pi(j) =+\infty$ if $j \notin {\rm lb}(\calS_\pi)$. Introducing an $n \times n$ lower triangular matrix $E_\pi$ using the 
successor map $s_\pi$ (see \eqref{eq:E-pi} for detail), and setting
$\calH_\pi$ be the integer matrix whose columns are 
the elements $\bth \in \calS_\pi$, we prove in \autoref{thm:main-M}  that the mutation matrix $M$ 
 (denoted as $\widetilde{B}_\tau$ for $\tau = {\rm id}$ in \cite[Theorem 11.1]{GY:Poi-CGL}) in the Goodearl-Yakimov initial seed 
 $({\bf y}, M)$ in $\CC(x_1, \ldots, x_n)$ associated to $\pi$ is given by
\[
M = (E_\pi)^t \calH_\pi,
\]
where $(E_\pi)^t$ is the transpose of  $E_\pi$. In particular,  entries of
$M$ are $0, \pm 1$ or $\pm 1$ times  Cartan integers associated to interlacing level sets in $[1, n]$ associated to $\pi_0$ (see  \autoref{thm:main-M} for detail).

\subsubsection{Symmetric Poisson CGL extensions of Cartan type}\label{ss:Cartan-intro}
With again $\TT$ acting on $\mX = \CC^n$ via $\bbeta = (\beta_1, \ldots, \beta_n) \in X^*(\TT)^n 
\subset (\t^*)^n$ as in \eqref{eq:t-action-intro}, 
in the last part of the paper we consider log-canonical Poisson structures $\pi_0$ 
on $\mX = \CC^n$ of the form
\begin{equation}\label{eq:pi0-pair-intro}
\pi_0 =-\sum_{j < k} \la \beta_j, \beta_k\ra  x_jx_k \frac{\partial}{\partial x_j} \wedge \frac{\partial}{\partial x_k},
\end{equation}
where $\lara$ is a symmetric bilinear form on $\t^*$ such that 
$\la \beta_j, \beta_j \ra \neq 0$ for every $j \in [1, n]$. Such a 
$\pi_0$, automatically of symmetric $\TT$-action type,  is said to be of {\it strongly symmetric $\TT$-action type}, and we call 
$(\lara, \bbeta)$ an $n$-dimensional {\it $\TT$-action pair}. In such a setting, for
$\beta \in \t^*$ with $\la \beta, \beta\ra \neq 0$ we have the 
reflection operator $s_\beta$ on $\t^*$ defined by
\[
s_\beta(\xi) = \xi -\frac{2\la \beta, \xi\ra}{\la \beta, \beta\ra}\beta = \xi -a_{{}_{\beta, \xi}} \beta, \hs \hs \xi \in \t^*,
\]
and we call $a_{{}_{\beta, \xi}} = \frac{2\la \beta, \xi\ra}{\la \beta, \beta\ra} \in \CC$ the 
{\it Cartan number} between
$\beta$ and $\xi$.
Using the sequence $\bbeta = (\beta_1,\ldots, \beta_n)$, we construct the sequence 
$\bgamma = (\gamma_1, \ldots, \gamma_n)$ in $\t^*$ by
\begin{equation}\label{eq:gammaj-intro}
\gamma_j = s_{\beta_1} \cdots s_{\beta_{{j-1}}} \beta_j \in \t^*, \hs \hs j \in [1, n].
\end{equation}
We show in \autoref{pr:lara-beta} that the set  $\calS(\pi_0)$ for $\pi_0$ in \eqref{eq:pi0-pair-intro} 
 can be entirely read off 
from the sequence $\bgamma$ via the pattern in which distinct elements appear in $\bgamma$ and the Cartan numbers between them. 

We say a symmetric $\TT$-Poisson CGL extension $\pi$ is {\it strongly symmetric} if
its log-canonical term $\pi_0$ is of the form in \eqref{eq:pi0-pair-intro}. Starting from any
$\TT$-action pair $(\lara, \bbeta)$ and defining $\pi_0$ as in \eqref{eq:pi0-pair-intro}, we then get the maximal family $\pi^{\calS(\pi_0)}(c)$, parametrized by $c \in \CC^{\calS(\pi_0)}$, of strongly symmetric $\TT$-Poisson CGL extensions.

Let now 
$A = (a_{i, j})$ be a $r \times r$ symmetrizable generalized Cartan matrix with a given symmetrizer
$(d_i)_{i \in [1, r]}$ (see $\S$\ref{ss:pair-Cartan}), and let $\{\alpha_1, \ldots, \alpha_r\}$ be
the associated simple roots. 
One then has a natural 
$r$-dimensional complex algebraic torus $\TT_A$ and a symmetric bilinear form $\lara_A$ on $\t_A^*$ 
such that $\la \alpha_i, \alpha_i\ra = 2d_i$ for each $i \in [1, r]$, where $\t_A$ is the Lie algebra o $\TT_A$.
Set $s_i = s_{\alpha_i}$ for $i \in [1, r]$.
Let ${\bf i} = (i_1, i_2, \ldots, i_n)$ be any sequence in $[1, r]$, and define 
\[
\beta_j({\bf i}) = s_{i_1}s_{i_2} \cdots s_{i_{j-1}}\alpha_{i_j}, \hs\hs j \in [1, n].
\]
Then $(\lara_A, \bbeta({\bf i}))$ is a $\TT_A$-action pair, and $\pi_0$ defined by
$(\lara_A, \bbeta({\bf i}))$ via \eqref{eq:pi0-pair-intro} is now given by
\[
\pi^{(A,{\bf i})}_0 = -\sum_{1\leq j < k \leq n} \langle \beta_j({\bf i}), \, \beta_k({\bf i})\rangle_A 
x_jx_k\frac{\partial}{\partial x_j} \wedge \frac{\partial}{\partial x_k}.
\]
The sequence 
$\bgamma$ in this case turns out to be given by 
$(\gamma_1, \, \gamma_2, \, \ldots, \, \gamma_n) = (\alpha_{i_1}, \, \alpha_{i_2}, \, \ldots, \, \alpha_{i_n})$. 
Consequently, the set $\calS(\pi_0)$ can be directly read off from
the sequence ${\bf i}$. More specifically, for $j \in [1, n]$, define $j^+ =+\infty$ if $i_{j'} \neq i_j$ for all 
$j' \in [j+1,k]$, and otherwise let
$j^+ ={\rm min}\{j' \in [j+1, k]: i_{j'} = i_j\}$. Set
\[
\calJ = \{j \in [1, n]: j^+\neq +\infty\}.
\]
We prove in $\S$\ref{ss:Poi-Cartan} that $\calS(\pi^{(A, {\bf i})}_0) = \{\bth^{(j, j^+)}: j 
\in \calJ\}$, where  for $j \in \calJ$, 
\[
\bth^{(j, j^+)}= (0, \,\ldots,\, 0, \,-1,\, a_{i_{j+1}, i_j}, \,\ldots,\, a_{i_{j^+-1}, i_j},
\, -1, \, 0, \,\ldots,\, 0)^t.
\]
Identifying $\calS(\pi^{(A, {\bf i})}_0)$ with $\calJ$  and 
$\CC^{\calS(\pi^{(A, {\bf i})}_0)}$ with $\CC^{\calJ} =\{c = (c_j)_{j \in \calJ}\}$, we then have the maximal family
\[
\pi^{(A, {\bf i})}(c) := \pi^{\calS(\pi^{(A, {\bf i})}_0)}(c), \hs c \in \CC^{\calJ},
\]
of admissible algebraic Poisson deformations of $\pi^{(A, {\bf i})}_0$, which is also the set of all
symmetric $\TT_A$-Poisson CGL extensions with log-canonical term $\pi^{(A, {\bf i})}_0$.

For an arbitrary algebraic torus $\TT$, we say that a symmetric $\TT$-Poisson CGL extension is
{\it of Cartan type} if 
the action of $\TT$ on $\CC^n$ factors through a surjective morphism $\rho: \TT \to \TT_A$ of algebraic tori
for some $r \times r$ symmetrizable generalized Cartan matrix $A$, and if 
\[
\pi = \pi^{(A, {\bf i})}(c)
\]
for some sequence ${\bf i}=
(i_1, \ldots, i_n)$ in $[1, r]$ and some $c  \in \CC^{\calJ}$. 
We prove in \autoref{pr:Cartan-type} that 
a symmetric $\TT$-Poisson CGL extension $(\CC^n, \pi)$ is  of Cartan type if and only if it log-canonical
term $\pi_0$ is defined by a distinguished (\autoref{def:disting}) and integral (\autoref{def:int})
$\TT$-action pair
$(\lara, \bbeta)$.

\subsubsection{The standard Poisson structure on Bott-Samelson cells}
Assume now that $A$ is of finite type with rank $r$, 
and let $G$ be a connected and simply connected Lie group of the same Cartan-Killing
type as $A$. Choose a maximal torus $T$ of $G$,  a Borel subgroup $B$ of $G$ containing $T$, and a labeling of the
corresponding set of simple roots such that the associated Cartan matrix is $A$. The choice of the pair $(T, B)$, together with
that of a symmetrizer $(d_i)_{i \in [1, r]}$, then give rise to a so-called {\rm standard multiplicative Poisson structure}
$\pi_{\rm st}$ on $G$, making $(G, \pi_{\rm st})$ into a Poisson Lie group. On the other hand, any sequence 
${\bf i} = (i_1, \ldots, i_n)$ in $[1, r]$ defines a so-called Bott-Samelson cell $\calO^{(s_{i_1}, \ldots, s_{i_n})}$
associated to $G$, which is isomorphic to $\CC^n$ via the {\it Bott-Samelson coordinates} $(x_1, \ldots, x_n)$
and carries the {\it standard Poisson structure} coming from the Poisson Lie group $(G, \pi_{\rm st})$
(see $\S$\ref{ss:BS-cells} for detail and for the relation between generalized Schubert cells and Bott-Samelson cells). We prove

\medskip
\noindent
{\bf Theorem D.} (\autoref{thm:BS-cells}) {\it Let $A$ be of finite type of rank $r$. For any sequence
${\bf i} = (i_1, \ldots, i_n)$ in $[1, r]$ and under the identification 
$\calO^{(s_{i_1}, \ldots, s_{i_n})} \cong \CC^n$ via Bott-Samelson coordinates, the standard Poisson structure
on $\calO^{(s_{i_1}, \ldots, s_{i_n})}$ coincides with the Poisson structure $\pi^{(A, {\bf i})}(c)$
for $c = (-\la \alpha_{i_j}, \alpha_{i_j}\ra)_{j \in [1, n], j^+ \leq n}$.}

\medskip
For any symmetrizable generalized Cartan matrix $A$, 
not necessarily of finite type, and for any sequence ${\bf i} = (i_1, \ldots, i_n)$,
the choice of $c = (-\la \alpha_{i_j}, \alpha_{i_j}\ra)_{j \in [1, n], j^+ \leq n}$ 
ensures that the symmetric Poisson CGL extension 
$(\CC^n, \pi^{(A, {\bf i})}(c))$ is normal in the sense of 
\cite[Chapter 9]{GY:Poi-CGL} and thus gives rise to a 
cluster structure on $\CC[x_1, \ldots, x_n]$ compatible with both the $\TT$-action and the Poisson structure
by \cite[Theorem 11.1]{GY:Poi-CGL}.
Applying \autoref{thm:main-M}, we show in \autoref{cor:M-Shen-Weng} that the Goodearl-Yakimov initial mutation matrix coincides with the one 
defined by L. Shen and D. Weng in \cite{ShenWeng:DBS} using string diagrams.

\subsection{Acknowledgments} J.-H. Lu's research has been partially supported by the 
Research Grants Council (RGC) of the Hong Kong SAR, 
China (GRF 17306724). M. Matviichuk thanks Zheng Hua and Hong Kong University for the hospitality during his visits. For a part of the project, M. Matviichuk was supported by the start-up fund at CUHK. We 
are grateful to Zheng Hua, Zihang Liu, Brent Pym, Travis Schedler, and Milen Yakimov for helpful and stimulating discussions.

\section{Poisson cohomology and formal Poisson deformations}\label{s:forml-deform}
\subsection{Poisson cohomology of log-canonical Poisson structures}\label{ss:cohom-log}
Throughout the paper, let 
$\mX = \CC^n$ with linear coordinates $(x_1, \ldots, x_n)$ defined by the standard basis 
$(e_1, \ldots, e_n)$ of $\CC^n$. 
A Poisson structure $\pi_0$ on $\mX$ is said to be  log-canonical in $(x_1, \ldots, x_n)$ if
 \begin{equation}\label{eq:pi0}
\pi_0 = \sum_{1\le j<k\le n} \lambda_{j, k} x_jx_k \frac{\partial}{\partial x_j} \wedge \frac{\partial}{\partial x_k},\hs \lambda_{j, k} \in \CC \; \;\mbox{for} \;\; 1 \leq j < k \leq n.
\end{equation}
 Correspondingly we
have the Poisson bracket $\{\, , \, \}_{\pi_0}$ on $\CC[x_1, \ldots, x_n]$ with
$$
\{x_j,\, x_k\}_{\pi_0} = \lambda_{j, k} x_j x_k, \hs 1\leq j < k \leq n. 
$$
 Setting $\lambda_{k, j} = -\lambda_{j, k}$
 for $1 \leq j < k \leq n$ and $\lambda_{j, j} = 0$ for $j \in [1, n]$,
the skew-symmetric matrix 
\[
\blambda = (\lambda_{j, k})_{j, k \in [1, n]}
\]
is called the \textit{coefficient matrix} of $\pi_0$.
Writing $x^f = x_1^{f_1} \cdots x_n^{f_n}$ for $f = (f_1, \ldots, f_n) \in (\ZZ_{\geq 0})^n$, 
then
\begin{equation}\label{eq:bra-f-g}
\{x^f, \; x^g\}_{\pi_0} =( f^t \blambda g) \, x^{f+g}, \hs f, \, g \in \ZZ^n.
\end{equation}
Let $\mfX(\mX)=\bigoplus_{p=0}^n \mfX^p(\mX)$ be the space of algebraic polyvector fields on $\mathsf{X}$, i.e.,
$\mfX^0(\mX) = \CC[x_1, \ldots, x_n]$ and
\[
\mfX^p(\mX) = \left\{\sum_{1\leq j_1 < \ldots < j_p \leq n} a_{j_1, \ldots, j_p} 
\frac{\partial}{\partial x_{j_1}} \wedge \cdots \wedge \frac{\partial}{\partial x_{j_p}}: \; a_{j_1, \ldots, j_p} 
\in \CC[x_1, \ldots, x_n], 
\; \forall \; 1\leq j_1 < \ldots < j_p \leq n\right\}
\]
for $p \in [0, n]$. Recall then  (see \cite[Chapter 4]{Laurent-Gengoux2013}) that 
the (algebraic)  Poisson cohomology $\mathsf{H}_{\pi_0}(\mathsf{X})$ of $(\mX, \pi_0)$
is defined as 
 the cohomology of the cochain complex
\[
\xymatrix{
0 \ar[r]&\mfX^0(\mX) \ar[r]^-{\mathsf{d}_{\pi_0}} &\mfX^1(\mX) 
\ar[r]^-{\mathsf{d}_{\pi_0}} & \mfX^2(\mX)  \ar[r]^-{\mathsf{d}_{\pi_0}} &\; \cdots\; \ar[r]^-{\mathsf{d}_{\pi_0}}& 
\mfX^n(\mX) \ar[r] & 0,
}
\]
where $\mathsf{d}_{\pi_0}(-) = [\pi_0,-]_{\rm Sch}$ and $[~,~]_{\rm Sch}$ is the Schouten bracket on 
$\mfX(\mX)$.

Consider the $(\CC^\times)^n$-action on $\mX = \CC^n$, given in the coordinates
$(x_1, \ldots, x_n)$ by
\begin{equation}\label{eq:CCsn-CCn}
\lambda \cdot x_j =\lambda_jx_j, \hs
\lambda \in \CCsn, \; j \in [1, n],
\end{equation}
and the corresponding $(\CC^\times)^n$-action on $\mfX(\mX)$. As $\pi_0$ is $(\CC^\times)^n$-invariant,
we have an induced $(\CC^\times)^n$ action on the Poisson cohomology $\mathsf{H}_{\pi_0}(\mathsf{X})$,
which we use as a tool for studying $\mathsf{H}_{\pi_0}(\mathsf{X})$. 
Consider first the decomposition of $\XX$ into $(\CC^\times)^n$-weight spaces. For $j \in [1, n]$, set
\[
\partial_j = x_j\frac{\partial}{\partial x_j} \in \mfX^1(\mX)^{0}, 
\]
where $0 = (0, \ldots, 0)^t \in \ZZ^n$, and for $J = \{j_1, \ldots, j_p\} \subset [1, n]$ with $j_1 < \cdots < j_p$, set
\[
\partial_J =\partial_{j_1} \wedge \cdots \wedge \partial_{j_p} \in \mfX^p(\mX)^{0}.
\]
For $\bfw = (\bfw_1, \ldots, \bfw_n)^t \in \ZZ^n$, let $x^{\bfw} = x_1^{\bfw_1}\cdots 
\cdots x_n^{\bfw_n}$, a Laurent monomial in $(x_1, \ldots, x_n)$. The proof of the following 
\autoref{lm:X-w} is straightforward.

\begin{lemma-definition}\label{lm:X-w}
For $\bfw \in \ZZ^n$, one has $\mfX(\mX)^\bfw \neq 0$ if and only if $\bfw \in (\ZZ_{\geq -1})^n$, and setting
\begin{equation}\label{eq:Jw}
J_\bfw = \{j \in [1, n]: \bfw_j = -1\},
\end{equation}
then for 
$p \in [0, n]$ one has 
$\mfX^p(\mX)^\bfw \neq 0$ if and only if 
$|J_\bfw|\leq p$, and in such a case
\begin{equation}\label{eq:Xw}
\mfX^p(\mX)^\bfw= {\rm Span}_{\CC}\{\,x^\bfw \partial_J: \; J_{\bf w} \subset J \subset [1, n], \; |J| = p\}
\hs \mbox{and} \hs \dim \left(\mfX^p(\mX)^\bfw\right) = \left(\!\begin{array}{c} n-|\Jw|\\ p-|\Jw|\end{array}\!\right).
\end{equation}
In particular, $\dim \mfX^p(\mX)^\bfw = 1$ if $p = |J_\bfw|$. 
For $p \in [0, n]$, we also set
\begin{equation}\label{eq:calW-p}
\calW_p= \{\bfw \in \ZZone: |J_\bfw| \leq p\} = \{\bfw \in \ZZ^n: \, \mfX^p(\mX)^\bfw\neq 0\}.
\end{equation}
Note that $\calW_0 = (\ZZ_{\geq 0})^n$ an $\calW_p \subset \calW_{p'}$ if $p \leq p'$.
We call $x^\bfw \partial_J \in \mfX^p(\mX)$, where $\bfw \in \calW_p$,  
$J_{\bf w} \subset J$, and $|J| = p$, a {\it monomial $p$-vector field} on $\mX$.
\end{lemma-definition}

We also extend the Schouten bracket to the space $\mfX_{\rm Laurent}(\mX)$ of {\it Laurent poly-vector fields} on $\mX$, i.e., finite sums of Laurent monomial
poly-vector fields of the form
$x^\bfw \partial_J$, where $\bfw \in \ZZ^n$ and $J \subset [1, n]$.
For $\bfw \in \ZZ^n$ and $K = \{k_1, \ldots, k_q\} \subset [1, n]$ with $k_1 < \cdots < k_q$, introduce
\[
\bfw \cdot \partial_K = \sum_{i=1}^q (-1)^{i+1} \bfw_{k_i} \partial_{k_1} \wedge \cdots \wedge \widehat{\partial}_{k_i} \wedge \cdots \wedge \partial_{k_q} \in \mfX^{q-1}(\mX)^{(\CC^\times)^n},
\]
where $\widehat{\partial}_{k_i}$ means that ${\partial}_{k_i}$ is omitted in the exterior product.

\begin{lemma}\label{lm:wvJK}
For $\bfu, \bfw \in \ZZ^n$ and $J, K \subset [1, n]$, one has 
\[
[x^{\bfu}\partial_{K}, \; x^\bfw \partial_J]_{\rm Sch} =  x^{\bfu +\bfw }\left((-1)^{|K|+1}(\bfw \cdot \partial_K) \wedge \partial_{J}
-\partial_K \wedge (\bfu \cdot \partial_{J}\right) \in \mfX_{\rm Laurent}(\mX)^{\bfu + \bfw}.
\]
\end{lemma}

\begin{proof}
The statement is
proved by a direct calculation.
\end{proof}

For $\bfw \in \ZZ^n$, let 
$u_{\bfw} = \sum_{j=1}^n \left(\sum_{k=1}^n  \lambda_{j,k} \bfw_k\right) \partial_j = (\partial_1, \ldots, \partial_n) \blambda \bfw \in \mfX^1(\mX)^{0}$.

\begin{lemma}\label{lm:dpi-0}
For $\bfw \in \ZZ^n$ and $V \in 
\mfX_{\rm Laurent}(\mX)^\bfw$, one has 
$[\pi_0, \, V]_{\rm Sch} = u_{\bfw} \wedge V \in \mfX_{\rm Laurent}(\mX)^\bfw$.
\end{lemma}

\begin{proof} By \autoref{lm:wvJK}, we have
\[
[\pi_0, \, V]_{\rm Sch} = \frac{1}{2}\sum_{j, k \in [1, n]} \lambda_{j,k} [\partial_j \wedge \partial_k, \, V]_{\rm Sch} 
=\frac{1}{2}\sum_{j,k \in [1, n]} \lambda_{j,k} (\bfw_k \partial_j - \bfw_j \partial_k) \wedge V
= u_{\bfw}\wedge V.
\]
\end{proof}

Turning to the Poisson cohomology $\mH_{\pi_0}(\mX)$, we have 
\begin{equation}\label{eq:H-X}
\mH_{\pi_0}(\mX) = \bigoplus_{\bfw \in (\ZZ_{\geq -1})^n} \mH_{\pi_0}(\mX)^\bfw
=\bigoplus_{\bfw \in (\ZZ_{\geq -1})^n} \frac{{\rm Ker}(\dpi|_{\mfX(\mX)^\bfw}: \mfX(\mX)^\bfw \longrightarrow
\mfX(\mX)^\bfw)}{{\rm Im}(\dpi|_{\mfX(\mX)^\bfw}: \mfX(\mX)^\bfw \longrightarrow
\mfX(\mX)^\bfw)}.
\end{equation}

\begin{notation}\label{nota:CC-I}
{\rm 
With $(e_1, \ldots, e_n)$ as the standard (column vector) basis of $\CC^n$, for  $J \subset [1, n]$, let
\begin{equation}\label{eq:CC-I}
\CC^J = \sum_{j \in J} \CC e_j \hs \mbox{and} \hs (\CC^J)_0 = \left\{\sum_{j \in J} a_je_j \in
\CC^J: \;\sum_{j \in J} a_j = 0\right\}.
\end{equation}
}
\end{notation}
For the convenience of the reader we provide a proof of the following \cite[Lemma 3.2]{MPS:quantum}.

\begin{proposition}\label{pr:Hw-0}
For  $\bfw \in (\ZZ_{\geq -1})^n$, one has  $\mH_{\pi_0}(\mX)^\bfw\neq 0$ if and only if 
\begin{equation}\label{eq:contributingLinAlg}
\blambda \bfw \in \CC^{\Jw}, \hs \mbox{equivalently}, \hs \blambda \bfw \in  (\CC^{\Jw})_0,
\end{equation}
and in such a case, $\dpi|_{\mfX(\mX)^\bfw} = 0$ and $\mH_{\pi_0}(\mX)^\bfw \cong \mfX(\mX)^\bfw$.
\end{proposition}

\begin{proof} Let $\bfw \in (\ZZ_{\geq -1})^n$. We have $\mfX(\mX)^\bfw \neq 0$ by \autoref{lm:X-w}.
Write $\blambda \bfw = (u_1, \ldots, u_n)^t \in \CC^n$, so that
$u_\bfw = \sum_{j=1}^n u_j \partial_j$.
Note first that if $\blambda \bfw \in \CC^{\Jw}$, then 
$\sum_{j \in \Jw} u_j = -\bfw^t \blambda \bfw = 0$, so 
$\blambda \bfw \in \CC^{\Jw}$ if and only if $\blambda \bfw \in (\CC^{\Jw})_0$.

Assume first that \eqref{eq:contributingLinAlg} holds. 
Then $u_\bfw \in {\rm Span}_{\CC}\{\partial_j: j \in \Jw\}$.
By \eqref{eq:Xw}, 
$u_\bfw \wedge V = 0$ for every $V \in \mfX(\mX)^\bfw$. Thus 
$\dpi|_{\mfX(\mX)^\bfw} = 0$, so $\mH_{\pi_0}(\mX)^\bfw \cong \mfX(\mX)^\bfw \neq 0$.

Suppose now that \eqref{eq:contributingLinAlg} does not hold. We prove that ${\rm Ker}(\dpi|_{\mfX(\mX)^\bfw})
= {\rm Im} (\dpi|_{\mfX(\mX)^\bfw})$. To this end, choose any $i \in [1, n] \backslash \Jw$ such that
$u_i \neq 0$, and let 
$\alpha = u_i^{-1}dx_i$. 
Then $(\alpha, u_\bfw) = x_i$, and 
\begin{equation}\label{eq:alpha-V}
\iota_\alpha (u_\bfw \wedge V) = x_i V - u_\bfw \wedge \iota_\alpha V
\end{equation}
for $V \in \mfX(\mX)^\bfw$, 
where $\iota_\alpha: \mfX^p(\mX) \to \mfX^{p-1}(\mX)$ is the contraction operator by $\alpha$. 
Suppose that
$V \in \mfX(\mX)^\bfw$ and $\dpi(V) = 0$, so that $u_\bfw \wedge V = 0$ by \autoref{lm:dpi-0}. Then \eqref{eq:alpha-V}  gives
\[
x_i V  = u_\bfw \wedge \iota_\alpha V.
\]
We claim that $\iota_{\alpha} V \in x_i \mfX^{p-1}(\mX)$, so by writing $\iota_{\alpha} V = x_iV^\prime$
with 
$V' \in \mfX^{p-1}(\mX)$, we have 
\[
V' \in \mfX^{p-1}(\mX)^\bfw \hs \mbox{and} \hs 
V = u_\bfw \wedge  V'=\dpi (V').
\]
Indeed, by \eqref{eq:Xw}, $V$ is a finite $\CC$-linear combination of terms of the form $V_K =x^\bfw \partial_{\Jw}\wedge
\partial_K$ for a collection $\calK$ of $K \subset [1,n]$ such that $\Jw \cap K =\emptyset$. For every 
$K \in \calK$  such that 
 $\iota_\alpha (V_K) \neq 0$, we have $i \in K$ and $\bfw_i \geq 0$, so 
 $\iota_\alpha (V_K) = x_i V^\prime_K$, where 
 $V^\prime_K = \pm u_i^{-1} x^\bfw \partial_{\Jw}\wedge \partial_{K\backslash \{i\}})$, so the claim follows.
 \end{proof}

\begin{remark}\label{rmk:H-X}
{\rm
For $\bfw \in (\ZZ_{\geq -1})^n$ satisfying \eqref{eq:contributingLinAlg} and $V \in \XX^\bfw$, 
we also denote the image of $V$ in $\mH_{\pi_0}(\mX)^\bfw$ by $V$, and we write
$\mH_{\pi_0}(\mX)^\bfw = \XX^\bfw$.
\hfill $\diamond$
}
\end{remark}

The following elementary observation will be used several times in the paper.

\begin{lemma}\label{lm:wj-wk}
1) Suppose that  $\bfw = (\bfw_1, \ldots, \bfw_n)^t \in \CC^n$ and  $j, k \in [1, n]$, $j \neq k$, are
such that $\blam \bfw = a (e_j-e_k)$ for some $a \in \CC^\times$. Then $\bfw_j = \bfw_k$;

2) More generally, let $\bfw$ be as in 1), and suppose that 
$\bfw' = (\bfw_1^\prime, \ldots, \bfw_n^\prime)^t \in \CC^n$ and  $j', k' \in [1, n]$, $j' \neq k'$, are
such that $\blam \bfw' = a' (e_{j'}-e_{k'})$ for some $a' \in \CC^\times$. Then
$a (\bfw^\prime_j-\bfw^\prime_k) + a' (\bfw_{j'}-\bfw_{k'})=0$.
\end{lemma}

\begin{proof}
1) It follows from $0 = \bfw^t \blam \bfw = a\bfw^t (e_j-e_k)=a(\bfw_j-\bfw_k)$ that $\bfw_j = \bfw_k$.
Similarly, 2) follows from $(\bfw')^t \blambda \bfw = -\bfw^t \blambda \bfw^\prime$.
\end{proof}

\subsection{Poisson cohomology of $\TT$-log-symplectic log-canonical Poisson structures}\label{ss:T-cohom-log}
 Assume now that $\TT$ is a complex  algebraic torus, and 
 let $X^*(\TT)$ be the character lattice. For $\beta \in 
 X^*(\TT)$, we write the corresponding group homomorphism $\TT\to \CC^\times$ as $t \mapsto t^\beta$. 
 Let $\t$ be the Lie algebra of $\TT$. 
 Identifying $\beta \in X^*(\TT)$ with its differential $d_e \beta \in \t^*$ at the identity element $e \in \TT$, we also regard $X^*(\TT)$ as a lattice in $\t^*$. Fix now
$\bbeta = (\beta_1, \beta_2, \ldots, \beta_n) \in X^*(\TT)^n \subset (\t^*)^n$ 
 and let $\TT$ act on $\mX = \CC^n$ by
 \begin{equation}\label{eq:t-CCn}
 t\cdot  (x_1, x_2, \ldots, x_n)=(t^{\beta_1} x_1, \, t^{\beta_2}x_2, \,\ldots, \, t^{\beta_n}x_n).
 \end{equation}
Recall the following definition already stated in $\S$\ref{ss:intro-1}.

 \begin{definition}\label{de:T-log-symp}
A log-canonical Poisson structure $\pi_0$ on $\mX=\CC^n$ with coefficient matrix $\blambda$ is said to be {\it $\TT$-log-symplectic} if the 
system of linear equations 
\begin{equation}\label{eq:open}
\begin{cases}
    \blambda \bfw = 0 \in \CC^n, \\ 
    \boldsymbol{\beta} \bfw = 0 \in \t^*
\end{cases}
\end{equation}
for $\bfw \in \CC^n$ has only the zero solution. 
 \end{definition}
 
 \begin{remark}\label{rk:open-T-leaf}
 {\rm
Property \eqref{eq:open} is equivalent to requiring that at any $x \in \CCsn\subset \CC^n = \mX$, the tangent spaces at $x$ of the symplectic leaf of $\pi_0$
through $x$ and of the $\TT$-orbit through $x$ span $T_x\mX$. Defining a \textit{$\TT$-leaf} of $\pi_0$ to be
the union of all $\TT$-translations of a symplectic leaf of $\pi_0$, condition \eqref{eq:open} is then equivalent to 
$(\CC^\times)^n$ being an (open) $\TT$-leaf of $\pi_0$ in $\CC^n$ (see \cite{LM:T-leaves}). 
 \hfill $\diamond$
 }
 \end{remark}

\begin{remark}\label{rk:not-effective}
 {\rm
In \autoref{de:T-log-symp} the $\TT$-action on $\mX = \CC^n$ is not required to be effective. 
If $\TT^\prime$ is another algebraic torus and
if $\rho: \TT \to \TT^\prime$ is a surjective morphism of algebraic tori, then 
$\rho^*: (\t^\prime)^* \to \t^*$ is injective. By letting $\TT$ act on $\mX$ through 
$\rho$ and the $\TT^\prime$-action, a $\TT^\prime$-log-symplectic log-canonical Poisson structure $\pi_0$ 
on $\mX$ is then also $\TT$-log symplectic.
 \hfill $\diamond$
 }
 \end{remark}
 
For the rest of the section, we fix a $\TT$-action on $\mX = \CC^n$ as in \eqref{eq:t-CCn} and
 a $\TT$-log-symplectic log-canonical Poisson structure $\pi_0$ 
on $\mX$ with Poisson
coefficient matrix $\blambda$. 

Since we are only interested in $\TT$-invariant Poisson deformations of $\pi_0$,  
we consider the space $\XX^\TT$ of $\TT$-invariant algebraic polyvector fields on $\mX$.
As the $\TT$-weight for every non-zero element in $\XX^\bfw$ is $\bbeta \bfw$ for every $\bfw \in
(\ZZ_{\geq -1})^n$, we have 
\[
\mfX(\mX)^\TT = \bigoplus_{\bfw \in (\ZZ_{\geq -1})^n \,\cap \, \ker \bbeta} \XX^\bfw,
\]
where, by abusing notation,  $\bbeta$ also denotes the lattice map
\begin{equation}\label{eq:bbeta-map}
\bbeta: \;\; \ZZ^n \longrightarrow X^*(\TT), \;\; \bfw \longmapsto \bbeta \bfw,
\end{equation}
and the corresponding linear map $\CC^n \to \t^*$. In view of \autoref{pr:Hw-0}, we set 
\begin{equation}\label{eq:calW}
\calW_{\blambda, \bbeta} = \{\bfw \in (\ZZ_{\geq -1})^n: \, \blambda \bfw \in (\CC^{\Jw})_0, \; \bbeta \bfw =0\}.
\end{equation}
For $p \in [0, n]$, let $\mH_{\pi_0}^p(\mX)^\TT \subset \mH_{\pi_0}^p(\mX)$ be the space of  $\TT$-fixed vectors in
$\mH_{\pi_0}^p(\mX)$.

\begin{lemma}\label{lm:Hk-T} For every $p \in [0, n]$, the decomposition of 
$\mH_{\pi_0}^p(\mX)^\TT$ into $(\CC^\times)^n$-weight spaces is given by
\begin{equation}\label{eq:Hk-T}
\mH_{\pi_0}^p(\mX)^\TT =  \bigoplus_{\bfw \in \calW_{\blambda, \bbeta}, \; |\Jw| \leq p} \mH_{\pi_0}^p(\mX)^\bfw =
\mH_{\pi_0}^p(\mX)^{0} \oplus \; \bigoplus_{\bfw \in \calW_{\blambda, \bbeta}, \; 2\leq |\Jw| \leq p}\mH_{\pi_0}^p(\mX)^\bfw,
\end{equation}
and for every $\bfw \in \calW_{\blambda, \bbeta}$ such that $|\Jw| \leq p$, one has 
$\dim \mH^p_{\pi_0}(\mX)^\bfw = \left(\begin{array}{c} n-|\Jw|\\ p-|\Jw|\end{array}\right)$.
\end{lemma}

\begin{proof}
The first identity in \eqref{eq:Hk-T} follows from  \autoref{pr:Hw-0}. 
If $\bfw \in \calW_{\blambda, \bbeta}$ is such that $|\Jw| =0$ or $1$, then $(\CC^{\Jw})_0 =0$, and the  $\TT$-log-symplectic condition on $\pi_0$ implies that $\bfw = 0$. We thus have the second identity of \eqref{eq:Hk-T}. 
The rest of the statement follows from  \autoref{pr:Hw-0} and \eqref{eq:Xw}.
\end{proof}

Note that for $p \in [0, n]$, we have
\[
\mH_{\pi_0}^p(\mX)^{0}  = {\rm Span}_{\CC} \left\{\partial_J: J \subset [1, n], \, |J|=p\right\}.
\]
In particular, by \autoref{lm:Hk-T}, we have 
\begin{align}\label{eq:H0-T}
\mH_{\pi_0}^0 (\mX)^\TT & =\mH_{\pi_0}^0 (\mX)^{0} =\CC,\\
\label{eq:H1-T}
\mH_{\pi_0}^1 (\mX)^\TT &= \mH_{\pi_0}^1 (\mX)^{0}={\rm Span}_\CC \{\partial_j: j \in [1, n]\},\\
\label{eq:H2-T}
\mH_{\pi_0}^2 (\mX)^\TT &= \mH_{\pi_0}^2 (\mX)^{0} \;\oplus \;\bigoplus_{\bfw \in \calW_{\blambda, \bbeta}, \; |\Jw|=2} \mH_{\pi_0}^2(\mX)^\bfw.
\end{align}
As we are only interested in non-log-canonical deformations of $\pi_0$, 
we now consider only the non-zero $(\CC^\times)^n$-weights in  $\mH_{\pi_0}^2(\mX)^\TT$. 
In view of \eqref{eq:H2-T}, 
we set\footnote{The set $\calS(\pi_0)$ depends not only on $\pi_0$ but also on the $\TT$-action on 
$\CC^n$ given by $\bbeta$, but for notational simplicity we omit $\bbeta$ from the notation.} 
\begin{align}\nonumber
\calS(\pi_0) & =\mbox{the set of all the non-zero} \; (\CC^\times)^n\mbox{-weights in}\; \mH_{\pi_0}^2(\mX)^\TT\\
\label{eq:calS-pi0}
& =\{\bth \in \calW_{\blambda, \bbeta}: \,|J_\bth| = 2\} = \{\bth \in (\ZZ_{\geq -1})^n:   \,|J_{\bth}|=2, \, \blambda \bth \in (\CC^{J_{\bth}})_0, \, \bbeta 
\bth = 0\}.
\end{align}
For $\bth = (\bth_1, \ldots, \bth_n)^t \in \calS(\pi_0)$ with $\Jt = \{j, k\}$ and $j < k$,  let 
\begin{equation}\label{eq:V-theta}
V_{\bth} = x^{\btheta} \partial_{J_{\btheta}}=
\left(\prod_{i \neq j,k} x_i^{\bth_i}\right) \frac{\partial}{\partial x_j} \wedge \frac{\partial}{\partial x_k}
\in \mfX^2(\mX)^{\bth},
\end{equation}
which is uniquely defined by $\bth$. 
We now have the following reformulation of \autoref{lm:Hk-T} and \eqref{eq:H2-T}
(recall the notational convention from \autoref{rmk:H-X}).

\begin{lemma}\label{lm:HP2decomp} For any $\bth \in \calS(\pi_0)$, one has $\mH_{\pi_0}^2(\mX)^\bth = \CC V_\bth$. Moreover,
\[
\mathsf{H}^2_{\pi_0}(\mathsf{X})^\TT = \mathsf{H}^2_{\pi_0}(\mathsf{X})^{0} 
\oplus \bigoplus_{\bth \in \calS(\pi_0)} \mH_{\pi_0}^2(\mX)^\bth.
\]
\end{lemma}

\begin{lemma}\label{lm:ij-unique}
For any $j, k\in [1, n]$ and $j \neq k$, there exists at most one $\bth \in \calS(\pi_0)$ such that
$\Jt = \{j, k\}$. In particular, $\calS(\pi_0)$ is a finite subset of $(\ZZ_{\geq -1})^n$.
\end{lemma}

\begin{proof}
Let $\bth, \bth' \in \calS(\pi_0)$ be such that $\Jt = J_{\bth'}=\{j, k\}$, and write 
$\blambda \bth = a (e_j-e_k)$ and $\blambda \bth^\prime = a^\prime (e_j-e_k)$, where 
$a, a' \in \CC$. Then $\blambda (a' \bth - a \bth^\prime) = 0$ and $\bbeta (a' \bth - a \bth^\prime) =0$. The $\TT$-log-symplecticity of $\pi_0$ implies that $a' \bth - a \bth^\prime = 0$. 
It follows from $\btheta^\prime_j = \btheta_j = -1$ that $a = a'$. The
$\TT$-log-symplecticity of $\pi_0$ also implies that $a \neq 0$. Thus $\bth' = \bth$.
\end{proof}

Following a technique introduced in \cite{Matviichuk2020}, 
we now encode the set $\calS(\pi_0)$ 
using the {\it smoothing diagram} $\Delta_n(\pi_0)$ of $\pi_0$ defined as follows.
Let $\Delta_n$ be a complete graph with $n$ vertices. 
For easy visualization, we place the $n$ vertices of $\Delta_n$ on a circle and label them by $i \in [1, n]$. 

 \begin{definition}\label{de:Delta-n}
{\rm 1) Starting from the complete graph $\Delta_n$, color an edge $\ep: j\edge k$ of $\Delta_n$, 
and call it {\it smoothable} (with respect to $\pi_0$)
if there is (necessarily unique)  $\bth^\ep =(\bth^\ep_1, \ldots, \bth^\ep_n)^t \in \calS(\pi_0)$ with $J_{\bth^\ep} =\{j,k\}$. 
Additionally, if $\ep: j\edge k$ is smoothable, for each $i \notin \{j, k\}$
such that $\bth^{\ep}_{i} >0$ add $\bth^{\ep}_{i}$ 
arcs to the angle $j\edge i\edge k$ and say that 
edge $\ep$ has $\bth^{\ep}_i$ arcs at the vertex $i$. The complete graph $\Delta_n$, 
 with the smoothable (i.e. colored)  edges and the
 arcs at the vertices, is called the {\it smoothing diagram of $\pi_0$} and is denoted by $\Delta_n(\pi_0)$.

2) Denote by $\calE(\pi_0)$ the set of all smoothable edges of $\Delta_n(\pi_0)$.  
One then has the bijection
\begin{equation}\label{eq:E-S}
\calE(\pi_0) \longrightarrow \calS(\pi_0), \;\;  \ep \longmapsto \bth^\ep,\hs \mbox{and its inverse}\hs
\calS(\pi_0) \longrightarrow \calE(\pi_0), \;\; \bth \longmapsto \ep^\bth, 
\end{equation}
where $\ep^\bth: j \edge k$ if $\bth \in \calS(\pi_0)$ and $J_\bth = \{j, k\}$. 
The subgraph of $\Delta_n(\pi_0)$ consisting of vertices $\{1, 2, \ldots, n\}$ and 
the smoothable edges of $\Delta_n(\pi_0)$ is  denoted by
$\Gamma(\pi_0)$ and called the {\it smoothing graph} of $\pi_0$.

3) Introduce $\calE^+(\pi_0) = 
\{(j, k): j < k \; \mbox{and}\;  j \edge k \; \mbox{is in} \; \calE(\pi_0)\}$,
and for $(j, k) \in \calE^+(\pi_0)$, let $\bth^{(j, k)} \in \calS(\pi_0)$ be such that 
$J_{\bth^{(j, k)}} = \{j, k\}$. We thus have a bijection
\begin{equation}\label{eq:E-plus-S}
\calE^+(\pi_0) \longrightarrow \calS(\pi_0), \;\; (j, k) \longmapsto \bth^{(j, k)}.
\end{equation}
Assigning the arrow $j \rightarrow k$ for 
each $(j, k) \in \calE^+(\pi_0)$, we denote the resulted oriented graph by
$\Gamma^+(\pi_0)$ and call it the {\it oriented smoothing graph}
of $\pi_0$.
}
\end{definition}

\begin{remark}\label{rmk:smoothable-1}
{\rm 
Following \cite{MPS:hilbert}, elements in $\calS(\pi_0)$ are also called {\it smoothable weights} for $\pi_0$. 
The reason for the  use of the term {\it smoothable} will explained in \autoref{rmk:smoothable}.\hfill $\diamond$
}
\end{remark}

\begin{lemma}\label{lm:rational}
Suppose that $\ep: i \edge j$ and $\ep': j \edge k$ are two smoothable edges, where $i \neq k$,  and let
\[
\blambda \bth^{\ep} = a (e_i-e_j) \hand 
\blambda \bth^{\ep'} = a' (e_j-e_k),
\]
where $a, a' \in \CC^\times$. Then $a(1+\bth^{\ep'}_i) = a' (1+\bth^{\ep}_k)$. Consequently,
$a/a' \in \QQ_{>0}$. 
\end{lemma}

\begin{proof}
The statement follows from  \autoref{lm:wj-wk} and the facts that $\bth^{\ep'}_j = \bth^\ep_j = -1$,
$\bth^{\ep}_k \in \ZZ_{\geq 0}$, and $\bth^{\ep'}_i \in \ZZ_{\geq 0}$.
\end{proof}

\begin{lemma}\label{lm:ijk}
1) For every vertex $i \in [1, n]$, there are at most two smoothable edges containing $i$;

2) The smoothing graph of $\pi_0$ is a disjoint union of chains and cycles.

\end{lemma}

\begin{proof}
To prove 1), suppose that $\ep_s: i\edge j_s$, $s = 1, 2, 3$,  are three distinct smoothable edges, and let
$\blambda \,\bth^{\ep_s} = a_s (e_i - e_{j_s})$ for $s = 1, 2, 3$. Then 
by \autoref{lm:rational},   $\alpha_{1}/\alpha_2, \alpha_{2}/\alpha_3, \alpha_3/\alpha_1 \in \QQ_{<0}$, 
which is a contradiction. 2) is a direct consequence of 1).
\end{proof}

\begin{remark}\label{rmk:S2}
{\rm
If $\bth \in \calS(\pi_0)$, then $-\bth \notin \calS(\pi_0)$. Indeed, if $J_\bth = \{j, k\}$, then 
$\blambda \bth = a(e_j-e_k)$ for some $a \in \CC^\times$. If $-\bth \in \calS(\pi_0)$ with $J_{-\bth} = \{j', k'\}$,
then $-a(e_j-e_k) = a' (e_{j'}-e_{k'})$ for some $a' \in \CC^\times$, so $\{j, k\} = \{j', k'\}$, contradicting the fact that $J_\btheta$ and $J_{-\btheta}$ are clearly disjoint.
\hfill $\diamond$
}
\end{remark}

We now use the smoothing graph $\Gamma(\pi_0)$ to give a criterion on linear independence of subsets of $\calS(\pi_0)$. 
For a non-empty subset $\calS$ of $\calS(\pi_0)$, recall from  \autoref{nota:Wm} that
$\calS_{\geq 2}$ denotes the set of all elements in $\ZZ^n$ that can be written as sums of two or more elements of $\calS$. For a subset $\calE$ of $\calE(\pi_0)$, we say that 
$\calE$ {\it contains no cycles} if the graph $\Gamma(\pi_0)$ does not have cycles with all edges in $\calE$. 

\begin{lemma-definition}\label{ld:indep}
Let $\calS\subset \calS(\pi_0)$ and let $\calE \subset \calE(\pi_0)$ be its image under the bijection
$\calS(\pi_0) \to \calE(\pi_0)$ in \eqref{eq:E-S}. The following statements are equivalent:

1) $\calS \subset \CC^n$ is $\CC$-linearly independent;

2) $\calS \subset \ZZ^n$ is $\ZZ$-linearly independent;

3) $0 \notin \calS_{\geq 1}$;

4) $\calE$ contains no cycles.

\noindent
We say $\calS \subset\calS(\pi_0)$ is linearly independent if $\calS$ satisfies one of the above four equivalent conditions.
\end{lemma-definition}

\begin{proof}Clearly 1) implies 2) and 2) implies 3).
Note  that $0 \notin \calS$ and by \autoref{rmk:S2},  
3) is equivalent to $0 \notin \calS_{\geq 3}$.
Assume 3) and suppose that  $\calE$ contains a cycle 
\[
j_1 \edge j_2 \edge \cdots \edge j_p \edge j_{p+1}=j_1,
\]
where $p \geq 3$ and $j_1,\ldots, j_p$
are pairwise distinct.
Let $\ep_s: j_s \edge j_{s+1}$ for $s\in [1, p]$, and let 
    $$    \boldsymbol{\lambda}\boldsymbol{\theta}^{\ep_s} = a_s (e_{j_s} - e_{j_{s+1}}), 
    \hs a_s\in\mathbb{C}^\times, \; s\in [1, p].
    $$
    Then $\boldsymbol{\lambda}\sum_{s=1}^p a_s^{-1} \boldsymbol{\theta}^{\ep_s} = 0$
    and $\boldsymbol{\beta}\sum_{s=1}^p 
    a_s^{-1} \boldsymbol{\theta}^{\ep_s} = 0$,
    so $\sum_{s=1}^p a_s^{-1} \boldsymbol{\theta}^{\ep_s} = 0$
    by the $\TT$-log-symplecticity of $\pi_0$.
    Applying \autoref{lm:rational} repeatedly, we see
    that $a_s$, for $s \in [1, p]$, are positive rational multiples of each other. Thus $\sum_{s=1}^p{r_s} \boldsymbol{\theta}^{\ep_s} = 0$
    for some $r_s \in \QQ_{>0}$, $s \in [1, p]$. By clearing out the denominators, we may assume that $r_s \in \ZZ_{>0}$ 
    for each $s \in [1, p]$, and hence $0 \in \calS_{\geq 3}$, a contradiction.  Thus 3) implies 4).

Assume 4).  By \autoref{lm:ijk},  after permuting indices we may assume that each  edge in $\calE$
is of the form $j\edge j+1$ for some $j \in [1, n-1]$. We can thus identify $\calE$ with a subset $J$ of $[1, n-1]$ 
such that the edges in $\calE$ are precisely all 
$\ep_j: j\edge j+1$ for $j \in J$. For $j \in J$, let $a_j \in \CC^\times$ be such that
\[
\blambda \bth^{\ep_j} = a_j (e_j-e_{j+1}).
\]
Note also that  $\{e_1-e_2, e_2-e_3, \ldots, e_{n-1}-e_n\}$ is a basis of $(\CC^n)_0 = 
\{(\gamma_1, \ldots, \gamma_n)^t
\in \CC^n: \sum_{j=1}^n \gamma_j = 0\}$.
 Suppose that 
$\sum_{j \in J} \mu_j \bth^{\ep_j} = 0$, where $\mu_j \in \CC$ for  $j \in J$. Then 
$\sum_{j \in J} \mu_ja_j (e_j-e_{j+1}) = 0$, from which it follows that $\mu_j = 0$ for every $j \in J$. Thus 
$\calS$ is linearly independent over $\CC$. Hence 4) implies 1).
\end{proof}

\begin{example}\label{ex:elliptic-0}
{\rm
Let $\CC^\times$ act on $\CC^3$ given in the coordinates by 
$\lambda\cdot (x_1,x_2, x_3) = (\lambda x_1, \lambda x_2, \lambda x_3)$
and let
\[
\pi_0 = x_1x_2 \frac{\partial}{\partial x_1} \wedge \frac{\partial}{\partial x_2}
+ x_2x_3 \frac{\partial}{\partial x_2} \wedge \frac{\partial}{\partial x_3} 
+x_3x_1 \frac{\partial}{\partial x_3} \wedge \frac{\partial}{\partial x_1}.
\]
Then $\pi_0$ is $\CC^\times$-log-symplectic, and a direct calculation gives
\[
\calS(\pi_0) = \{\bth^{(1, 2)} =(-1, -1, 2)^t, \;\; \bth^{(1, 3)} =(-1, 2, -1)^t, \;\;
\bth^{(2, 3)} =(2, -1, -1)^t\}.
\]
The set $\calS(\pi_0)$ is encoded in the smoothing diagram below.

\begin{figure}[h]
    \centering
\pgfdeclarelayer{background layer}
\pgfdeclarelayer{foreground layer}
\pgfsetlayers{background layer,main,foreground layer}
\begin{tikzpicture}[baseline=-1ex,
rotate=360/2/3,
scale=0.9,line join = round 
    ] 
    
    \begin{pgfonlayer}{main}
\foreach \n in {1,...,3}
{
    \coordinate (v\n) at ({-90+((\n-1)*360/3)}:1.5);
}
\node[below right] at (v1) {$1$};
\node[above] at (v2) {$2$};
\node[below left] at (v3) {$3$};

\foreach \n in {1,...,3}
{
	\foreach \m in {1,...,3}
	\draw (v\n) -- (v\m);
}
\end{pgfonlayer}

\draw[thick,red,-] (v1) ++({0.4*cos(60)},{0.4*sin(60)})  arc[start angle=60,end angle=120,radius=0.4];
\draw[thick,red,-] (v1) ++({0.55*cos(60)},{0.55*sin(60)})  arc[start angle=60,end angle=120,radius=0.55];
\draw[thick,red,-] (v2) ++({0.4*cos(180)},{0.4*sin(180)})  arc[start angle=180,end angle=240,radius=0.4];
\draw[thick,red,-] (v2) ++({0.55*cos(180)},{0.55*sin(180)})  arc[start angle=180,end angle=240,radius=0.55];
\draw[thick,red,-] (v3) ++({0.4*cos(300)},{0.4*sin(300)})  arc[start angle=300,end angle=360,radius=0.4];
\draw[thick,red,-] (v3) ++({0.55*cos(300)},{0.55*sin(300)})  arc[start angle=300,end angle=360,radius=0.55];

\draw[very thick, blue] (v1)--(v2);
\draw[very thick, blue] (v2)--(v3);
\draw[very thick, blue] (v3)--(v1);

\begin{pgfonlayer}{foreground layer}
\foreach \n in {1,...,3}
{
	\draw[fill] (v\n) circle (0.03);
}
\end{pgfonlayer}
 \end{tikzpicture}
    \caption{Smoothing diagram appearing in \autoref{ex:elliptic-0}}
    \label{fig:SmDiagCycle}
\end{figure}

In this example, $\bth^{(1, 2)} + \bth^{(1, 3)}+\bth^{(2,3)}=0$, and the three edges in $\calE(\pi_0)$ form a 
triangle.
\hfill $\diamond$
}
\end{example}

By the notation set up in \eqref{eq:VW},  for $p \in [0, n]$ we have
\[
\mathsf{H}_{\pi_0}^p(\mathsf{X})^{\calS_{\geq 2}} := 
\bigoplus_{\bfw \in \calS_{\geq 2}} \mathsf{H}_{\pi_0}^p(\mathsf{X})^{\bfw}.
\]
The following \autoref{prop:indep-unobstr}  plays a key role in proving our main results 
in $\S$\ref{ss:formal-deform} and $\S$\ref{s:alg-deform} on
the existence and uniqueness of Poisson deformations of $\pi_0$. 

\begin{proposition}\label{prop:indep-unobstr}
For any $\TT$-log-symplectic log-canonical Poisson structure $\pi_0$ on $\mX = \CC^n$, and 
for any non-empty linearly independent  subset $\calS$ of $\calS(\pi_0)$, one has 
\begin{equation}\label{eq:H-2-3}
\mathsf{H}_{\pi_0}^2(\mathsf{X})^{\calS_{\geq 2}} = \mathsf{H}_{\pi_0}^3(\mathsf{X})^{\calS_{\geq 2}} =0.
\end{equation}
\end{proposition}

\begin{proof}
Let $p \in \{2, 3\}$.
As $\mathsf{H}_{\pi_0}^p(\mathsf{X})^{\calS_{\geq 2}} \subset \mathsf{H}_{\pi_0}^p(\mathsf{X})^{\TT}$ 
and $0 \notin\calS_{\geq 2}$ by  \autoref{ld:indep}, it follows from  \autoref{lm:Hk-T} that
\eqref{eq:H-2-3} is equivalent to 
\begin{equation}\label{eq:p-2-3}
\calS_{\geq 2} \cap \{\bfw \in \calW_{\blambda, \bbeta}: |J_\bfw| \in \{2, 3\}\} = \emptyset,
\end{equation}
where recall that $J_\bfw = \{j \in [1, n]: \bfw_j = -1\}$. Let again $\calE \subset \calE(\pi_0)$ be such that 
$\calS = \{\bth^\ep: \ep \in \calE\}$. By 
\autoref{ld:indep}, $\calE$ has no cycles, 
so we may assume that $\calE$ consists of all $\ep_j: j \edge j+1$ for $j$ in a subset $J$ of $[1, n-1]$. 
To prove \eqref{eq:p-2-3}, suppose  to the contrary that there exists 
$\bfw =(\bfw_1, \ldots, \bfw_n)^t \in \calS_{\geq 2} \cap\calW_{\blambda, \bbeta}$ 
with $|J_\bfw| = 2$ or $3$. Replacing $\calS$ by 
a subset of $\calS$ if necessary, we may assume that $\bfw = \sum_{j \in J} b_j\bth^{\ep_j}$, where 
$b_j \in \ZZ_{\geq 1}$ for every $j \in J$. The assumption that $\bfw \in \calS_{\geq 2}$ then implies that
 $\sum_{j \in J} b_j \geq 2$.
Note that  $|J|\geq 2$, for otherwise $J = \{j\}$ and $b_j \geq 2$, which would give
$\bfw_j = -b_j \leq -2$.
 For $j \in J$, let $a_j \in \CC^\times$ be such that
$\blambda \bth^{\ep_j} = a_j(e_j-e_{j+1})$. Let $i = {\rm min} (J_\bfw)$ and 
$l = {\rm max} (J_\bfw)$. 

{\bf Case 1.} $J_\bfw = \{i, l\}$. Then $\blambda \bfw = \sum_{j \in J}a_j b_j (e_j-e_{j+1})$ on the 
one hand, and
\[
\blambda \bfw = a (e_i-e_l) = a(e_i - e_{i+1}) + a(e_{i+1}-e_{i+2}) + \cdots + a(e_{l-1}-e_l)
\]
for some $a \in \CC^\times$ on the other. By the linear independence of the set 
$\{e_1-e_2, \ldots, e_{n-1}-e_n\}$, we have $J = \{i,i+1, \ldots, l-1\}$ and $a_jb_j = a$ for every $j \in J$.
Let
\[
{\bf u} = b_i \bth^{\ep_i} =(\bfu_1, \ldots, \bfu_n)^t
\hs \mbox{and} \hs {\bf v}=b_{i+1}\bth^{\ep_{i+1}} + \cdots 
b_{l-1} \bth^{\ep_{l-1}}=(\bfv_1, \ldots, \bfv_n)^t.
\]
Then ${\bf u} + \bfv = \bfw$ and
$\blambda \bfv = a (e_{i+1}-e_l)$. By \autoref{lm:wj-wk}, $\bfv_{i+1} = \bfv_l$. It now follows from
$-1 = \bfw_l = \bfu_l + \bfv_l$ that 
$\bfw_{i+1} = \bfu_{i+1} + \bfv_{i+1}  =-b_i + \bfv_l = -b_i -1-\bfu_l$.
As $l \geq i+2$, we have $\bfu_l \geq 0$, which gives $\bfw_{i+1} \leq -b_i-1 \leq -2$, 
contradicting the assumption that $\bfw \in (\ZZ_{\geq -1})^n$. 

{\bf Case 2.} $J_\bfw = \{i, k, l\}$ with $i < k < l$. 
Again  $\blambda \bfw = \sum_{j \in J}a_j b_j (e_j-e_{j+1})$ on the 
one hand, and
\begin{align*}
\blambda \bfw &= a(e_i-e_k) + b(e_k-e_l) \\
&=a(e_i-e_{i+1}) + \cdots + a (e_{k-1}-e_k) +b(e_k-e_{k+1}) + \cdots +b (e_{l-1} - e_l)
\end{align*}
for some $a, b \in \CC$ on the other. It follows that $J = \{i, i+1, \ldots, l-1\}$ again and that 
$a_jb_j = a$ for all $i \leq j \leq k-1$ and $a_jb_j = b$ for all $k \leq j \leq l-1$.
Let
$\bfu = \sum_{j=i}^{k-1} b_j \bth^{\ep_j}$ and $\bfv = \sum_{j=k}^{l-1} b_j \bth^{\ep_j}$.
Then $\bfw = \bfu + \bfv$,
$\blambda \bfu = a(e_i-e_k)$ and  $\blambda \bfv = b (e_k-e_l)$.
By \autoref{lm:wj-wk},  $\bfu_i = \bfu_k$ and 
$\bfv_k = \bfv_l$. On the other hand, we have $\bfu_l =\sum_{j=i}^{k-1} b_j \bth^{\ep_j}_l \geq 0$ and 
$\bfv_i =\sum_{j=k}^{l-1} b_j \bth^{\ep_j}_i\geq 0$. As $\bfu_i + \bfv_i = \bfw_i = -1$ and
$\bfu_l + \bfv_l = \bfw_l = -1$, one has
$\bfw_k = \bfu_k + \bfv_k = \bfu_i + \bfv_l = -1-\bfv_i -1-\bfu_l \leq -2$,
contradicting the assumption that $\bfw \in (\ZZ_{\geq -1})^n$. This proves \eqref{eq:p-2-3}.
\end{proof}

\subsection{Formal Poisson deformations of $\pi_0$}\label{ss:formal-deform} 

Let again $\mX = \CC^n$.  We first recall the notion of formal Poisson structures on $\mX$.

For a set $\mu = (\mu_1, \ldots, \mu_l)$ of 
formal commuting parameters, where $l \in \ZZ_{\geq 1}$, let
\begin{equation}\label{eq:mu-m}
\mfX(\mX)\bbmu = \bigoplus_{p=0}^n \mfX^p(\mX)\bbmu = \bigoplus_{p=0}^n \bigoplus_{m=0}^{\infty} \mu^m\mfX^p(\mX),
\end{equation}
where for  $m \in \ZZ_{\geq 0}$ and any subspace $\mfX \subset \mfX(\mX)$, $\mu^m\mfX$ consists of finite sums 
of the form 
\begin{equation}\label{eq:V}
\sum_{1\leq i_1\leq i_2 \leq \cdots \leq i_m\leq l} \mu_{i_1}\mu_{i_2}\cdots \mu_{i_m} V_{i_1, i_2, \ldots, i_m}
\end{equation}
with $V_{i_1, i_2, \ldots, i_m}\in \mfX$ for $1\leq i_1\leq i_2 \leq \cdots \leq i_m\leq l$.
Denote also by $[~,~]_{\rm Sch}$ the $\CC[\![\mu]\!]$-linear extension of the Schouten bracket $[~,~]_{\rm Sch}$   from $\mfX(\mX)$ 
to $\mfX(\mX)\bbmu$. One then has
\begin{equation}\label{eq:ppmm}
[\mu^m\mfX^p(\mX), \; \mu^{m'}\mfX^{p'}(\mX)]_{\rm Sch} \subset \mu^{m+m'}\mfX^{p+p'-1}(\mX),
\hs m, m' \geq 0, \, p, p' \in [0, n].
\end{equation}
By a \emph{gauge equivalence} of $\XX\bbmu$ we mean a
$\mathbb{C}[\![\mu]\!]$-algebra automorphism of $\XX\bbmu$ induced by 
a $\mathbb{C}[\![\mu]\!]$-algebra automorphism of $\mfX^0(\mX)\bbmu$ which gives
the identity automorphism on $\mfX^0(\mX) = \mfX^0(\mX)\bbmu/(\mu)$.

\begin{remark}\label{rmk:gauge}
{\rm
For $v \in \mu\mfX^1(\mX)\bbmu$, define ${\rm ad}_v$ and $\exp(v): \mfX(\mX)\bbmu \rightarrow \mfX(\mX)\bbmu$ 
 by 
\[
{\rm ad}_v(V) = [v, \, V]_{\rm Sch} \hs \mbox{and} \hs \exp(v)(V)=
\sum_{i=0}^\infty \frac{1}{i!} ({\rm ad}_v)^i (V).
\]
Then $\exp(v)$ is a gauge equivalence of $\XX\bbmu$. Composition of finitely many gauge equivalences 
of $\XX\bbmu$ is again a gauge equivalence of $\XX\bbmu$. Moreover, for a sequence 
$v_m\in \mu^m \mfX^1(\mathsf{X})[\![\mu]\!]$, $m\ge 1$, the sequence of gauge equivalences $T_m =  \exp(v_m) ... \exp(v_2) \exp(v_1)$ converges in the $(\mu)$-adic topology. Indeed, for any 
$f \in \mfX^0(\mX)\bbmu$ and 
$k \geq 1$, the $\mu^k$-term of $T_m(f)$ is independent of $m$ for all $m\ge k$. The limit of $T_m$ is thus again a gauge equivalence of $\XX\bbmu$. 
\hfill $\diamond$
}
\end{remark}
\begin{definition}\label{def:Poi-family}
{\rm By a {\it formal Poisson structure on $\mX$ with multi-parameter $\mu = (\mu_1, \ldots, \mu_l)$} we mean any element in 
\begin{equation}\label{eq:Poi-mu}
{\rm Poi}(\mX)\bbmu := \{\pi(\mu) \in \mfX^2(\mX)[\![\mu]\!]: [\pi(\mu),\, \pi(\mu)]_{\rm Sch} = 0\}.
\end{equation}
}
\end{definition}

\begin{remark}
{\rm 
Our use of ${\rm Poi}(\mX)\bbmu$ here involves an abuse of notation: we  denote by ${\rm Poi}(\mX)$
the set of all algebraic Poisson structures on $\mX$, but 
 by our definition in 
\eqref{eq:Poi-mu},
elements in ${\rm Poi}(\mX)\bbmu$ are not necessarily formal power series in $\mu$ of elements in ${\rm Poi}(\mX)$.
\hfill $\diamond$
}
\end{remark}

Let now $\pi_0$ be a $\TT$-log-symplectic log-canonical Poisson structure  on $\mX = \CC^n$ as in 
$\S$\ref{ss:T-cohom-log}, 
and recall that $\calS(\pi_0)$ is the set of all non-zero $(\CC^\times)^n$-weights in $\mH^2_{\pi_0}(\mX)^\TT$. Recall also that for each $\bth \in \calS(\pi_0)$ we have 
$V_{\bth} \in \mfX^2(\mX)^{\bth}$ defined in \eqref{eq:V-theta}. For $\calS \subset \calS(\pi_0)$ and integer $m \geq 1$, recall  from  \autoref{nota:Wm} the definitions of 
$\calS_m \subset \ZZ^n$ and $\calS_{\geq m} \subset \ZZ^n$.
Note that if $\calS$ is linearly independent, then $\calS_m \cap \calS_{m'} =\emptyset$ if $m \neq m'$.
By the notation set up in \eqref{eq:VW}, for $\mW\subset (\ZZ_{\geq -1})^n$ and $p \in [0, n]$ we have
\[
\mfX^p(\mX)^\mW := \sum_{\bfw \in \mW} \mfX^p(\mX)^{\bfw}.
\]
We now prove our results on formal Poisson deformations of  $\pi_0$.

\begin{theorem}\label{thm:mainDef}
Let $\pi_0$ be a $\TT$-log-symplectic log-canonical Poisson structure  on $\mX = \CC^n$, and assume that
$\mathcal{S}\subset\mathcal{S}(\pi_0)$ is linearly independent.
Let $c=(c_{\bth})_{\bth \in \mathcal{S}}$ be a set of  formal commuting parameters, and let
$\pi_1^\calS(c) = \sum_{\bth \in \mathcal{S}} c_{\bth} V_{\bth}$.
Then there exists a formal Poisson deformation of $\pi_0$ of the form
\begin{equation}\label{eq:defPowerSeries}
\pi^\calS(c) = \pi_0 +  \pi_1^\calS(c) + \pi_2^\calS(c) + \pi_3^\calS(c) + \cdots, 
\end{equation}
with $\pi_m^\calS(c)\in c^m \mfX^2(\mathsf{X})^{\calS_m}$ for $m\ge 2$.
Moreover, such a $\pi^\calS(c) \in {\rm Poi}(\mX)\bbc$ is unique up to a gauge equivalence of $\XX\bbc$.
\end{theorem}

\begin{proof} There is nothing to prove if $\calS =\emptyset$, so assume that $\calS \neq \emptyset$. 
With $\pi_0(c) = \pi_0$ and 
$\pi_1(c) :=\pi_1^\calS(c)$, 
we prove that there  exist $\pi_k(c)\in c^k \mfX^2(\mathsf{X})^{\calS_k}$, $k \geq 2$,  such that
\begin{equation}\label{eq:solvingDefPowerSeries}
   \left[\sum_{k=0}^{m} \pi_k(c),\;\,\sum_{k=0}^{m} \pi_k(c)\right]_{\rm Sch} = 0 
   \text{ mod } c^{m+1}\mfX^3(\mX)\bbc,\hs \forall\; m \geq 1.
\end{equation}
By looking at the $c^{m'} \mfX^3(\mX)$-component 
for each $1\leq m' \leq m$, we see that \eqref{eq:solvingDefPowerSeries} is equivalent to 
\begin{equation}\label{eq:pipi-m}
2[\pi(0), \, \pi_{m'}(c)]_{\rm Sch} + \sum_{k=1}^{m'-1} [\pi_k(c), \, \pi_{m'-k}(c)]_{\rm Sch} = 0 \in
c^{m'} \mfX^3(\mX), \hs \forall \;\; 1\leq m' \leq m. 
\end{equation}
Note that \eqref{eq:solvingDefPowerSeries} holds for $m = 1$ as $[\pi_0, \pi_1(c)]_{\rm Sch} = 0$.
Let $m \geq 2$ and assume that 
$\pi_{m'}(c)\in c^{m'}\mfX^2(\mX)^{\calS_{m'}}$ have been solved  from 
\eqref{eq:pipi-m} for every $1\leq m' \leq m-1$. We need to solve for $\pi_m(c) \in c^m\mfX^2(\mX)^{\calS_m}$ from
\begin{equation}\label{eq:pim-inductive}
2[\pi(0), \, \pi_{m}(c)]_{\rm Sch} + \sum_{k=1}^{m-1} [\pi_k(c), \, \pi_{m-k}(c)]_{\rm Sch} = 0 \in
c^{m} \mfX^3(\mX).
\end{equation}
Set
${\rm Obs}_m = \sum_{k=1}^{m-1} [\pi_k(c), \, \pi_{m-k}(c)]_{\rm Sch} \in
c^m \mfX^3(\mX)^{\calS_m}$, 
and note that 
$$
\left[\sum_{k=0}^{m-1} \pi_k(c),\,\;\sum_{k=0}^{m-1} \pi_k(c)\right]_{\rm Sch} = {\rm Obs}_m \text{ mod } c^{m+1}\mfX^3(\mX)\bbc.
$$
It is a standard fact that $\dpi({\rm Obs}_m)=0$. Indeed,
$$
\dpi({\rm Obs}_m) = \left[\pi_0,{\rm Obs}_m\right]=\left[\sum_{k=0}^{m-1} \pi_k(c), \left[\sum_{k=0}^{m-1} \pi_k(c),\,\;\sum_{k=0}^{m-1} \pi_k(c)\right]_{\rm Sch} \right]_{\rm Sch} ~~~\text{ mod } c^{m+1}\mfX^3(\mX)\bbc,
$$
and the right-hand-side expression is equal to zero by the graded Jacobi identity for the Schouten 
bracket\footnote{See  \cite[Page 369]{Laurent-Gengoux2013} for another proof that 
$\dpi({\rm Obs}_m)=0$ using the expression 
${\rm Obs}_m = \sum_{k=1}^{m-1} [\pi_k(c), \, \pi_{m-k}(c)]_{\rm Sch}$}. 
By \autoref{prop:indep-unobstr}, $\mH^3_{\pi_0}(\mX)^{\calS_{\geq 2}} = 0$. Thus there exists $\pi_m(c)\in c^m\mfX^2(\mathsf{X})^{\calS_m}$ such that 
\begin{equation}\label{eq:pi-m}
2[\pi_0,\,\pi_m(c)]_{\rm Sch} +{\rm Obs}_m = 0.
\end{equation}
 Setting
$\pi_m^\calS(c) = \pi_m(c)$ for $m \geq 2$, then 
$\pi^\calS(c) = \sum_{m=0}^{\infty} \pi_m^\calS(c)  \in {\rm Poi}(\mX)\bbc$.

To prove uniqueness, suppose that 
\[
\pi'(c) = \sum_{m=0}^\infty \pi'_m(c) \in {\rm Poi}(\mX)\bbc
\]
is such that $\pi'_0(c)=\pi_0$, 
$\pi'_1(c)=\pi_1(c)$, and $\pi'_m(c)\in c^m\mfX^2(\mathsf{X})^{\calS_m}$
for all $m \geq 2$. Let $m \geq 2$ be the smallest such that $\pi'_m(c)\not=\pi_m(c)$. Then
$[\pi_0, \, \pi_m(c)]_{\rm Sch} = -\frac{1}{2} {\rm Obs}_m = [\pi_0, \, \pi_m^\prime(c)]_{\rm Sch}$.
As $\mH^2_{\pi_0}(\mX)^{\calS_{\geq 2}} = 0$ by \autoref{prop:indep-unobstr}, there exists $v_m\in c^m\mfX^1(\mathsf{X})^{\calS_m}$ such that 
$-[v_m,\pi_0]_{\rm Sch} = \pi'_m(c) -\pi_m(c)$. Thus
\begin{align*}
\exp(v_m)(\pi'(c)) & = \exp( v_m)(\pi_0 + \pi'_1(c)+ \cdots +  \pi'_m(c)  \;\;\;{\rm mod}\;\;c^{m+1}\mfX^2(\mX)\\
& = \pi_0 +   [v_m, \pi_0] +\pi_1(c) + \cdots + \pi_{m-1} (c)+  \pi'_m(c) \;\;\;{\rm mod}\;\;c^{m+1}\mfX^2(\mX)\\
& = \pi_0 + \pi_1(c) + \cdots + \pi_{m-1}(c) +  {\pi}_m(c) \;\;\;{\rm mod}\;\;c^{m+1}\mfX^2(\mX).
\end{align*}
Continuing this process for all $m$, we obtain a gauge equivalence bringing $\pi'(c)$ to $\pi(c)$.
\end{proof}

\begin{definition}\label{def:admissible-formal}
For a linearly independent subset $\calS$ of $\calS(\pi_0)$, we call any $\pi^\calS(c) \in {\rm Poi}[\mX]\bbc$ in 
\autoref{thm:mainDef} an {\it $\calS$-admissible formal Poisson deformation of $\pi_0$}.
\end{definition}

When $\calS(\pi_0)$ is linearly independent, we now show that any $\calS(\pi_0)$-admissible 
formal Poisson deformation $\pi^{\calS(\pi_0)}(c)$ is  {\it versal}  in the following sense.

\begin{theorem}\label{thm:maximal-formal}
Let $\pi_0$ be a $\TT$-log-symplectic log-canonical Poisson structure  on $\mX = \CC^n$, and assume that $\calS(\pi_0)$ is linearly independent. Let $\mu$ be any set of formal commuting parameters, 
and let $\pi(\mu)$ be any formal Poisson deformation of
$\pi_0$ of the form
\[
\pi(\mu) = \pi_0 + \pi_1(\mu) +  \pi_2(\mu) + \pi_3(\mu) + \cdots \in {\rm Poi}(\mathsf{X})[\![\mu]\!],
\]
where for every $m \geq 1$,  $\pi_m(\mu)\in \mu^m \mfX^2(\mathsf{X})^\TT$ and contains no log-canonical terms.
Then up to a gauge equivalence of $\XX[\![\mu]\!]$, we have
$\pi(\mu)=\pi^{\calS(\pi_0)}(c(\mu))$ for some  
$c(\mu) = (c_\bth(\mu)\in \mu \mathbb{C}[\![\mu]\!])_{\bth \in \calS(\pi_0)}$.
\end{theorem}

\begin{proof}
For each $\bth \in \calS(\pi_0)$, the component of $\pi(\mu)$ with $(\CC^\times)^n$-weight $\boldsymbol{\theta}$ is of the form $c_\bth(\mu) V_{\boldsymbol{\theta}}$ for some $c_\bth(\mu)\in \mu\mathbb{C}[\![\mu]\!]$. Set 
$c(\mu) = (c_\bth(\mu))_{\bth\in\calS(\pi_0)}$.
Set
\[
\pi'(\mu) := \pi^{\calS(\pi_0)}(c(\mu)) \in {\rm Poi}(\mX)\bbmu.
\]
Let $\pi^{\prime\prime}(\mu) =\pi(\mu)-\pi'(\mu) \in \mfX(\mX)\bbmu$ and write
\[
\pi'(\mu) = \sum_{m=0}^\infty \pi'_m(\mu) \hs \mbox{and} \hs 
\pi^{\prime\prime}(\mu) = \sum_{m=1}^\infty \pi''_m(\mu),
\]
where $\pi'_0(\mu)=\pi_0$, and  $\pi'_m(\mu) \in \mu^m \mfX^2(\mathsf{X})^{ \calS(\pi_0)_{\ge1}}$ and 
$\pi''_m(\mu)\in\mu^m \mfX^2(\mathsf{X})^{\ker \bbeta\setminus \calS(\pi_0)}$
for $m \geq 1$. By the assumption on $\pi(\mu)$, 
we have $\pi''_m(\mu)\in\mu^m \mfX^2(\mX)^{\ker \bbeta  \setminus (\calS(\pi_0) \cup \{0)\}}$
for every $m \geq 1$.  
We now prove by induction on $m$ that we can apply a gauge equivalence to $\pi(\mu)$ so that new $\pi''_k(\mu)=\pi_k(\mu)-\pi'_k(\mu)$ satisfies $\pi''_k(\mu)=0$ for all $k\le m$. Indeed, 
since $\pi(\mu)$ and $\pi'(\mu)$ are Poisson, assuming $\pi''_k(\mu)=0$ for all $k<m$, we then have 
\begin{align*}
0& = \sum_{k=0}^m [\pi_k(\mu),\; \pi_{m-k}(\mu)]_{\rm Sch} = 
\sum_{k=1}^{m-1} [\pi_k(\mu),\; \pi_{m-k}(\mu)]_{\rm Sch} + 2[\pi_0, \pi_m(\mu)]_{\rm Sch} \\
& =\sum_{k=1}^{m-1} [\pi_k^\prime(\mu),\; \pi_{m-k}^\prime(\mu)]_{\rm Sch} + 
2[\pi_0, \pi_m^\prime(\mu)]_{\rm Sch}+2[\pi_0, \pi_m^{\prime\prime}(\mu)]_{\rm Sch}\\
&= \sum_{k=0}^{m} [\pi'_k(\mu),\;\pi'_{m-k}(\mu)]_{\rm Sch} + 
2 [\pi_0,\;\pi''_m(\mu)]_{\rm Sch}= 2 [\pi_0,\;\pi''_m(\mu)]_{\rm Sch}.
\end{align*}
By \autoref{lm:HP2decomp}, the term $\pi''_m(\mu)$ is $\mathsf{d}_{\pi_0}$-exact. After applying the gauge equivalence 
$\exp(v_m)$ to $\pi(\mu)$, where $\mathsf{d}_{\pi_0}(v_m)=-\pi''_m(\mu)$, we obtain $\pi''_m(\mu)=0$.
\end{proof}

\section{Algebraic Poisson deformations}\label{s:alg-deform}
Let again $\TT$ be any complex algebraic torus.
We now turn to $\TT$-invariant algebraic Poisson deformations of 
any $\TT$-log-symplectic log-canonical Poisson structure $\pi_0$ on $\mX=\CC^n$.  
More precisely, after some preparation in $\S$\ref{ss:WWW}, 
we give in $\S$\ref{ss:existence-alg} a sufficient condition, called Property \eqref{eq:W0}, 
for the formal Poisson deformations of $\pi_0$ in \autoref{thm:mainDef} to be algebraic, and we prove the analog of \autoref{thm:maximal-formal} on the uniqueness of such algebraic deformations up to
algebraic changes of variables. In $\S$\ref{ss:W1} we show uniqueness
(without algebraic changes of variables) of 
$\TT$-invariant algebraic Poisson deformations of $\pi_0$ under the stronger  Property \eqref{eq:W1}. 

\subsection{Notation and some preparatory lemmas}\label{ss:WWW}
Let again $\mX = \CC^n$ with linear coordinates $(x_1, \ldots, x_n)$ and the 
$\CCsn$-action given in \eqref{eq:CCsn-CCn}. 
Recall that ${\rm Poi}(\mX)$ denotes the set of all algebraic Poisson structures on $\mX = \CC^n$.  

\begin{notation}\label{nota:pi-pi0}
{\rm
Let $\pi \in {\rm Poi}(\mX)$ be arbitrary. Each monomial term 
(see \autoref{lm:X-w}) in $\pi$ has a $\CCsn$-weight $\bfw \in \ZZone$.
We write $\pi$ uniquely as  
\begin{equation}\label{eq:pi-tail}
\pi=\pi_0 + \pi_{\rm tail},
\end{equation}
where $\pi_0 \in \mfX^2(\mX)$ is the sum of all the monomial terms in $\pi$ with $\CCsn$-weight $0$, and 
$\pi_{\rm tail} \in \mfX^2(\mX)$ is the sum of all the monomial terms in $\pi$ 
with non-zero $\CCsn$-weights. 
We set
\begin{align}\label{eq:mW-pi}
 \mW(\pi) &= \mbox{the set of all the non-zero} \;\CCsn\mbox{-weights of the monomial terms in}\; \pi\\
 \nonumber
& = \mbox{the set of all the} \;\CCsn\mbox{-weights of the monomial terms in}\; \pi_{\rm tail}.
\end{align}
We call $\pi_0$ the log-canonical term of $\pi$.
\hfill $\diamond$
}
\end{notation}

Let ${\rm Aut}(\mX)$ be the group of all bi-regular isomorphisms from $\mX = \CC^n$ to itself.

\begin{definition}\label{def:varphi}
{\rm
For $\varphi \in {\rm Aut}(\mX)$
and $\pi \in {\rm Poi}(\mX)$, let $\varphi_* \pi \in {\rm Poi}(\mX)$ be such that
$\varphi: (\mX,  \pi) \rightarrow (\mX, \, \varphi_*\pi)$
is an isomorphism of Poisson varieties, so that
\begin{equation}\label{eq:Aut-action}
{\rm Aut}(\mX) \times {\rm Poi}(\mX) \longrightarrow {\rm Poi}(\mX), \;\; (\varphi, \pi) \longmapsto \varphi_* \pi,
\end{equation}
is a left action of ${\rm Aut}(\mX)$ on ${\rm Poi}(\mX)$. 
If $\pi, \pi' \in {\rm Poi}(\mX)$ are such that
$\pi' = \varphi_* \pi$ for some $\varphi \in {\rm Aut}(\mX)$, we also say that $\pi$ and $\pi'$ are related by an algebraic change of variables of $\mX$. 
}
\end{definition}

We will be concerned with some special algebraic changes of variables on $\mX = \CC^n$.

Recall (\autoref{lm:X-w}) that 
$\calW_p$, for $p \in [0, n]$, denotes the set of all $(\CC^\times)^n$-weights in $\mfX^p(\mX)$, i.e.,
$\calW_p = \{\bfw \in (\ZZ_{\geq -1})^n: |\Jw| \leq p\}$.
Note that
\[
(\ZZ_{\geq 0})^n = \calW_0 \subset \calW_1 \subset \calW_2 \subset \cdots \subset \calW_n.
\]
In particular, $\bfw \in \calW_1 \backslash \calW_0$ if and only if $\bfw \in \ZZone$ and has
exactly one $-1$ entry.

\begin{lemma}\label{lm:t-flow}
For every $\bfw =(\bfw_1, \ldots, \bfw_n)^t\in \calW_1\backslash \calW_0$ the space
$\mfX^1(\mX)^\bfw$ is $1$-dimensional, and if $k \in [1, n]$ is such that $\bfw_k = -1$, then  a basis vector of $\mfX^1(\mX)^\bfw$ is given  by the vector field
\begin{equation}\label{eq:v-w}
v_\bfw =\left(\prod_{j \neq k} x_j^{\bfw_j}\right) \frac{\partial}{\partial x_k}.
\end{equation}
For $t \in \CC$, the time $t$-flow $\varphi_\bfw^t: \mX \to \mX$ of $v_\bfw$ is 
 given by
\begin{equation}\label{eq:t-flow}
(\varphi_\bfw^t)^*(x_1, \ldots, x_n) = 
(x_1, \;\ldots,\; x_{k-1}, \;x_k+t\prod_{j \neq k} x_j^{\bfw_j}, \;
x_{k+1}, \; \ldots, \;x_n). 
\end{equation}
Moreover, for any $\pi \in {\rm Poi}(\mX)$ and any $t \in \CC$, one has
$(\varphi^{t}_\bfw)_* \pi \in {\rm Poi}(\mX)$ given by
\[
(\varphi^{t}_\bfw)_* \pi = \exp(-tv_\bfw) \pi =\pi - t[v_\bfw, \pi] +
\frac{1}{2} t^2[v_\bfw, [v_\bfw, \pi]] + \cdots \in {\rm Poi}(\mX).
\]
\end{lemma}

\begin{proof}
The fact that $\mfX^1(\mX)^\bfw = \CC v_\bfw$ is a special case of \autoref{lm:X-w}. The proof of the remaining statements is standard.
\end{proof}

For a finite subset $\mW$ of $\ZZ^n$ recall from 
\autoref{nota:Wm} the definitions of $\mW_m \subset \mW_{\geq m} \subset \ZZ^n$ for $m \in \ZZ_{\geq 1}$. 

\begin{lemma}\label{lm:expv}
Let $\pi \in {\rm Poi}(\mX)$ with log-canonical term  $\pi_0$ and suppose that $0 \notin 
\mW(\pi)_{\geq 1}$. Then for any
\[
\bfw \in (\mW(\pi)_{\geq 1} \cap \calW_1) \backslash \calW_0
\]
and  any $t \in \CC$,  the algebraic Poisson structure $\pi' := (\varphi^{-t}_\bfw)_* \pi=\exp(tv_\bfw)\pi$ 
also has log-canonical term $\pi_0$, 
and one has  $\mW(\pi') \subset \mW(\pi)_{\geq 1}$.
\end{lemma}

\begin{proof} 
Let $\pi = \pi_0 + \pi_{\rm tail}$, where $\pi_{\rm tail} \in \mfX^2(\mX)^{\mW(\pi)}$. Then
\[
\pi' = \exp(tv_\bfw) \pi =\pi_0 + \pi_{\rm tail} + t[v_\bfw, \pi_0] + t[v_\bfw, \pi_{\rm tail}] + 
\frac{1}{2}t^2[v_\bfw, [v_\bfw, \pi_0]] + \frac{1}{2}t^2[v_\bfw, [v_\bfw, \pi_{\rm tail}]]+\cdots.
\]
It follows from the assumptions that $\bfw \in \mW(\pi)_{\geq 1}$ and $0 \notin \mW(\pi)_{\geq 1}$
that  $\pi'$ again has $\pi_0$ as its log-canonical term, 
 and $\mW(\pi') \subset \mW(\pi)_{\geq 1}$.
\end{proof}

In the rest of this subsection, we prepare several combinatorial lemmas on subsets of $\ZZ^n$.

\begin{lemma}\label{lm:generator-W0}
Let $\mW$ be  any finite subset of $\ZZ^n$ such that
$0\notin \mW_{\geq 1}$. Then 
$\mW \backslash \mW_{\geq 2}$
is the one and only  minimal set of 
generators of $\mW_{\geq 1}$ under addition. In particular 
\begin{equation}\label{eq:WWW}
\mW_{\geq 1} = 
(\mW \backslash \mW_{\geq 2})_{\geq 1}.
\end{equation}
\end{lemma}

\begin{proof}
As $\mW$ is finite, minimal sets of generators of $\mW_{\geq 1}$ under addition exist. 
Let $\calS$ be any minimal set of generators of $\mW_{\geq 1}$ under addition. 
We first prove that $\calS \subset \mW \backslash \mW_{\geq 2}$.
Let $\bfw \in \calS$. Suppose that $\bfw \in \mW_{\geq 2}$.
Then $\bfw = \bfw^{(1)} + \cdots + \bfw^{(m)}$, where $m \geq 2$ and 
$\bfw^{(j)} \in \mW$ for each $j \in [1, m]$. By  writing
$\bfw^{(j)} = \bfw^{(j, 1)} + \cdots +\bfw^{(j, n_j)}$, where $\bfw^{(j, k)} \in \calS$ for each 
$j \in [1, m]$ and $k \in [1, n_j]$, we get
\[
\bfw= \bfw^{(1, 1)} + \cdots +\bfw^{(1, n_1)} + \cdots + \bfw^{(m, 1)} + \cdots +\bfw^{(m, n_m)},
\]
where we note that $\sum_{j=1}^m n_j \geq m \geq 2$. As $0 \notin \mW_{\geq 1}$, 
we must have $\bfw \notin \{\bfw^{(j, k)}:j \in [1, m], k \in [1, n_j]\}$, 
so $\calS\backslash \{\bfw\}$ generates $\mW_{\geq 1}$, contradicting the minimality of $\calS$ as a set of generators of $\mW_{\geq 1}$. Thus $\bfw \notin
\mW_{\geq 2}$. As $\bfw \in \calS \subset \mW_{\geq 1}$, we have $\bfw \in \mW$.
Thus $\calS \subset \mW \backslash \mW_{\geq 2}$.
On the other hand, for any $\bfw \in \mW \backslash \mW_{\geq 2}$, by writing $\bfw$ as a sum of elements in 
$\calS$ and using the fact that $\bfw \notin \mW_{\geq 2}$, we see that $\bfw \in \calS$. 
Thus $\calS= \mW\backslash \mW_{\geq 2}$.
The rest of \autoref{lm:generator-W0} follows from the definitions.
\end{proof}

\begin{remark}\label{rmk:no-W0}
{\rm
The statement of \autoref{lm:generator-W0} on the uniqueness of minimal generating sets of $\mW$
may not hold without 
the assumption that $0 \notin \mW_{\geq 1}$. Indeed, consider $\mW =\{\bfw^{(1)}, \bfw^{(2)}, \bfw^{(3)}\}
\subset \ZZ^3$, where
\[
\bfw^{(1)} =(-1, -1, 2)^t, \hs \bfw^{(2)} =(-1, 2, -1)^t, \hs
\bfw^{(3)}=(2, -1, -1)^t.
\]
We have $\bfw^{(1)} + \bfw^{(2)}+\bfw^{(2)} = 0$, and  both $\mW$ and $-\mW = \{-\bfw^{(1)}, -\bfw^{(2)}, -\bfw^{(3)}\}$ are minimal sets of
generators of $\mW_{\geq 1}$ under addition.
\hfill $\diamond$
}
\end{remark}

\begin{definition}\label{def:W0}
A finite subset $\mW$ of $\ZZone$ is said to have Property \eqref{eq:W0} if 
\begin{equation}\label{eq:W0}
\mW_{\geq 1} \cap (\ZZ_{\geq 0})^n = \emptyset, \hs \mbox{equivalently}\hs \mfX^0(\mX)^{\mW_{\geq 1}} = 0\tag{W0},
\end{equation}
which is further equivalent to every element
of $\mW_{\geq 1} \cap \ZZone$ having at least one $-1$ entry.
\end{definition}
 
To give other equivalent characterizations of Property \eqref{eq:W0} on $\mW \subset \ZZ^n$, we first prove a lemma.

\begin{lemma}\label{lm:increasingSubsequenceNn} 
    Let $\ell \geq 1$ and 
     ${\bf v}^{(1)}$, ${\bf v}^{(2)}$, ...  be any sequence in 
    $(\mathbb{Z}_{\ge -1})^\ell$. Replacing by a subsequence if necessary,  we can ensure that  ${\bf v}^{(m)} \le {\bf v}^{(m+1)}$ coordinate-wise for each $m\ge1$.
\end{lemma}

\begin{proof}
    Since the set $\mathbb{Z}_{\ge -1}$ is well-ordered, by passing to a subsequence, we can assume that the first coordinate of ${\bf v}^{(m)}$ is less than or equal to the first coordinate of ${\bf v}^{(m+1)}$ for every $m\ge1$. Applying the same argument to the second, third coordinate and so on, we obtain the claim.
\end{proof}

\begin{lemma}\label{lm:three-conditions}
For any finite subset $\mW$ of $\ZZone$, the following statements are equivalent:

1) $\mW$ has Property \eqref{eq:W0};

2)  $\mW_{\geq N+1} \cap \calW_2 =\emptyset$ for some integer $N \geq 1$;

3) $0 \notin \mW_{\geq 1}$ and $\mW_{\geq 1} \cap \calW_2$ is a finite set; 
\end{lemma}

\begin{proof}
Assuming 1), we prove the stronger claim that 
$\mW_{\geq N+1} \cap (\ZZ_{\geq -1})^n =\emptyset$ for some $N \geq 1$.
Suppose on the contrary that  there is an infinite subset $\calM$ of $\ZZ_{\geq 1}$ such that $\mW_{m}\cap(\ZZ_{\geq -1})^n\neq \emptyset$
for every $m \in \calM$. Choose
$\bfw^{(m)} \in \mW_{m}\cap(\ZZ_{\geq -1})^n$ for each $m \in \calM$ and write
\[
\bfw^{(m)} = \sum_{\bfw \in \mW}\alpha_{\bfw}^{(m)} \bfw,
\]
where $\alpha_{\bfw}^{(m)} \in \ZZ_{\geq 0}$ 
for every $m \in \calM$ and $\bfw \in \mW$, and
$\sum_{\bfw \in \mW} \alpha_{\bfw}^{(m)} = m$ for every $m \in \calM$. Applying
\autoref{lm:increasingSubsequenceNn} to $\{(\alpha^{(m)}_{\bfw})_{\bfw \in \mW}: m \in \calM\}$ and 
replacing $\calM$ by an infinite subset if necessary, we may assume that $\alpha_{\bfw}^{(m)}\leq \alpha_{\bfw}^{(m')}$
for all $\bfw \in \mW$ and all $m, m' \in \calM$ such that $m \leq m'$. Applying 
\autoref{lm:increasingSubsequenceNn} to $\{\bfw^{(m)}: m \in \calM\}\subset (\ZZ_{\geq -1})^n$ and replacing $\calM$ by an infinite subset if necessary,
we may assume that $\bfw^{(m)} \leq \bfw^{(m')}$ coordinate-wise for all 
$m, m' \in \calM$ such that $m \leq m'$.
Choose any  $m, m' \in \calM$ such that $m' \geq m+1$. Then, as  
$\sum_{\bfw \in \mW} (\alpha^{(m')}_\bfw-\alpha^{(m)}_\bfw) = m' - m \geq 1$, we have
\[
{\bf w}^{(m')} - {\bf w}^{(m)} = 
\sum_{\bfw \in \mW}(\alpha^{(m')}_\bfw-\alpha^{(m)}_\bfw) \bfw\in \mW_{\ge 1}\cap (\ZZ_{\geq 0})^n,
\]
contradicting the assumption that $\mW_{\ge 1}\cap (\ZZ_{\geq 0})^n =\emptyset$. Thus 1) implies 2).

Assume 2). As $\mW$ is finite, the set $\mW_m$ is finite for every $m \geq 1$, so 2) implies that
\[
\mW_{\geq 1} \cap \calW_2 = \bigcup_{m=1}^{N} \mW_m \cap \calW_2
\]
is finite. If $0 \in \mW_{\geq 1}$, then $0 \in \mW_m$ for some $m \geq 1$, so $0 = p\cdot 0 
\in \mW_{pm} \cap \calW_2$ for every $p \geq 1$, contradicting 2). Thus 2) implies 3).

Assume 3).
If $\mW$ does not have Property \eqref{eq:W0}, then there exist
$\bfw \in \mW_{\geq 1} \cap (\ZZ_{\geq 0})^n$ and $\bfw \neq 0$. Then $\{m\bfw: m \in \ZZ_{\geq 1}\}$
is an infinite subset of $\mW_{\geq 1} \cap \calW_2$, contradicting 3). Thus 3) implies 1).
\end{proof}

\begin{remark}\label{rk:W0-more}
{\rm
It follows from the proof of \autoref{lm:three-conditions} that the set $\calW_2$ in
statement 2) and 3) in \autoref{lm:three-conditions} can be replaced by any set $\calW$ satisfying
$(\ZZ_{\geq 0})^n \subset \calW \subset \ZZone$.
\hfill $\diamond$
}
\end{remark}

\subsection{Algebraic Poisson deformations with Property \eqref{eq:W0}}\label{ss:existence-alg} 
Let $\pi_0$ be a
$\TT$-log-symplectic log-canonical Poisson structure on $\mX=\CC^n$ 
as in $\S$\ref{ss:T-cohom-log}. 
Recall that 
$\calS(\pi_0)$ is the (finite) set of all non-zero weights in $\mH_{\pi_0}^2(\mX)^\TT$.
By \autoref{ld:indep}, a non-empty $\calS \subset \calS(\pi_0)$ is linearly independent if and only if $0 \notin \calS_{\geq 1}$. 
Note also that when $\calS \subset \calS(\pi_0)$ is linearly independent, we have $\calS_m\cap \calS_{m'}
=\emptyset$ for $m, m' \geq 1$ and  $m \neq m'$, so we have the direct sum decomposition
\[
\mfX^2(\mX)^{\calS_{\geq 1}} = \bigoplus_{m=1}^{\infty} \mfX^2(\mX)^{\calS_{m}}.
\]
By convention when $\calS = \emptyset$ we have $\calS_{\geq 1}=\emptyset$ and $\mfX^2(\mX)^{\calS_{\geq 1}}=0$.
Recall again that ${\rm Poi}(\mX)\subset \mfX^2(\mX)$ denotes the set of all algebraic Poisson structures on $\mX = \CC^n$.

\begin{definition}\label{def:calS-admi} 
Assume that $\calS \subset \calS(\pi_0)$ is linearly independent. 

1) By an {\it $\calS$-admissible algebraic Poisson deformation of $\pi_0$} we mean any 
$\pi \in {\rm Poi}(\mX)^{\calS_{\geq 0}}$  such that 
\[
\pi = \pi_0 \;\;\;{\rm mod}\;\;\mfX^2(\mX)^{\calS_{\geq 1}};
\]

2) For a given $\pi_1 \in \mfX^2(\mX)^\calS$, if  $\pi \in {\rm Poi}(\mX)^{\calS_{\geq 0}}$ is such that 
\[
\pi = \pi_0 + \pi_1 \;\;\;{\rm mod}\;\;\mfX^2(\mX)^{\calS_{\geq 2}},
\]
we say that $\pi$ is 
an {\it $\calS$-admissible algebraic Poisson deformation of $\pi_0$ along $\pi_1$}. 

3) An $\calS$-admissible algebraic Poisson deformation of $\pi_0$ along some 
$\pi_1 =\sum_{\bth \in \calS} c_\bth V_\bth$ such that $c_\bth \in \CC^\times$ for every $\bth \in \calS$ 
is also called a {\it maximal} $\calS$-admissible algebraic Poisson deformation of $\pi_0$.
\end{definition}

For an arbitrary subset $\calS$ of $\calS(\pi_0)$, we now show that Property \eqref{eq:W0} of 
$\calS$  is a sufficient condition for the existence of 
$\calS$-admissible algebraic Poisson deformations of $\pi_0$. 
By \autoref{lm:three-conditions} the following statements are equivalent for $\calS \subset \calS(\pi_0)$:

1) $\calS$ has Property \eqref{eq:W0}, i.e., $\calS_{\geq 1} \cap (\ZZ_{\geq 0})^n = \emptyset$;

2) $\calS_{\geq N+1} \cap \calW_2 = \emptyset$ for some $N \geq 1$;

3) $\calS$ is linearly independent and  $\calS_{\geq 1} \cap \calW_2$ is finite.

\begin{remark}\label{rk:pi-pim}
{\rm 
Assume that $\calS \subset \calS(\pi_0)$ has Property \eqref{eq:W0}, and let 
$N \geq 1$ be such that 
$\calS_{\geq N+1} \cap \calW_2 = \emptyset$. Let
$\pi_1 \in \mfX^2(\mX)^\calS$. Then an $\calS$-admissible algebraic Poisson deformation of $\pi_0$ along $\pi_1$ is the same as a bi-vector field 
$\pi = \pi_0 + \pi_1 + \pi_2 + \cdots+ \pi_N$, where
$\pi_m \in \mfX^2(\mX)^{\calS_m}$ for $m \in [2, N]$, such that 
\begin{equation}\label{eq:pi-pim-1}
2[\pi_0, \, \pi_m]_{\rm Sch} + \sum_{k=1}^{m-1} [\pi_k,\pi_{m-k}]_{\rm Sch} =0 \in \mfX^3(\mathsf{X})^{\calS_m}, \hs 
m \in [2, 2N].
\end{equation}
\hfill $\diamond$
} 
\end{remark}

For $\bfw \in \calW_1\backslash \calW_0$, recall the vector field $v_\bfw \in \mfX^1(\mX)^\bfw$
in \eqref{eq:v-w}. 

\begin{lemma-notation}\label{lm-nota:calG-S}
If $\calS \subset \calS(\pi_0)$ has Property \eqref{eq:W0}, then the vector subspace
\[
{\mathfrak{\g}}_{{}_{\calS_{\geq 2} \cap \calW_1}} := \sum_{\bfw \in \calS_{\geq 2} \cap \calW_1} \CC v_\bfw
\]
of $\mfX^1(\mX)$ is a finite dimensional nilpotent Lie algebra with respect to $[-\, , \, -]_{\rm Sch}$
on $\mfX^1(\mX)$. We denote by
$\calG_{\calS_{\geq 2} \cap \calW_1}$ the unipotent Lie group with Lie algebra 
${\mathfrak{\g}}_{{}_{\calS_{\geq 2} \cap \calW_1}}$, and let 
$\calG_{\calS_{\geq 2} \cap \calW_1}$ act on $\mX$ via the induced group homomorphism
$\calG_{\calS_{\geq 2} \cap \calW_1} \to {\rm Aut}(\mX)$.
\end{lemma-notation}

\begin{proof}
As $\calS$ has Property \eqref{eq:W0}, the set $\calS_{\geq 2} \cap \calW_1$, being a subset of
$\calS_{\geq 1} \cap \calW_2$, is finite, so ${\mathfrak{\g}}_{{}_{\calS_{\geq 2} \cap \calW_1}}$ is 
finite dimensional. Suppose that  $\bfw^{(1)}, \ldots, \bfw^{(m)} \in 
\calS_{\geq 2} \cap \calW_1$ are such that 
\[
v :=[v_{\bfw^{(m)}}, \cdots, [v_{\bfw^{(2)}}, v_{\bfw^{(1)}}]_{\rm Sch}]_{\rm Sch} \neq 0.
\]
Then every monomial term of $v$ has $\CCsn$-weight in 
$\calS_{\geq m} \cap \calW_1\subset \calS_{\geq m} \cap \calW_2$. As 
$\calS_{\geq m} \cap \calW_2 = \emptyset$ for $m$ large enough, 
$({\mathfrak{\g}}_{{}_{\calS_{\geq 2} \cap \calW_1}}, [-\, , \, -]_{\rm Sch})$ is a nilpotent Lie algebra.
\end{proof}

\begin{theorem}\label{thm:4.8-W0} Let $\pi_0$ be any
$\TT$-log-symplectic log-canonical Poisson structure on $\mX=\CC^n$. If $\calS \subset \calS(\pi_0)$
has  Property \eqref{eq:W0}, then for any 
$\pi_1 \in \mfX^2(\mX)^{\calS}$, the set of all 
$\calS$-admissible algebraic Poisson deformations of $\pi_0$ along $\pi_1$ is non-empty and is a single 
$\calG_{\calS_{\geq 2} \cap \calW_1}$-orbit in ${\rm Poi}(\mX)$. 
\end{theorem}

\begin{proof} As Property \eqref{eq:W0} of $\calS$ implies that $\calS$ is linearly independent,
by \autoref{thm:mainDef} we have an $\calS$-admissible 
formal Poisson structure  
$\pi^\calS(c) \in {\rm Poi}(\mX)\bbc$ in the formal parameters $c = (c_\btheta)_{\btheta \in \calS}$. 
As $\calS_{\geq N+1} \cap \calW_2 = \emptyset$ for some $N \geq 1$, 
$\pi^\calS(c) \in {\rm Poi}(\mX)\bbc$ is polynomial in $c$. 
Write the given $\pi_1 \in \mfX^2(\mX)^\calS$ as 
$\pi_1 = \sum_{\bth \in \calS} c^\prime_\bth V_{\bth}\in 
\mfX^2(\mX)^{\calS}$, where $c^\prime_\btheta \in \CC$ for each $\btheta \in \calS$.
By evaluating $\pi^\calS(c)$
at $c = (c^\prime_\bth\in \CC)_{\bth \in \calS}$, we get an $\calS$-admissible algebraic Poisson
deformation of $\pi_0$ along $\pi_1$. 

Let $\pi$ be any
$\calS$-admissible algebraic Poisson deformation of $\pi_0$ along $\pi_1$, and write
 $\pi = \pi_0 + \pi_1 + \pi_{\!{}_{\geq 2}}$, where  
$\pi_{\!{}_{\geq 2}} \in \mfX^2(\mX)^{\calS_{\geq 2}}$. 
Let $\bfw \in \calS_{\geq 2} \cap \calW_1 \subset \calW_1 \backslash \calW_0$, and let
$v_\bfw \in \mfX^1(\mX)^\bfw$ as in \eqref{eq:v-w}; let
$\{\varphi_\bfw^t: t \in \CC\}$ be its flow.  For 
 $t \in \CC$, by \autoref{lm:t-flow} one has
\begin{align*}
(\varphi_\bfw^{t})_*\pi& = \exp(-tv_\bfw) \pi =\exp(-tv_\bfw) \pi_0 +  \exp(-tv_\bfw)\pi_1 + \exp(-tv_\bfw)\pi_{\!{}_{\geq 2}}\\
&= \pi_0  +\pi_1 -  t[v_\bfw, \pi_0]\;\;\; \mbox{mod}\;\; \mfX^2(\mX)^{\calS_{\geq 2}}.
\end{align*}
As $\bfw \in  \calS_{\geq 2}$, we have
$[v_\bfw, \pi_0] \in \mfX^2(\mX)^\bfw \subset 
\mfX^2(\mX)^{\calS_{\geq 2}}$. Thus $(\varphi_\bfw^{t})_*\pi$ is again 
an $\calS$-admissible
algebraic Poisson deformation of $\pi_0$ along $\pi_1$. This shows that all elements in the orbit
$\calG_{\calS_{\geq 2} \cap \calW_1}\cdot\pi$ are  $\calS$-admissible
algebraic Poisson deformation of $\pi_0$ along $\pi_1$.

Conversely, let $N \geq 1$ be such that $\mfX^2(\mX)^{\calS_{\geq N+1}} = 0$, and let 
\[
\pi = \pi_0 + \pi_1 + \pi_2 + \cdots + \pi_{N} \hs \mbox{and} \hs 
\pi^\prime = \pi_0 + \pi_1 + \pi_2^\prime + \cdots + \pi^\prime_{N}
\]
be two $\calS$-admissible algebraic Poisson deformations of $\pi_0$ along $\pi_1$, with
$\pi_m, \pi_m^\prime  \in \mfX^2(\mX)^{\calS_m}$ for all $m \in [2, N]$.  If $N = 1$ we have $\pi = \pi'$, so we 
assume that  $N \geq 2$.
Suppose that $\pi \neq \pi'$, and let $m \in [2, N]$ be the smallest index such that $\pi'_m\not=\pi_m$. 
Applying \eqref{eq:pi-pim-1} to both $\pi$ and $\pi'$,  one has
$\mathsf{d}_{\pi_0}(\pi_m)  = \mathsf{d}_{\pi_0}(\pi'_m)$.
As $\calS$ is linearly independent, $\mH^2_{\pi_0}(\mX)^{\calS_{\geq 2}} = 0$
by \autoref{prop:indep-unobstr}.
Thus there exists $v_m\in \mfX^1(\mathsf{X})$ such that 
$\dpi(v_m)  = \pi'_m -\pi_m$. As the $\CCsn$-weight of every monomial term in $\pi'_m - \pi_m$ lie
in $\calS_m \cap \calW_1$, we have $v_m =\sum_{{\bf w}\in \calS_m \cap \calW_1}  t_\bfw v_{\bf w}$, 
where $t_\bfw \in \CC$ and  $v_\bfw$ is given in \eqref{eq:v-w}. 
For each $\bfw \in \calS_m \cap \calW_1$, 
$$
\exp(t_\bfw v_{\bf w}) (\pi') = \pi' + [t_\bfw v_{\bf w},\pi'] \mod \mfX^2(\mathsf{X})^{\calS_{\geq m+1}} 
= \pi' + [t_\bfw v_{\bf w},\pi_0] \mod \mfX^2(\mathsf{X})^{\calS_{\geq m+1}}.
$$
By applying $\prod_{{\bf w}\in\calS_m \cap \calW_1}\exp(t_\bfw v_{\bf w})$ 
(in any order for the product), we obtain
$$
\pi^{\prime\prime}:=\prod_{{\bf w}\in\calS_m \cap \calW_1}\exp(t_\bfw v_{\bf w}) (\pi') = \pi' + [v_m, \pi_0] 
\mod \mfX^2(\mathsf{X})^{\calS_{\geq m+1}} = \pi \mod \mfX^2(\mathsf{X})^{\calS_{\geq m+1}}.
$$
If $\pi^{\prime \prime} = \pi$ we are done.
If not, we continue the process which will stop as $N$ is finite, and we 
obtain an element $\varphi \in \calG_{\calS_{\geq 2} \cap \calW_1}$ such that $\varphi_* \pi' = \pi$.
\end{proof}

\begin{remark}[Method of undetermined coefficients]\label{rmk:in-practice} 
{\rm 
Suppose that $\calS \subset \calS(\pi_0)$ is non-empty and has Property \eqref{eq:W0}. 
Then $\calS_{\geq 2} \cap \calW_2$ is a finite set, and every $\bfw \in \calS_{\geq 2}\cap 
\calW_2$ has either one or two $-1$ entries.  By \autoref{lm:X-w}, 
for every $\bfw \in \calS_{\geq 2}\cap 
\calW_2$, we have 
\[
\dim \mfX^2(\mX)^\bfw = \begin{cases} n-1, & \hs \mbox{if} \;\; |J_\bfw| = 1,\\
1, & \hs \mbox{if} \;\; |J_\bfw| = 2,\end{cases}
\]
and $\mfX^2(\mX)^\bfw$ has a $\CC$-basis consisting of monomial bi-vector fields of the form
\begin{equation}\label{eq:basis}
x^\bfw \partial_j \wedge \partial_k, \hs \mbox{where}\;\; 1 \leq j < k \leq n \;\; \mbox{and}\;\;
J_\bfw \subset \{j, k\}.
\end{equation}
Let $N \geq 1$ be the smallest integer such that 
$\calS_{m} \cap \calW_2= \emptyset$ for all $m \geq N+1$.
Given $\pi_1=\sum_{\bth \in \calS} c_\bth V_\bth \in \mfX^2(\mX)^\calS$, where
$c_\bth \in \CC$ for each $\bth \in \CC$, we can solve for all $\calS$-admissible algebraic 
Poisson deformations $\pi$ of $\pi_0$ along $\pi_1$ by setting 
$\pi = \pi_0 + \pi_1 + \pi_2 + \cdots + \pi_N$, expressing  $\pi_m \in \mfX^2(\mX)^{\calS_m}$ 
for each $m \in [2, N]$ as a linear combination of the bi-vector fields in \eqref{eq:basis},
 and solving for 
the coefficients in $\pi_m$ recursively from \eqref{eq:pi-pim-1}
for all $m \in [2, N]$. This {\it method of undetermined coefficients}, when applied either by hand or by 
a computer program, can be facilitated by
the formulas in \autoref{lm:wvJK} and \autoref{lm:dpi-0}.  
\hfill $\diamond$
}
\end{remark}

The following example demonstrates that the $\calG_{\calS_{\geq 2} \cap \calW_1}$-orbit of $\calS$-admissible Poisson deformations of $\pi_0$ can indeed contain more than one element.

\begin{example}\label{ex:11}
{\rm
Let $\TT = (\CC^\times)^2$ and write elements in $X^*(\TT)$ as integral column vectors.
Let $\TT$ act  on $\CC^4$ via $\bbeta = (\beta_1, \beta_2, \beta_3, \beta_4)$, where
$\beta_1 = (-1,0)^t$, $\beta_2=(0,1)^t$, $\beta_3=\beta_4 =(-1,-1)^t$, and let 
\[
\boldsymbol{\lambda} = 
\begin{pmatrix}
    0 & 1 & -1 & 0 \\
    -1&0&0&1\\
    1&0&0&-1\\
    0&-1&1&0.
\end{pmatrix}
\]
One checks directly (but see also \autoref{ex:111})  that $\pi_0$ is $\TT$-log-symplectic, and that, in the
notation in \autoref{de:Delta-n}, one has $\calE^+(\pi_0) =\{(2, 3), (1, 4)\}$ and
$\calS(\pi_0) = \{\bth^{(2,3)}, \bth^{(1,4)}\}$, where
\[
\bth^{(2,3)} = (1,-1,-1,0)^t \hs \mbox{and} \hs \bth^{(1,4)} = (-1,1,2,-1)^t.
\]
The set $\calS(\pi_0)$ is encoded by the smoothing diagram in \autoref{fig:SmDiag11} below.

\begin{figure}[h]
    \centering
\pgfdeclarelayer{background layer}
\pgfdeclarelayer{foreground layer}
\pgfsetlayers{background layer,main,foreground layer}
\begin{tikzpicture}[baseline=-1ex,
rotate=360/2/4,
scale=0.9,line join = round 
    ] 
    
    \begin{pgfonlayer}{main}
\foreach \n in {1,...,4}
{
    \coordinate (v\n) at ({-90+((\n-1)*360/4)}:1.5);
}
\node[below right] at (v1) {$1$};
\node[above right] at (v2) {$2$};
\node[above left] at (v3) {$3$};
\node[below left] at (v4) {$4$};

\foreach \n in {1,...,4}
{
	\foreach \m in {1,...,4}
	\draw (v\n) -- (v\m);
}
\end{pgfonlayer}

\draw[thick,red,-] (v1) ++({0.4*cos(45)},{0.4*sin(45)})  arc[start angle=45,end angle=90,radius=0.4];
\draw[thick,red,-] (v2) ++({0.4*cos(180)},{0.4*sin(180)})  arc[start angle=180,end angle=225,radius=0.4];
\draw[thick,red,-] (v3) ++({0.4*cos(225)},{0.4*sin(225)})  arc[start angle=225,end angle=270,radius=0.4];
\draw[thick,red,-] (v3) ++({0.55*cos(225)},{0.55*sin(225)})  arc[start angle=225,end angle=270,radius=0.55];

\draw[very thick, blue] (v1)--(v4);
\draw[very thick, blue] (v2)--(v3);

\begin{pgfonlayer}{foreground layer}
\foreach \n in {1,...,4}
{
	\draw[fill] (v\n) circle (0.03);
}
\end{pgfonlayer}
 \end{tikzpicture}
    \caption{Smoothing diagram appearing in \autoref{ex:11}}
    \label{fig:SmDiag11}
\end{figure}

One checks directly that 
\[
\calS(\pi_0)_2\cap \calW_2  = \{\bth^{(2,3)}+ \bth^{(1,4)}= (0, 0, 1, -1)^t\}, \hs
\calS(\pi_0)_3 \cap \calW_2= \{2\bth^{(2,3)}+ \bth^{(1,4)}= (1, -1, 0, -1)^t\},
\]
and $\calS(\pi_0)_{\geq 4} \cap \calW_2 = \emptyset$. Thus $\calS(\pi_0)$ has Property 
\eqref{eq:W0}. Moreover, 
\begin{align*}
V_{\bth^{(2,3)}} &= x_1 \frac{\partial}{\partial x_2} \wedge \frac{\partial}{\partial x_3} = 
x_1x_2^{-1}x_3^{-1}\partial_2 \wedge \partial_3,\\
V_{\bth^{(1,4)}} &= x_2x_3^2 \frac{\partial}{\partial x_1} \wedge \frac{\partial}{\partial x_4}
=x_1^{-1}x_2x_3^2 x_4^{-1} \partial_1 \wedge \partial_4,\\
\mfX^2(\mX)^{\bth^{(2,3)}+ \bth^{(1,4)}} &= {\rm Span}_{\CC} 
\left\{x_1x_3 \frac{\partial}{\partial x_1} \wedge \frac{\partial}{\partial x_4}, \; 
x_2x_3 \frac{\partial}{\partial x_2} \wedge \frac{\partial}{\partial x_4}, \; 
x_3^2 \frac{\partial}{\partial x_3} \wedge \frac{\partial}{\partial x_4}\right\}\\
& = {\rm Span}_{\CC} \{x_3x_4^{-1} \partial_1 \wedge \partial_4, \; 
x_3x_4^{-1} \partial_2 \wedge \partial_4, \; x_3x_4^{-1} \partial_3 \wedge \partial_4\},\\
\mfX^2(\mX)^{2\bth^{(2,3)}+ \bth^{(1,4)}} &= {\rm Span}_{\CC} \left\{
x_1 \frac{\partial}{\partial x_2} \wedge \frac{\partial}{\partial x_4}\right\}
={\rm Span}_{\CC} \{x_1x_2^{-1}x_4^{-1} \partial_2 \wedge \partial_4\}.
\end{align*}
Let $\pi_1 = c_1V_{\bth^{(2,3)}} + c_2V_{\bth^{(1,4)}}$, where $c_1, c_2 \in \CC$,
and let
$\pi = \pi_0+\pi_1+\pi_2+\pi_3$, where
\begin{align*}
\pi_2 &=
Ax_1x_3 \frac{\partial}{\partial x_1} \wedge \frac{\partial}{\partial x_4}+ 
Bx_2x_3 \frac{\partial}{\partial x_2} \wedge \frac{\partial}{\partial x_4}+ 
Cx_3^2 \frac{\partial}{\partial x_3} \wedge \frac{\partial}{\partial x_4} \\
& =x_3x_4^{-1} \left(A\partial_1 \wedge\partial_4 + B\partial_2 \wedge \partial_4 + 
C\partial_3 \wedge \partial_4\right)\\
\pi_3 &=Dx_1 \frac{\partial}{\partial x_2} \wedge \frac{\partial}{\partial x_4} = D x_1x_2^{-1}x_4^{-1} \partial_2 \wedge \partial_4,
\end{align*}
and $A, B, C, D \in \CC$ are undetermined coefficients. Solving for $A, B, C, D$ from
\[
2[\pi_0, \pi_2]_{\rm Sch} + [\pi_1, \pi_1]_{\rm Sch}=0 \hs \mbox{and} \hs
2[\pi_0, \pi_3]_{\rm Sch} + 2[\pi_1, \pi_2]_{\rm Sch}=0,
\]
we get $B = A-2c_1c_2, C = c_1c_2-A$, and $D = -c_1A$, where $A \in \CC$ is arbitrary. We thus conclude that
for given $c_1, c_2 \in \CC$, the $\calS(\pi_0)$-admissible
algebraic Poisson deformations of $\pi_0$ along $\pi_1 =c_1 V_{\bth^{(2,3)}} + c_2V_{\bth^{(1,4)}}$ form 
a one-parameter
family $\{\pi_{(A)}(c_1, c_2): A \in \CC\}$ given by
\[
\pi_{(A)}(c_1, c_2)  = \pi_0+c_1V_{\bth^{(2,3)}} + c_2V_{\bth^{(1,4)}} +2c_1c_2 
x_3x_4^{-1}(\partial_3\wedge \partial_4 -2\partial_2 \wedge \partial_4) +AV,
,\]
where $V = x_3x_4^{-1} \left(\partial_1 \wedge\partial_4 + \partial_2 \wedge \partial_4 -
\partial_3 \wedge \partial_4\right)-c_1 x_1x_2^{-1}x_4^{-1} \partial_2 \wedge \partial_4$.
The corresponding Poisson bracket on $\CC^4$, denoted by $\{\,, \, \}_{(A, c_1, c_2)}$, is given by
\begin{align*}
    \{x_1,x_2\}_{(A, c_1, c_2)}=&x_1x_2,\\
    \{x_1,x_3\}_{(A, c_1, c_2)}=&-x_1x_3,\\
    \{x_1,x_4\}_{(A, c_1, c_2)}=&0+c_2x_2x_3^2+Ax_1x_3,\\
    \{x_2,x_3\}_{(A, c_1, c_2)}=&0+c_1x_1,\\
    \{x_2,x_4\}_{(A, c_1, c_2)}=&x_2x_4- Ac_1x_1+(A-2c_1c_2)x_2x_3,\\
    \{x_3,x_4\}_{(A, c_1, c_2)}=&-x_3x_4+(c_1c_2-A)x_3^2.
\end{align*}

Note that we have  $\calS(\pi_0)_{\geq 2} \cap \calW_1 =\{\bfw =\bth^{(2,3)}+ \bth^{(1,4)}\}$, and 
\[
v_\bfw = x_3x_4^{-1} \partial_4 = x_3 \frac{\partial}{\partial x_4}.
\]
A direct calculation shows that for any  $A, c_1, c_2 \in \CC$ we have
\begin{align*}
\exp(-Av_\bfw) \left(\pi_{(A)}(c_1, c_2)\right)& = \pi_{(A)}(c_1, c_2) -A [v_\bfw, \pi_{(A)}(c_1, c_2)] + \frac{1}{2}
A^2[v_\bfw, [v_\bfw, \pi_{(A)}(c_1, c_2)]]+ \cdots \\
& =\pi_{(A)}(c_1, c_2)-A[v_\bfw, \pi_{(A)}(c_1, c_2)] =\pi_{(A)}(c_1, c_2)-A[v_\bfw, \pi_0] + c_1A[v_\bfw, V_{\bth^{(2,3)}}]\\
& = \pi_{(A)}(c_1, c_2)-AV = \pi_{(0)}(c_1, c_2).
\end{align*}
Equivalently, one checks directly that we have the $\CC^\times$-equivariant Poisson isomorphism
\[
\varphi: \;\; (\CC^4, \{\,, \, \}_{(A, c_1, c_2)}) \longrightarrow (\CC^4, \, \{\, , \, \}_{(0,c_1c_2)})
\]
where $\varphi^* x_i = x_i$ for $i = 1, 2, 3$ and $\varphi^*x_4 = x_4+Ax_3$.
\hfill $\diamond$
}
\end{example}

We now turn to an arbitrary $\TT$-invariant algebraic Poisson structure $\pi$ on $\mX = \CC^n$.
Recall 
from \eqref{eq:mW-pi} that $\mW(\pi) \subset \calW_2$ is the set of all the non-zero $\CCsn$-weights in the 
monomial terms of $\pi$ and that $\mW(\pi)_{\geq 1} = \bigcup_{m \geq 1} \mW(\pi)_m$, where
$\mW(\pi)_m = \{\bfw^{(1)}+\cdots + \bfw^{(m)}: \bfw^{(1)}, \ldots, \bfw^{(m)} \in \mW(\pi)\}$ for $m \geq 1$.

\begin{definition}\label{def:pi-W0} 
An algebraic Poisson structure $\pi$ on $\mX=\CC^n$ is said to have {\it Property \eqref{eq:W0}}
if $\mW(\pi)$ does, i.e., if 
$\mW(\pi)_{\geq 1} \cap (\ZZ_{\geq 0})^n = \emptyset$.
\end{definition}

Let now $\pi_0$ be a $\TT$-log-symplectic log-canonical Poisson structure on $\mX = \CC^n$.  Then for 
any $\calS\subset \calS(\pi_0)$ with Property \eqref{eq:W0}, all  $\calS$-admissible algebraic Poisson 
deformations of $\pi_0$
have Property \eqref{eq:W0} by the definition of being $\calS$-admissible. 
For the converse statement, we first prepare a lemma.

\begin{lemma}\label{lm:w-Spi0}
Let $\pi$ be any $\TT$-invariant algebraic Poisson structure on $\mX = \CC^n$ with log-canonical term $\pi_0$, and let  $\bfw \in \mW(\pi)\backslash \mW(\pi)_2$. If
$\bfw \notin  \calS(\pi_0)$, then $\bfw  \in  \mW(\pi)\cap  \calW_1$.
\end{lemma}

\begin{proof}
Let $\bfw \in \mW(\pi)\backslash \mW(\pi)_2$ and write $\pi = \pi_0 + \pi_{\bfw} + \pi_{\bfw}^\prime$, where $\pi_{\bfw} \in 
\mfX^2(\mX)^{\bfw}$ is non-zero, and $\pi_\bfw^\prime \in  \mfX^2(\mX)^{\mW(\pi)\backslash \{\bfw\}}$. It then follows from $[\pi, \pi]_{\rm Sch} = 0$ that 
\[
0 = [\pi_0 + \pi_\bfw + \pi_\bfw^\prime, \;\pi_0 + \pi_\bfw + \pi_\bfw^\prime] = 2[\pi_0, \pi_\bfw] +2[\pi_0, \pi_\bfw^\prime] + 2[\pi_\bfw, \pi_\bfw^\prime]
+ [\pi_\bfw, \pi_\bfw]+[\pi_\bfw^\prime, \pi_\bfw^\prime].
\]
As the only 
weight $\bfw$-component on the right-hand side of the above equation is $[\pi_0, \pi_\bfw]$, 
we have $[\pi_0, \pi_\bfw] = 0$. Suppose now that $\bfw \notin \calS(\pi_0)$. 
Since $\bfw \in (\ker \bbeta) \backslash \{0\}$,  we have
$\mH^2_{\pi_0}(\mX)^\bfw = 0$ by the definition of $\calS(\pi_0)$. Thus $\pi_\bfw = [\pi_0, v]$ for some non-zero $v \in \mfX^1(\mX)^\bfw$.
In particular, 
$\bfw \in \mW(\pi) \cap \calW_1$.
\end{proof}

\begin{theorem}\label{thm:4.9-W0}
For any $\TT$-log-symplectic log-canonical Poisson structure $\pi_0$ on $\mX = \CC^n$, 
every $\TT$-invariant algebraic Poisson structure $\pi$ that has $\pi_0$ as its log-canonical term and 
has Property \eqref{eq:W0} is isomorphic, via a $\TT$-equivariant algebraic change of 
variables on $\mX = \CC^n$, to a maximal  (\autoref{def:calS-admi}) $\calS$-admissible algebraic Poisson deformation of $\pi_0$ for some
$\calS \subset \calS(\pi_0)$ with Property \eqref{eq:W0}.
\end{theorem}

\begin{proof}  Let $\pi$ be any $\TT$-invariant algebraic Poisson structure with log-canonical term $\pi_0$ 
and assume that $\pi$ has Property \eqref{eq:W0}. Let
$\calS_\pi = \mW(\pi)\backslash \mW(\pi)_{\geq 2}$.
By \autoref{lm:generator-W0}, we have
$(\calS_\pi)_{\geq 1} = \mW(\pi)_{\geq 1}$. It follows that $\calS_\pi$ has Property \eqref{eq:W0}.

If $\calS_\pi \subset \calS(\pi_0)$, then $\pi$ is $\calS_\pi$-admissible and even maximal $\calS_\pi$-admissible by \autoref{lm:admi-max}, which finishes the proof in this case.
Suppose that there exists $\bfw \in \calS_\pi$ and $\bfw \notin \calS(\pi_0)$. 
As in the proof of \autoref{lm:w-Spi0},
writing $\pi =\pi_0 + \pi_{\bfw} + \pi_{\bfw}^\prime$, where $\pi_{\bfw} \in 
\mfX^2(\mX)^{\bfw}$ is non-zero and $\pi_\bfw^\prime \in  \mfX^2(\mX)^{\mW(\pi)\backslash \{\bfw\}}$,
we have $\pi_\bfw = [\pi_0, v]$ for some non-zero $v \in \mfX^1(\mX)^\bfw$.
In particular, 
$\bfw \in \mW(\pi) \cap \calW_1$. Property \eqref{eq:W0} of $\mW(\pi)$ implies that 
$\bfw \in \ZZone$ and has exactly one $-1$ entry. By \autoref{lm:t-flow}, 
$v = tv_\bfw$, where $t \in \CC^\times$ and $v_\bfw \in \mfX^1(\mX)^\bfw$ is given in
\eqref{eq:v-w}. Let $\pi' = \exp(v) \pi$. 
As $2 {\bf w} \notin \calW_2$, we have $[v, [v, \pi_0]] = 0$ and
$[v, \pi_{\bfw}] = 0$. It now follows from $[v, \pi_0] +\pi_\bfw = 0$ that
\begin{align*}
\pi'& = \exp(v)\pi_0 + \exp(v)\pi_\bfw +\exp(v)\pi_{\bfw}^\prime = \pi_0 + [v, \pi_0] +\pi_{\bfw} +\exp(v)\pi_\bfw^\prime = \pi_0 +\exp(v)\pi_\bfw^\prime\\
& = 
\pi_0 +  \pi_\bfw^\prime + [v, \pi_\bfw^\prime] + \frac{1}{2}[v, [v, \pi_\bfw^\prime]]+ \cdots.
\end{align*}
As $0 \notin \mW(\pi)_{\geq 1}$, we see that $\pi'$ has log-canonical term $\pi_0$ and that
\[
\mW(\pi') \subset \{k\bfw + \bfw': k \in \ZZ_{\geq 0}, \,\bfw' \in \mW(\pi), \, \bfw' \neq \bfw\}
\subset \mW(\pi)_{\geq 1}\backslash \{\bfw\}.
\] 
In particular $\pi'$ has 
Property \eqref{eq:W0}. Moreover, 
since $\bfw \notin \mW(\pi)_{\geq 2}$, we have 
\[
\mW(\pi^\prime)_{\geq 1} \cap \calW_2 \subset (\mW(\pi)_{\geq 1} \cap \calW_2)\backslash \{\bfw\}.
\]
Consider $\calS_{\pi'} = \mW(\pi')\backslash \mW(\pi')_{\geq 2}$. If $\calS_{\pi'} \subset \calS(\pi_0)$
we are done. If not, we carry out the same process to $\pi^\prime$ to obtain 
an algebraic Poisson deformation $\pi^{\prime\prime}$ of $\pi_0$ with 
\[
\mW(\pi^{\prime\prime})_{\geq 1} \cap \calW_2 \subsetneq \mW(\pi^\prime)_{\geq 1} \cap \calW_2
\subsetneq \mW(\pi)_{\geq 1} \cap \calW_2.
\]
The process must stop as $\mW(\pi)_{\geq 1} \cap \calW_2$ is a finite set. 
At the end of the process, we obtain a $\TT$-invariant algebraic Poisson structure $\widetilde{\pi}$ with log-canonical term
$\pi_0$ such that $\widetilde{\pi}$ is 
isomorphic to $\pi$ via a $\TT$-equivariant algebraic change of variables and that  
$\calS_{\widetilde{\pi}} :=\mW(\widetilde{\pi})\backslash \mW(\widetilde{\pi})_{\geq 2}$
is contained in $\calS(\pi_0)$ and  has Property \eqref{eq:W0}. 
In particular,  $\widetilde{\pi}$ is $\calS_{\widetilde{\pi}}$-admissible.
The fact that $\widetilde{\pi}$ is maximal $\calS_{\widetilde{\pi}}$-admissible follows from the following
\autoref{lm:admi-max}.
\end{proof}

\begin{lemma}\label{lm:admi-max}
Assume that  $\calS \subset \calS(\pi_0)$ has Property \eqref{eq:W0}, and let $\pi$ be 
an $\calS$-admissible
algebraic Poisson deformation  of $\pi_0$. Then 
$\pi$ is maximal $\calS$-admissible if and only if $\calS = \mW(\pi)\backslash \mW(\pi)_{\geq 2}$.
\end{lemma}

\begin{proof}
Assume first that
$\pi$ is maximal $\calS$-admissible. Then 
$\calS \subset \mW(\pi)$ by definition, and as $\mW(\pi) \subset \calS_{\geq 1}$ and as $\calS$ is 
linearly independent, we have 
$\calS \subset \mW(\pi)\backslash \mW(\pi)_{\geq 2}$. Moreover, for any $\bfw \in\mW(\pi)\backslash \mW(\pi)_{\geq 2}$, if $\bfw \notin \calS$, then $\bfw \in \calS_{\geq 2} \subset \mW(\pi)_{\geq 2}$,
contradicting the assumption that $\bfw \notin \mW(\pi)_{\geq 2}$. Thus $\bfw \in \calS$.
It follows that $\calS = \mW(\pi)\backslash \mW(\pi)_{\geq 2}$. Conversely, assume that 
$\calS = \mW(\pi)\backslash \mW(\pi)_{\geq 2}$. Then $\calS \subset \mW(\pi)$. 
Write $\pi = \pi_0 + \pi_1 + \pi_{\geq 2}$, where $\pi_1 \in \mfX^2(\mX)^{\calS}$ 
and $\pi_{\geq 2} \in \mfX^2(\mX)^{\calS_{\geq 2}}$. 
As no monomial terms in $\pi_{\geq 2}$ have $(\CC^\times)^n$-weights in $\calS$, 
every element in $\calS$ appears as a $(\CC^\times)^n$-weight of some monomial term in 
$\pi_1$. Thus $\pi$ is maximal $\calS$-admissible.
\end{proof}

\subsection{Algebraic Poisson deformations under Property \eqref{eq:W1}}\label{ss:W1}
Recall again from \eqref{eq:calW-p} that 
$\calW_1 = \{\bfw \in \ZZone: |J_\bfw|\leq 1\} = \{\bfw \in \ZZ^n: \mfX^1(\mX)^\bfw \neq 0\}$. 

\begin{definition}\label{def:W1}
1) A finite subset $\mW$ of $\ZZone$ is said to have Property \eqref{eq:W1} if 
\begin{equation}\label{eq:W1}
\mW_{\geq 1} \cap \calW_1 = \emptyset, \hs \mbox{equivalently}\hs \mfX^1(\mX)^{\mW_{\geq 1}} = 0\tag{W1},
\end{equation}
which is also equivalent to every element in $\mW_{\geq 1} \cap \ZZone$ having at least two entries of $-1$;

2) An algebraic Poisson structure $\pi$ on $\mX = \CC^n$ is said to have Property \eqref{eq:W1} if
$\mW(\pi)$ does.
\end{definition}

\begin{remark}\label{rk:W0-to-W1}
{\rm
1) As $(\ZZ_{\geq 0})^n = \calW_0 \subset \calW_1$, Property \eqref{eq:W1} 
 implies Property \eqref{eq:W0};

2) If an algebraic Poisson structure $\pi$ on $\mX = \CC^n$ has Property \eqref{eq:W1}, then every $\bfw \in \mW(\pi)$ has exactly two $-1$ entries and 
$\dim \mfX^2(\mX)^\bfw = 1$ (see \autoref{lm:X-w}). Consequently, all the non-log-canonical  monomial terms in $\pi_{\rm tail}$ have different (non-zero) 
$\CCsn$-weights. 
\hfill $\diamond$
}
\end{remark}

We now have the stronger version of \autoref{thm:4.8-W0} under Property \eqref{eq:W1}.

\begin{theorem}\label{thm:4.8-W1}
Let $\pi_0$ be any $\TT$-log-symplectic log-canonical Poisson structure on $\mX = \CC^n$.
Suppose that $\calS \subset \calS(\pi_0)$ has Property \eqref{eq:W1}. Then for any 
$\pi_1 \in \mfX^2(\mX)^{\calS}$ there is one and only one
$\calS$-admissible algebraic Poisson deformation $\pi$ of $\pi_0$ along $\pi_1$, and 
$\pi$ has Property \eqref{eq:W1}.
\end{theorem}

\begin{proof} Property \eqref{eq:W1} of $\calS$ implies Property \eqref{eq:W0} 
and that $\calG_{\calS_{\geq 2} \cap \calW_1}$
is trivial. \autoref{thm:4.8-W1} now follows from \autoref{thm:4.8-W0}.
\end{proof}

For any $\TT$-log-symplectic log-canonical Poisson structure $\pi_0$
on $\mX = \CC^n$, we also have the following stronger version of \autoref{thm:4.9-W0} under Property \eqref{eq:W1}.

\begin{theorem}\label{thm:4.9-W1} 
Let $\pi$ be any $\TT$-invariant algebraic Poisson structure 
with Property \eqref{eq:W1} and 
log-canonical term $\pi_0$, and let $\calS_\pi =\mW(\pi)\backslash \mW(\pi)_{\ge2}$. 
Then $\calS_\pi=\mW(\pi)\backslash \mW(\pi)_{2}  \subset \calS(\pi_0)$ and has 
Property \eqref{eq:W1}, and $\pi$ is the unique 
$\calS_\pi$-admissible algebraic Poisson deformation of $\pi_0$ along $\pi_1$, where 
$\pi_1$ is the $\mfX^2(\mX)^{\calS_\pi}$-component of $\pi$ with  respect to the
$\CCsn$-weight space decomposition of $\mfX^2(\mX)$.
\end{theorem}

\begin{proof} Let $\pi$ be a $\TT$-invariant algebraic Poisson structure with log-canonical term $\pi_0$ and 
assume that $\pi$ has  Property \eqref{eq:W1}. Then in particular $\mW(\pi) \cap \calW_1 = \emptyset$.
Let 
$\calS_\pi =\mW(\pi)\backslash \mW(\pi)_{\geq 2}$ and 
\[
\calS_\pi^\prime =\mW(\pi)\backslash \mW(\pi)_{2}.
\]
and note 
that $\calS_\pi \subset \calS_\pi^\prime$. 
As $\mW(\pi) \cap \calW_1 = \emptyset$, by \autoref{lm:w-Spi0} we have $\calS_\pi^\prime \subset \calS(\pi_0)$.
As Property \eqref{eq:W1} of $\mW(\pi)$ implies  $0 \notin \mW(\pi)_{\geq 1}$, 
by \autoref{lm:generator-W0} we have $(\calS_\pi)_{\geq 1} = \mW(\pi)_{\geq 1}$, and thus we also have
$(\calS_\pi^\prime)_{\geq 1} = (\calS_\pi)_{\geq 1}=\mW(\pi)_{\geq 1}$. In particular $0 \notin (\calS_\pi^\prime)_{\geq 1}$.
By \autoref{ld:indep}, $\calS_\pi^\prime \subset \calS(\pi_0)$ is linearly independent. It now follows from
$\calS_\pi\subset \calS_{\pi}^\prime$ and $\calS_\pi^\prime \subset  (\calS_\pi)_{\geq 1}$ that 
$\calS_\pi = \calS_\pi^\prime$.

As $\mW(\pi) \subset (\calS_\pi)_{\geq 1}$ and $\calS_\pi \subset \mW(\pi)$, it follows from
\autoref{thm:4.8-W1} that $\pi$ is the unique  $\calS_\pi$-admissible algebraic Poisson deformation of $\pi_0$ along
with $\pi_1\in \mfX^2(\mX)^{\calS_\pi}$ as described.
\end{proof}

\begin{remark}\label{rm:S-pi-unieue}
{\rm
By \autoref{lm:admi-max}, the set $\calS_\pi$ in \autoref{thm:4.9-W1} is the only subset $\calS$ of 
$\calS(\pi_0)$
with Property \eqref{eq:W1} such that $\pi$ is a maximal $\calS$-admissible algebraic Poisson deformation of $\pi_0$.
\hfill $\diamond$
}
\end{remark}

Let again $\pi_0$ be any $\TT$-log-symplectic log-canonical Poisson structure 
on $\mX = \CC^n$.
For  $\calS \subset \calS(\pi_0)$, let 
\[
\CC^\calS = \{c=(c_\bth \in \CC)_{\bth \in \calS}\} \hs \mbox{and} \hs 
(\CC^\times)^\calS = \{c=(c_\bth \in \CC^\times)_{\bth \in \calS}\}.
\]

\begin{notation}\label{nota:pic-W1} 
{\rm For 
$\calS \subset \calS(\pi_0)$ with Property \eqref{eq:W1} and 
$c = (c_\bth)_{\bth \in \calS}  \in \CC^\calS$, let
\begin{equation}\label{eq:pi-1-c}
\pi_1^\calS(c) = \sum_{\bth \in \calS}{c_\bth} V_{\bth}\in \mfX^2(\mX)^{\calS},
\end{equation}
and we denote the unique $\calS$-admissible algebraic Poisson deformation of $\pi_0$ along $\pi_1^\calS(c)$ by
\begin{equation}\label{eq:pi-S-c}
\pi^\calS(c) = \pi_0 + \pi_1^\calS(c) + \pi_2^\calS (c) + \cdots +\pi_N^\calS(c), 
\end{equation}
where $N \geq 1$ and 
$\pi_m^\calS(c) \in \mfX^2(\mX)^{{\calS}_m}$ for $2 \leq m \leq N$. 
\hfill $\diamond$
}
\end{notation}

\begin{remark}\label{rk:pic-W1}
{\rm Assume $\calS \subset \calS(\pi_0)$ has Property \eqref{eq:W1} and let 
$\pi^\calS(c)$ be as in \eqref{eq:pi-S-c} for $c \in \CC^\calS$. 
By \autoref{rmk:in-practice} and 
\autoref{thm:4.8-W1}, $\{\pi_m^\calS(c): m \geq 2\}$ 
is the unique solution to 
the recursive relations in \eqref{eq:pi-pim-1}. It also follows from \eqref{eq:pi-pim-1} 
that for each $m \geq 2$, $\pi_m^\calS(c)$ is homogeneous polynomial in $c$ with homogeneous degree $m$.
For a given $c = (c_\bth)_{\bth \in \calS}  \in \CC^\calS$, by definition
$\pi^\calS(c)$ is maximal ${\rm Supp}(c)$-admissible, where 
${\rm Supp}(c) = \{\bth \in \calS: c_\bth \neq 0\}$.
\hfill $\diamond$
}
\end{remark}

\begin{proposition}\label{pr:scaling-to-c}
For any $\calS \subset \calS(\pi_0)$ with Property \eqref{eq:W1}, any two maximal $\calS$-admissible algebraic Poisson deformations of $\pi_0$ are isomorphic by a rescaling of the coordinates $(x_1, \ldots, x_n)$.
\end{proposition}

\begin{proof}
Assume that $\calS \neq \emptyset$ for there is nothing to prove when $\calS = \emptyset$.
Since $\calS$ is linearly independent, for every $\bth \in \calS$ we can choose  $\xi^\bth\in\mathbb{Q}^n$ such that $(\xi^\bth)^t  \boldsymbol{\theta} =1$ 
and $(\xi^\bth)^t {\bth'} = 0$ for $\bth'\in\calS\setminus\{\bth\}$. 
Let $M\in \mathbb{Z}_{\ge1}$ be such that $M\xi^\bth \in \mathbb{Z}^n$ for every  $\bth \in \calS$. Let 
$(\CC^\times)^\calS = \{t = (t_\bth \in \CC^\times)_{\bth \in \calS}\}$ act 
on $\mathsf{X} =\CC^n$ via
\begin{equation}\label{eq:t-xi}
t \cdot x_i \mapsto \left(\prod_{\bth \in \calS} t_\bth^{M \xi^{\bth}_i}\right) x_i,\hs i \in [1, n].
\end{equation}
Then $\mfX^2(\mX)^{\calS_m}$ is $(\CC^\times)^\calS$-invariant for every $m \geq 1$ and 
$t\cdot V_{\boldsymbol{\theta}} = t_\bth^M V_{\boldsymbol{\theta}}$ for every $t \in (\CC^\times)^\calS$ and 
$\bth \in \calS$. 

Let $c = (c_\bth\in \CC^\times)_{\bth \in \calS}\in (\CC^\times)^\calS$, and 
let ${\bf 1}_\calS = (1, \ldots, 1) \in (\CC^\times)^\calS$.
Choose $t = (t_\bth)_{\bth \in \calS} \in (\CC^\times)^\calS$ such that $t_\bth^M = c_\bth$ for each $\bth$. Then 
$t\cdot (\sum_{\bth \in \calS} V_{\bth}) = \sum_{\bth \in \calS} c_\bth V_{\bth}$. Thus 
$t\cdot \pi^\calS({\bf 1}_\calS)$ is an $\calS$-admissible algebraic Poisson 
deformation of $\pi_0$ along $\sum_{\bth \in \calS} c_\bth V_{\bth}$. By
\autoref{thm:4.8-W1}, $t\cdot \pi^\calS({\bf 1}_\calS) = \pi^\calS(c)$. Thus every 
$\calS$-maximal admissible algebraic Poisson deformation $\pi^\calS(c)$ of
$\pi_0$ is isomorphic to $\pi^\calS({\bf 1}_\calS)$ by a rescaling of the coordinates $(x_1, \ldots, x_n)$.
\end{proof}

\begin{remark}\label{rmk:scaling-to-c}
{\rm In the proof of \autoref{pr:scaling-to-c}, if the sub-lattice of $\ZZ^n$ 
generated by $\calS$ is saturated, then $\calS$ can be extended to a $\ZZ$-basis of $\ZZ^n$, so we can take
 $\xi^\bth\in\mathbb{Z}^n$ and thus $M = 1$. In this case we have
$c\cdot \pi^\calS({\bf 1}_\calS) = \pi^\calS(c)\in {\rm Poi}(\mX)$ for every $c \in (\CC^\times)^\calS$.
\hfill $\diamond$
}
\end{remark}

When $\calS(\pi_0)$ satisfies Property \eqref{eq:W1}, we then have the family $\pi^{\calS(\pi_0)}(c)$
of algebraic Poisson deformations of $\pi_0$, where $c \in \CC^{\calS(\pi_0)}$. To summarize the properties of the
family $\pi^{\calS(\pi_0)}(c)$, let
\[
\pi_1^{\calS(\pi_0)}(c) = \sum_{\bth \in \calS(\pi_0)} c_\bth V_\bth \hs \hs \mbox{for}\;\;\;
c = (c_\bth)_{\bth \in \calS(\pi_0)} \in \CC^{\calS(\pi_0)}.
\]
 Let 
$\calS(\pi_0)_{\geq 0} = \{0\} \sqcup \calS(\pi_0)_{\geq 1}$, and for 
$\calS \subset \calS(\pi_0)$ consider the direct sum decomposition 
\begin{equation}\label{eq:decomp}
\mfX^2(\mX)^{\calS(\pi_0)_{\geq 0}}  = \mfX^2(\mX)^{\calS_{\geq 0}} + \mfX^2(\mX)^{(\calS_{\geq 0})^o},
\end{equation}
where 
$\calS_{\geq 0} = \{0\} \sqcup \calS_{\geq 1}$ and  
$(\calS_{\geq 0})^o = \calS(\pi_0)_{\geq 0} \backslash \calS_{\geq 0}$. 
For  $c = (c_\bth)_{\bth \in \calS(\pi_0)} \in \CC^{\calS(\pi_0)}$ let
 $c|_\calS = (c_\bth)_{\bth \in \calS} \in \CC^{\calS}$. 

\begin{theorem}\label{thm:max-family-W1}
Suppose that $\pi_0$ is a $\TT$-log-symplectic log-canonical Poisson structure  on $\mX = \CC^n$
such that $\calS(\pi_0)$ has Property \eqref{eq:W1}. Then for 
every $c = (c_\bth)_{\bth \in \calS(\pi_0)} \in \CC^{\calS(\pi_0)}$ there is a unique 
$\TT$-invariant Poisson structure $\pi^{\calS(\pi_0)}(c)$ on $\mX$ of the form
\begin{equation}\label{eq:pic-max}
\pi^{\calS(\pi_0)}(c) = \pi_0 + \pi_1^{\calS(\pi_0)}(c) +\cdots + \pi_N^{\calS(\pi_0)}(c),
\end{equation}
where $N \geq 1$ and for each $m \in [2, N]$, $\pi_m^{\calS(\pi_0)}(c)
\in \mfX^2(\mX)^{\calS(\pi_0)_m}$ and is homogeneous polynomial in $c$ with homogeneous degree $m$.
Furthermore, for any $\calS \subset \calS(\pi_0)$ and $c = (c_\bth)_{\bth \in \calS(\pi_0)} \in \CC^{\calS(\pi_0)}$, one has
\begin{equation}\label{eq:pi-cc}
\pi^{\calS(\pi_0)}(c) = \pi^\calS(c|_\calS) \hs \mbox{{\rm mod}} \;\; \sum_{\bth \in \calS(\pi_0)\backslash \calS}
c_\bth \mfX^2(\mX)^{(\calS_{\geq 0})^o}.
\end{equation}
\end{theorem}

\begin{proof}  The existence and uniqueness of $\pi^{\calS(\pi_0)}(c)$ as in \eqref{eq:pic-max}
is a special case of \autoref{thm:4.8-W1} and \autoref{rk:pic-W1}.
Let $c \in \CC^{\calS(\pi_0)}$ and 
write $\pi^{\calS(\pi_0)}(c)  = P_1+P_2$ with $P_1 \in \mfX^2(\mX)^{\calS_{\geq 0}}$
and $P_2 \in \mfX^2(\mX)^{(\calS_{\geq 0})^o}$. As 
$[\pi^{\calS(\pi_0)}(c), \, \pi^{\calS(\pi_0)}(c)]_{\rm Sch} =0$,
one has
$[P_1, \, P_1]_{\rm Sch} + 2[P_1, \, P_2]_{\rm Sch} + [P_2, \, P_2]_{\rm Sch}=0$.
As $\calS_{\geq 0}$ is closed under addition in 
$\calS(\pi_0)_{\geq 0}$ and
$\calS(\pi_0)_{\geq 0}+ (\calS_{\geq 0})^o\subset  (\calS_{\geq 0})^o$,
we have
\[
[P_1, \, P_1]_{\rm Sch}=0 \hs \mbox{and} \hs 2[P_1, \, P_2]_{\rm Sch} + [P_2, \, P_2]_{\rm Sch}=0.
\]
Since  $\pi_1^\calS(c|_{\calS}) 
=\sum_{\bth \in \calS}c_\bth V_{\bth}$ is the $\mfX^2(\mX)^{\calS}$-component of $P_1$,
we know that $P_1$ is an $\calS$-admissible algebraic Poisson deformation of
$\pi_0$ along $\pi_1^\calS(c|_\calS) \in \mfX^2(\mX)^{\calS}$. By \autoref{thm:4.8-W1}, we have $P_1 = \pi^\calS(c|_{\calS})$. If $c$ is such that $c_\bth = 0$ for all $\bth \in \calS(\pi_0)\backslash \calS$, 
then both $\pi^{\calS(\pi_0)}(c)$ and $\pi^\calS(c|_\calS)\in \mfX^2(\mX)$ are $\calS(\pi_0)$-admissible
algebraic deformation of  $\pi_0$ along $\pi_1^\calS(c|_\calS)$, so $\pi^{\calS(\pi_0)}(c)=\pi^\calS(c|_\calS)$.
It follows that \eqref{eq:pi-cc} holds.
\end{proof}

We now formulate the following direct consequence of \autoref{thm:max-family-W1}.

\begin{corollary}\label{cor:member-summand}
Suppose that $\pi_0$ is a $\TT$-log-symplectic log-canonical Poisson structure  on $\mX = \CC^n$
such that $\calS(\pi_0)$ has Property \eqref{eq:W1}.

1) The  family $\pi^{\calS(\pi_0)}(c)$ in \eqref{eq:pic-max}, where
$c \in \CC^{\calS(\pi_0)}$, consists precisely of all $\calS$-admissible algebraic Poisson deformations of $\pi_0$
for all subsets $\calS$ of $\calS(\pi_0)$;

2) For any $\calS \subset \calS(\pi_0)$, every 
$\calS$-admissible algebraic Poisson deformation of $\pi_0$ is a direct summand of 
$\pi^{\calS(\pi_0)}(c)$ for some $c \in (\CC^\times)^{\calS(\pi_0)}$, with respect to the 
decomposition in \eqref{eq:decomp}.
\end{corollary}

\begin{definition}\label{defn:max-W1}
{\rm 
When $\calS(\pi_0)$ has Property \eqref{eq:W1}, we call 
the family $\pi^{\calS(\pi_0)}(c)$ in \eqref{eq:pic-max} 
the {\it maximal family of admissible algebraic Poisson 
deformations of $\pi_0$}, and we call 
any $\pi^{\calS(\pi_0)}(c)$ with $c \in (\CC^\times)^{\calS(\pi_0)}$
 a {\it maximal admissible algebraic Poisson deformation of $\pi_0$}.
}
\end{definition}

\subsection{Negatively bordered weights and iterated Poisson-Ore extensions}
Our primary examples of algebraic Poisson structures $\pi$ with Properties \eqref{eq:W0} or 
\eqref{eq:W1}, to be given in the next $\S$\ref{s:action-data}, have the additional property that 
the elements in $\mW(\pi)$ are {\it negatively bordered} in the sense we now define.
Recall that an element $\bfw \in \ZZ^n$ is a column vector and $\bfw^t$ stands for 
 the transpose of $\bfw$.

\begin{definition}\label{defn:n-b}
1) An element $\bfw \in (\ZZ_{\geq -1})^n$ is said to be {\it negatively bordered on the right} if 
there exists  $k \in [1, n]$ such that 
$\bfw^t = (\bfw_1, \ldots, \bfw_{k-1}, -1, 0, \ldots, 0)$.
For such a $\bfw$, we write 
${\rm rb}(\bfw) = k$.

2) A finite subset of $(\ZZ_{\geq -1})^n$ is said to be {\it negatively bordered on the right}
if all of its elements are;

3) An element $\bfw \in (\ZZ_{\geq -1})^n$ is said to be {\it negatively bordered on both sides} if 
there exist $1 \leq j < k \leq n$ such that 
$\bfw^t = (0, \ldots, 0, -1, \bfw_{j+1}, \ldots, \bfw_{k-1}, -1, 0, \ldots, 0)$, and we write
${\rm lb}(\bfw) = j$ and ${\rm rb}(\bfw) = k$;

4) A finite subset of $(\ZZ_{\geq -1})^n$ is said to be {\it negatively bordered on both sides}
if all of its elements are.
\end{definition}

\begin{lemma}\label{lm:nbr-W0} Let $\mW$ be any finite subset of $(\ZZ_{\geq -1})^n$.

1) If $\mW$ is negatively bordered on the right, so is  
$\mW_{\geq 1} \cap (\ZZ_{\geq -1})^n$, 
and $\mW$ has Property \eqref{eq:W0};

2) If $\mW$ is negatively bordered on both sides, so is  
$\mW_{\geq 1} \cap (\ZZ_{\geq -1})^n$, 
and $\mW$ has Property \eqref{eq:W1}.
\end{lemma}

\begin{proof}
Assume first that $\mW$ is negatively bordered on the right. Let $\bfw  \in \mW_{\geq 1}\cap \ZZone$ and write
\begin{equation}\label{eq:bfw-sum}
\bfw = (\bfw_1, \ldots, \bfw_n)^t = \sum_{i=1}^l a_i \bfw^{(i)},
\end{equation}
where $l \geq 1$ and $a_i \in \ZZ_{\geq 1}$ and 
$\bfw^{(i)} \in \mW$ for each $i \in [1, l]$.
Let $k = {\rm max}\{{\rm rb}(\bfw^{(i)}): i \in [1, l]\}$. Then $\bfw_k \leq -1$ and $\bfw_{k'} = 0$ for
all $k' \in [k+1, n]$. It thus follows from
$\bfw \in \ZZone$ that $\bfw_k = -1$, so $\bfw$ is negatively bordered on the right with ${\rm rb}(\bfw) = k$. In particular, every 
element in $\mW_{\geq 1} \cap \ZZone$ has at least one entry of $-1$. Thus
$\mW_{\geq 1} \cap (\ZZ_{\geq 0})^n = \emptyset$, i.e., 
$\mW$ has Property \eqref{eq:W0}. 

Assume now that $\mW$ is negatively bordered on both sides. Let again $\bfw
\in \mW_{\geq 1} \cap \ZZone$ be as in \eqref{eq:bfw-sum}, and let
$j = {\rm min}\{{\rm lb}(\bfw^{(i)}): i \in [1, l]\}$. The same arguments as above show that 
$\bfw_j = -1$ and $\bfw_{j'} = 0$ for all $j' \in [1, j-1]$. Moreover,
if $i \in [1, l]$ is such that $j = {\rm lb}(\bfw^{(i)})$, then
$j = {\rm lb}(\bfw^{(i)})< {\rm rb}(\bfw^{(i)}) \leq k$. Thus $\bfw^t = 
(0, \ldots, 0, -1, \bfw_{j+1}, \ldots, \bfw_{k-1}, -1, 0, \ldots, 0)$, so $\bfw$ is bordered on both sides. In particular, every 
element in $\mW_{\geq 1} \cap \ZZone$ has at least two entries of $-1$, so 
$\mW_{\geq 1} \cap \calW_1 = \emptyset$, i.e., 
$\mW$ has Property \eqref{eq:W1}. 
\end{proof}

To motivate the consideration of negatively bordered weights,  we 
recall the notion of iterated Poisson-Ore extensions from \cite{Oh06}.

\begin{definition}\label{def:Poi-Ore-1}
{\rm
An algebraic Poisson structure $\pi$ on $\mX = \CC^n$ is said to be an {\it iterated Poisson-Ore extension} in the (ordered system of linear) coordinates $(x_1, \ldots, x_n)$ if
\begin{equation}\label{eq:iterated-1}
\{\CC[x_1, \ldots, x_{k-1}], \;x_k\}_\pi \subset x_k \CC[x_1, \ldots, x_{k-1}] +\CC[x_1, \ldots, x_{k-1}],
\hs k \in [2, n].
\end{equation}
If, in addition to \eqref{eq:iterated-1}, one also has
\begin{equation}\label{eq:iterated-2}
\{x_k, \, \CC[x_{k+1}, \ldots, x_{n}]\}_\pi 
\subset x_k \CC[x_{k+1}, \ldots, x_{n}] +\CC[x_{k+1}, \ldots, x_{n}], \hs k \in [1, n-1],
\end{equation}
we say that $\pi$ is a {\it symmetric iterated Poisson-Ore extension} in $(x_1, \ldots, x_n)$. 
}
\end{definition}

Recall that for  $\pi \in {\rm Poi}(\mX)$ expressed in the coordinates $(x_1, \ldots, x_n)$, 
 $\mW(\pi)$ denotes the set of all 
non-zero $\CCsn$-weights of the monomial terms in $\pi$.

\begin{lemma}\label{lm:nbr-iterated}
Let $\pi \in {\rm Poi}(\mX)$.

1) If $\mW(\pi)$ is negatively bordered on 
the right, then $\pi$ is an iterated Poisson-Ore extension in
the  coordinates $(x_1, \ldots, x_n)$;

2) If $\mW(\pi)$ is negatively bordered on both sides,
then $\pi$ is a symmetric iterated Poisson-Ore extension in
the coordinates $(x_1, \ldots, x_n)$.
\end{lemma}

\begin{proof} 
We assume that $\pi$ is not log-canonical, for otherwise $\mW(\pi) = \emptyset$ and there is nothing to prove.

Assume first that $\mW(\pi)$ is negatively bordered on the right
and consider any monomial term 
\[
x^\bfw \partial_j \wedge \partial_k = x^{\bfw + e_j+e_k} \frac{\partial}{\partial x_j} \wedge 
\frac{\partial}{\partial x_k}
\]
of $\pi_{\rm tail}$,
where $\bfw \in \mW(\pi)$, $1 \leq j < k \leq n$, and $J_\bfw \subset \{j, k\}$.
Since $\bfw \in \mW(\pi)$ is negatively bordered on the right,
either ${\rm rb}(\bfw) = k$, in which case
$\bfw_k = -1$ and 
\[
x^\bfw \partial_j \wedge \partial_k = x_j \left(\prod_{i=1}^{k-1}x_i^{\bfw_i}\right) 
\frac{\partial}{\partial x_j} \wedge \frac{\partial}{\partial x_k} 
\in \CC[x_1, \ldots, x_{k-1}] \frac{\partial}{\partial x_j} \wedge \frac{\partial}{\partial x_k},
\]
or ${\rm rb}(\bfw) = j$, in which case
$\bfw_j = -1$, $\bfw_i = 0$ for all $i \in [j+1, n]$, and 
\[
x^\bfw \partial_j \wedge \partial_k = x_k \left(\prod_{i=1}^{j-1}x_i^{\bfw_i}\right) 
\frac{\partial}{\partial x_j} \wedge \frac{\partial}{\partial x_k} 
\in x_k\CC[x_1, \ldots, x_{k-1}] \frac{\partial}{\partial x_j} \wedge \frac{\partial}{\partial x_k}.
\]
It thus follows from \autoref{def:Poi-Ore-1} that $\pi$ is an iterated Poisson-Ore extension in
$(x_1, \ldots, x_n)$. 

If, in addition, $\mW(\pi)$ is negatively bordered on both sides, by considering
${\rm lb}(\bfw)$ for each $\bfw \in \mW(\pi)$, the arguments above show that 
$\pi$ is also an iterated Poisson-Ore extension in the (ordered system of linear) coordinates 
$(x_n, x_{n-1}, \ldots, x_1)$, so $\pi$ is 
a symmetric iterated Poisson-Ore extension in
$(x_1, \ldots, x_n)$.
\end{proof}

By \autoref{lm:nbr-W0}, when $\calS(\pi_0)$ is negatively bordered on the right, we can apply
\autoref{thm:4.8-W0} to obtain $\calS$-admissible algebraic Poisson deformations of $\pi_0$ for 
every $\calS \subset \calS(\pi_0)$, and 
all of them are iterated Poisson-Ore extensions in the ordered coordinates 
$(x_1, \ldots, x_n)$. If $\calS(\pi_0)$ is negatively bordered both sides, 
the $\calS$-admissible algebraic Poisson deformations of $\pi_0$ for 
all $\calS \subset \calS(\pi_0)$ are then 
symmetric iterated Poisson-Ore extensions in the ordered coordinates 
$(x_1, \ldots, x_n)$.

We now prove a fact to be used in $\S$\ref{ss:Cartan-integers}. For an arbitrary 
$\TT$-log-symplectic log-canonical Poisson
structure $\pi_0$ on $\CC^n$, recall that $\calE^+(\pi_0)$ denotes the set of all $(j, k)$ with 
$1 \leq j < k \leq n$ such that there is a (necessarily unique) $\bth^{(j, k)} \in \calS(\pi_0)$
with $\bth^{(j, k)}_j = \bth_k^{(j, k)} = -1$.

\begin{lemma}\label{lm:if-bordered-both} Suppose that $\pi_0$ is a $\TT$-log-symplectic log-canonical Poisson
structure on $\CC^n$ such that
$\calS(\pi_0)$ is bordered on both sides. If 
$(j, k), (k, l) \in \calE^+(\pi_0)$, then $a = a'$, where 
$a, a' \in \CC^\times$ are such that 
\[
\blambda \bth^{(j, k)} = a(e_j-e_k) \hs \mbox{and} \hs \blambda \bth^{(k, l)} = a'(e_k-e_l).
\]
\end{lemma}

\begin{proof}
By \autoref{lm:wj-wk}, we have $a(\bth^{(k, l)}_j-\bth^{(k, l)}_k) = -a^\prime(\bth^{(j,k)}_k-\bth^{(j,k)}_l)$.
It then follows from $\bth^{(k, l)}_j=\bth^{(j, k)}_l=0$ and $\bth^{(k, l)}_k=\bth^{(j,k)}_j=-1$ that
$a = a'$.
\end{proof}
\subsection{$\TT$-Pfaffians} Let again $\pi_0$ be any $\TT$-log-smplectic log-canonical Poisson structure on $\mX = \CC^n$.
We now make a digression on {\it $\TT$-Pfaffians} in order to explain the terminology {\it smoothable weights} for elements in $\calS(\pi_0)$.

Assume first that $\pi$ is an arbitrary $\TT$-invariant algebraic Poisson structure on 
$\mX = \CC^n$ of rank $2r$, i.e., the maximal dimension of the symplectic leaves of $\pi$ is $2r$. For $\xi \in \t$, let $\xi^\flat \in \mfX^1(\mX)$ be
the infinitesimal generator of the $\TT$-action on $\mX$ in the direction of $\xi$, i.e.,  
\[
\xi^\flat = \sum_{j=1}^n \beta_j(\xi) x_j \frac{\partial}{\partial x_j} = \sum_{j=1}^n \beta_j(\xi)\partial_j \in \mfX^1(\mX).
\]
By a {\it $\TT$-Pfaffian of $\pi$} we mean any non-zero element of $\mfX^n(\mX)$ of the form
\[
{\rm Pf}_\TT(\pi) = \pi^r \wedge \xi_1^\flat \wedge \cdots \wedge \xi_{n-2r}^\flat \in \mfX^n(\mX),
\]
where $\pi^r = \pi \wedge \cdots \wedge \pi \in \mfX^{2r}(\mX)$, and 
$\{\xi_1, \ldots, \xi_{n-2r}\}$ is a  linearly independent subset of $\t$. Existence of 
$\TT$-Pfaffian of $\pi$ is equivalent to $\pi$ having an open $\TT$-leaf, and in such a case 
${\rm Pf}_\TT(\pi)$ is unique up to a non-zero scalar multiple (\cite[Remark 2.7]{LM:T-leaves}).

Coming back to the $\TT$-log-symplectic log-canonical Poisson structure $\pi_0$, let 
$2r$ be the rank of the matrix $\blambda$. As $\pi_0$ is $\TT$-log-symplectic, by \autoref{rk:open-T-leaf}
there exists 
a linearly independent subset  $\{\xi_1, \ldots, \xi_{n-2r}\}$ of $\t$ such that a $\TT$-Pfaffian of $\pi_0$ is given by
\begin{equation}\label{eq:Pf-pi-0}
{\rm Pf}_\TT(\pi_0) = \pi_0^r \wedge \xi_1^\flat \wedge \cdots \wedge \xi_{n-2r}^\flat =x_1x_2 \cdots x_n \frac{\partial}{\partial x_1} \wedge \frac{\partial}{\partial x_2}
\wedge \cdots \wedge \frac{\partial}{\partial x_n}.
\end{equation}

\begin{lemma}\label{lm:Pf-bth}
Let $\calS$ consist of a single element $\bth = (\bth_1, \ldots, \bth_n)^t \in
 \calS(\pi_0)$ with $\bth_j = \bth_k = -1$ for $1\leq j < k\leq n$.  
Then for any $c \in \CC^\times$, the unique
$\calS$-admissible algebraic Poisson deformation of $\pi_0$ along $cV_\bth$ is 
 $\pi(c) = \pi_0 + cV_\bth$ (see \eqref{eq:V-theta}). Moreover, 
$\pi(c)$ also has rank $2r$ and has a $\TT$-Pfaffian
\begin{equation}\label{eq:Pf-pic}
{\rm Pf}_\TT(\pi(c)) = \pi(c)^r \wedge  \xi_1^\flat \wedge \cdots \wedge \xi_{n-2r}^\flat =  
 \left(\prod_{i\neq j, k} x_i\right)
 \left(x_jx_k -\frac{c}{a} \prod_{i\neq j, k} x_i^{\bth_i}\right)
 \frac{\partial}{\partial x_1} \wedge \cdots \wedge 
\frac{\partial}{\partial x_n},
\end{equation}
where $\xi_1, \ldots, \xi_{n-2r} \in \t$ are as in 
\eqref{eq:Pf-pi-0}, and $a \in \CC^\times$ is such that $\blambda \bth = a(e_j-e_k)$.
\end{lemma}

\begin{proof}
The set $\calS = \{\bth\}$ has Property \eqref{eq:W1}. 
Let $\bth' = \bth + e_j+e_k$ so that (see notation in $\S$\ref{ss:nota-intro})
\[
V_\bth = \left(\prod_{i \neq j, k} x_i^{\bth_i}\right)\frac{\partial}{\partial x_j} \wedge \frac{\partial}{\partial x_k} = 
x^{\bth'} \frac{\partial}{\partial x_j} \wedge \frac{\partial}{\partial x_k}.
\]
As $m\bth \notin \ZZone$ for any $m \geq 2$, the unique $\calS$-admissible 
algebraic Poisson deformation of $\pi_0$ along $\pi_1 = cV_{\bth}$  for any $c \in \CC^\times$ is $\pi(c) = \pi_0 + \pi_1 = \pi_0 + cV_{\bth}$. 

Consider the rational $1$-form $\alpha = -a^{-1}x_k^{-1} dx_k$ on $\mX$ and $u_\bth = 
a(\partial_j-\partial_k) \in \mfX^1(\mX)$. Applying the contraction operator
$\iota_\alpha$ to both sides of $u_\bth \wedge (cV_\bth) = 0$ and by \autoref{lm:dpi-0}, one gets
\[
0 = \iota_\alpha(u_\bth \wedge (cV_\bth)) = cV_\bth - u_\bth \wedge \iota_\alpha (cV_\bth)=
cV_\bth -[\pi_0, \; \iota_\alpha (cV_\bth)].
\]
One thus has $cV_\bth = [v, \, \pi_0]$, where 
\[
v = -\iota_\alpha (cV_\bth) 
= -\frac{c}{a} \left(\frac{1}{x_k} \prod_{i \neq j, k} x_i^{\bth_i}\right)\frac{\partial}{\partial x_j},
\]
regarded as a holomorphic  vector field on the Zariski open subset $U$ of 
$\mX$ defined by $x_k \neq 0$. Note that $v$ has $(\CC^\times)^n$-weight $\bth$ and $[v, V_\bth] = 0$
on $U$. Let $\varphi$ be the time $1$-flow of $v$ on $U$, so 
$\varphi^* x_i = x_i$ for $i \neq j$, and 
\[
\varphi^* x_j = x_j-\frac{c}{a} \left(\frac{1}{x_k} \prod_{i \neq j, k} x_i^{\bth_i}\right).
\]
By an analog of  \autoref{lm:t-flow} for rational vector fields, $\pi(c) = \exp(v) \pi_0 = (\varphi^{-1})_* \pi_0$ on $U$. In particular, $\pi(c)$ also has rank $2r$.
Moreover,
\[
\pi(c)^r \wedge  \xi_1^\flat \wedge \cdots \wedge \xi_{n-2r}^\flat =  \left(x_jx_k -\frac{c}{a} \prod_{i\neq j, k} x_i^{\bth_i}\right)
 \left(\prod_{i\neq j, k} x_i\right)\frac{\partial}{\partial x_1} \wedge \cdots \wedge 
\frac{\partial}{\partial x_n}.
\]
\end{proof}

\begin{remark}\label{rmk:smoothable}
{\rm 
In the setting of \autoref{lm:Pf-bth}, set
$\phi_i = x_i$ for $i \neq k$ and
\[
\phi_k = x_jx_k -\frac{c}{a}\prod_{i\neq j, k} x_i^{\bth_i}
 \in \CC[x_1, \ldots, x_n].
\]
It follows from $\bbeta \bth =0$ that $\phi_k$ is a $\TT$-weight vector with $\TT$-weight $\beta_j + \beta_k$. Note also that
$\phi_1, \ldots, \phi_n$ are algebraically independent. 
A direct computation using \eqref{eq:bra-f-g} and the fact that $\blambda \bth =a(e_j-e_k)$
shows that $\phi_1, \ldots, \phi_n$ have pair-wise log-canonical Poisson brackets with respect to $\pi(c)$. More
precisely, 
\[
\{\phi_i, \, \phi_k\} = \begin{cases} (\lambda_{i,j} + \lambda_{i, k})\phi_i\phi_k, &\hs i \in [1, n]\backslash\{j, k\},\\
\lambda_{j, k} \phi_j \phi_k, & \hs i \in \{j, k\}.\end{cases}
\]
Note that 
during the Poisson deformation  
from $\pi_0$ to $\pi(c)$ the zero divisor of ${\rm Pf}_\TT$
undergoes the transformation of having two irreducible components $\{x_j=0\}$ and $\{x_k=0\}$ merge to a single irreducible component given  by $\{\phi_k = 0\}$. As a result, the normal crossings singularity of the divisor at a generic point of the co-dimension two stratum $\{x_j=x_k=0\}$ gets smoothed out. This explains the term
{\it smoothable weights} for elements in $\calS(\pi_0)$.
\hfill $\diamond$
}
\end{remark}

\section{Action data and classification of symmetric Poisson CGL extensions}\label{s:action-data}

For a given complex algebraic torus $\TT$, we present in 
$\S$\ref{ss:action-data} our primary examples of $\TT$-log-symplectic log-canonical Poisson
structures $\pi_0$ on $\mX = \CC^n$, namely those defined by the so-called {\it $n$-dimensional
$\TT$-action data}. For such a $\pi_0$, we show that $\calS(\pi_0)$ is negatively bordered 
on the right, and when the $\TT$-action data is {\it symmetric}, $\calS(\pi_0)$ is negatively bordered on both sides. Our main result in $\S$\ref{s:action-data} is \autoref{thm:sym-CGL}, which gives a classification,
up to coordinate rescalings, of symmetric Poisson CGL extensions as defined by K. 
Goodearl and M. Yakimov in 
\cite{GY:Poi-CGL} 
in terms of their log-canonical part $\pi_0$ and subsets of $\calS(\pi_0)$.

\subsection{Log-canonical Poisson structures of action type}\label{ss:action-data}
Let again $\TT$ be a complex  algebraic torus with Lie algebra $\t$ and character lattice $X^*(\TT) \subset \t^*$. 

\begin{definition}\label{def:action-data-0} 
1) By an $n$-dimensional $\TT$-action datum we mean a pair $(\bbeta, \bfh)$, where 
\[
\bbeta = (\beta_1, \ldots, \beta_n) \in X^*(\TT)^n \hs \mbox{and} \hs 
{\bf h}=(h_1, \ldots, h_n)  \in \t^n,
\]
such that $\beta_j(h_j) \neq 0$ for every $j \in [1, n]$. Given such a pair $(\bbeta, \bfh)$, let $\TT$ act on $\CC^n$ via 
\begin{equation}\label{eq:T-CCsn-1}
t\cdot x_j = t^{\beta_j} x_j, \hs j \in [1, n],
\end{equation}
and define the log-canonical Poisson structure $\pi_0$ on $\CC^n$ by
\begin{equation}\label{eq:pi0-action}
\pi_0 =- \sum_{j < k} \beta_j(h_k) \partial_j \wedge \partial_k
=-\sum_{j < k} \beta_j(h_k) x_jx_k \frac{\partial}{\partial x_j} \wedge \frac{\partial}{\partial x_j}.
\end{equation}

2) A log-canonical Poisson structure $\pi_0$ on $\CC^n$  is said to be of {\it $\TT$-action type}
(in the ordered linear coordinate system $(x_1, \ldots, x_n)$)
if it is defined  as in 
\eqref{eq:pi0-action} by some $n$-dimensional $\TT$-action datum $(\bbeta, \bfh)$.
\end{definition}

\begin{remark}\label{rk:hk}
{\rm
Note that while the $\TT$-action on $\mX = \CC^n$ is determined by $\bbeta = (\beta_1, \ldots, \beta_n)$,
the Poisson structure $\pi_0$ does not change if ${\bf h}$ is changed to
$(h_1 + h_1^\prime, \ldots, h_n + h_n^\prime)$ where $\beta_j(h_k^\prime) = 0$ for $j, k \in [1, n]$. 
\hfill $\diamond$
}
\end{remark}

Let $\pi_0$ be a log-canonical Poisson structure on $\CC^n$ of $\TT$-action type, given as in
\eqref{eq:pi0-action}.
We now show that 
$\pi_0$  is $\TT$-log-symplectic. Note that  
the Poisson coefficient matrix of $\pi_0$ is 
\begin{equation}\label{eq:blam}
\blam= \left(\displaystyle \begin{array}{cccc} 0 & -\beta_1(h_2) &  \cdots & -\beta_1(h_n)\\
\beta_1(h_2) & 0  & \cdots & -\beta_2(h_n)\\
\cdots & \cdots & \cdots& \cdots \\
\beta_1(h_n) & \beta_2(h_n)  & \cdots & 0\end{array}\right).
\end{equation}
Set $\lambda_j = \beta_j(h_j) \neq 0$ for $j \in [1, n]$, and  $\nu_{j, k} = \beta_j(h_k) + \beta_k(h_j)$
for $1 \leq j < k \leq n$, and set
\begin{equation}\label{eq:bnu}
\bnu = \left(\begin{array}{cccc} \lambda_1 & \nu_{1,2}  & \cdots &\nu_{1, n}\\
0 & \lambda_2  & \cdots & \nu_{2, n}\\
\cdots  & \cdots & \cdots &\cdots\\
0 & 0  & \cdots & \lambda_n\end{array}\right) \hs \mbox{and} \hs
\bsigma = \left(\displaystyle \begin{array}{cccc} \beta_1(h_1) & \beta_2(h_1) &  \cdots & \beta_n(h_1)\\
\beta_1(h_2) & \beta_2(h_2)  & \cdots & \beta_n(h_2)\\
\cdots & \cdots & \cdots& \cdots \\
\beta_1(h_n) & \beta_2(h_n)  & \cdots & \beta_n(h_n)\end{array}\right),   
\end{equation}
so that  $\bsigma = \bnu +\blam$. We first prove an elementary linear algebra fact.

\begin{lemma}\label{lm:linear-alg-0}
With $\blam$ given in \eqref{eq:blam}, for any $g \in \CC^n$ the linear system 
\begin{equation}\label{eq:blam-chi-f}
\blam f = -g \hand \bbeta f = 0
\end{equation}
has a solution $f \in \CC^n$ if and only if $\bbeta \bnu^{-1}g=0$, and in this case
the solution $f \in \CC^n$ to \eqref{eq:blam-chi-f} is unique and  is given by
$f = \bnu^{-1}g$.
\end{lemma}

\begin{proof} 
Suppose that $f \in \CC^n$ satisfies
$\bbeta f = 0$. Evaluating both sides of $\bbeta f = 0$ at 
$h_j$ for every $j \in [1, n]$ gives $\bsigma f = 0$. As $\bnu = \bsigma - \blam$, 
the equations in \eqref{eq:blam-chi-f} are now equivalent to 
\[
\bnu f = g \hand \bbeta f = 0,
\]
from which \autoref{lm:linear-alg-0} follows.
\end{proof}

\begin{corollary}\label{cor:action-pi0-log-sym}
Every log-canonical Poisson structure $\pi_0$ on $\mX = \CC^n$ 
of $\TT$-action type, as given in 
\eqref{eq:pi0-action}, is  
$\TT$-log-symplectic for the $\TT$-action on $\mX$ defined  in \eqref{eq:T-CCsn-1}.
\end{corollary}

\begin{proof} The $\TT$-log-symplecticity of $\pi_0$ follows from \autoref{lm:linear-alg-0} by 
setting $g = 0$ in \eqref{eq:blam-chi-f}.
\end{proof}

Recall that $\calW_2 = \{\bfw =(\bfw_1, \ldots, \bfw_n)^t\in \ZZone: |J_\bfw|\leq 2\}$, where
$J_\bfw = \{j \in [1, n]: \bfw_j = -1\}$. For a $\TT$-log-symplectic log-canonical Poisson structure $\pi_0$ 
on $\mX = \CC^n$, recall that 
$\calS(\pi_0)\subset \calW_2$ is the set of all  non-zero $\CCsn$-weights in 
$\mH^2_{\pi_0}(\mX)^\TT$, and  that $|J_\bth| = 2$ for every $\bth \in \calS(\pi_0)$. 
Recall from  \autoref{de:Delta-n}  the edge set $\calE^+(\pi_0)$ of  the oriented smoothing graph of $\pi_0$, i.e., 
\begin{equation}\label{eq:calE-plus}
\calE^+(\pi_0) = \{(j, k): 1 \leq j < k \leq n  \;\mbox{and}\;J_\bth = \{j, k\} \; 
\mbox{for some (necessarily unique)} \; \bth \in 
\calS(\pi_0)\},
\end{equation}
and recall  the bijection 
$\calE^+(\pi_0) \rightarrow \calS(\pi_0), \, (j, k) \mapsto \bth^{(j, k)}$, 
where $J_{\bth^{(j, k)}} = \{j, k\}$ for each $(j, k) \in \calE^+(\pi_0)$.

For $\pi_0$ of $\TT$-action type, we now explain how $\calE^+(\pi_0)$ and $\calS(\pi_0)$ can be explicitly computed. 

\begin{proposition}\label{pr:Spi0-action}
For the log-canonical Poisson structure $\pi_0$ on $\mX = \CC^n$ 
defined by a $\TT$-action datum $(\bbeta, \bfh)$ as 
in \eqref{eq:pi0-action}, the set 
$\calE^+(\pi_0)$ consists of all pairs $(j, k)$ of 
$1\leq j < k \leq n$ such that 
\begin{equation}\label{eq:beta-nu-j-k}
\bbeta \bnu^{-1}e_j = \bbeta \bnu^{-1}e_k \hs \mbox{and} \hs 
\lambda_k (\bnu^{-1}e_j -\bnu^{-1}e_k) \in \calW_2. 
\end{equation}
Moreover, for all $(j, k) \in \calE^+(\pi_0)$ one has
$\bth^{(j, k)} =\lambda_k (\bnu^{-1}e_j -\bnu^{-1}e_k)$, 
$\blam \bth^{(j, k)} = -\lambda_k(e_j-e_k)$, and 
\begin{equation}\label{eq:bth-jk-0}
\bth^{(j, k)} = 
(\bth^{(j, k)}_1, \;\ldots,\; \bth^{(j, k)}_{j-1},\; -1, \;\bth^{(j, k)}_{j+1},\; \ldots,\;
\bth^{(j, k)}_{k-1}, \; -1,\; 0, \;\ldots, \;0)^t,
\end{equation}
with $\bth^{(j, k)}_i \in \ZZ_{\geq 0}$ for all $i \notin \{j, k\}$.
In particular, $\calS(\pi_0)$ is negatively bordered on the right. 
\end{proposition}

\begin{proof} 
By \eqref{eq:calS-pi0}, the set  $\calS(\pi_0)$ consists of all $\bth \in \calW_2$ with exactly 
two $-1$ entries at some $j, k \in [1, n]$ with $j < k$  and satisfy 
\begin{equation}\label{eq:btheta-0}
\blam \btheta = a (e_j-e_{k}) \hs \mbox{and} \hs \beta \btheta = 0
\end{equation}
for some $a \in \CC^\times$. By \autoref{lm:linear-alg-0}, \autoref{eq:btheta-0} has a solution 
$\bth \in\CC^n$ for given $a \in \CC^\times$
and $1\leq j < k \leq n$ if and only if 
$\bbeta \bnu^{-1}e_j = \bbeta \bnu^{-1}e_k$, and in such a case the unique solution of
\eqref{eq:btheta-0} in $\CC^n$ is given by 
\begin{equation}\label{eq:bth-jk-00}
\bth^{(j, k)}  = -a(\bnu^{-1}e_j - \bnu^{-1}e_k).
\end{equation}
Assume that the pair $(j, k)$ with $1 \leq j < k \leq n$ satisfy  $\bbeta \bnu^{-1}e_j = \bbeta \bnu^{-1}e_k$,
and write  $\bth^{(j, k)}$ in \eqref{eq:bth-jk-00} as 
$\bth^{(j, k)} = (\bth^{(j, k)}_1, \ldots, \bth^{(j, k)}_n)^t$. 
By \autoref{lm:wj-wk}, the identity 
$\blam \btheta^{(j, k)} = a (e_j-e_{k})$ implies that 
$\bth^{(j, k)}_j = \bth^{(j, k)}_k$. 
As $\bnu^{-1}$ is upper-triangular with $k^{\rm th}$-diagonal entry $\lambda_k^{-1}$, 
we have $\bth^{(j, k)}_k = a \lambda_k^{-1}$ by 
\eqref{eq:bth-jk-00}. Thus 
$\bth^{(j, k)}_j = \bth^{(j, k)}_k = -1$  if and only if $a = -\lambda_k$. We hence conclude that 
for $1 \leq j < k \leq n$, 
\eqref{eq:btheta-0} has a solution $\bth =(\bth_1, \ldots, \bth_n)^t \in \calW_2$ with $\bth_j=\bth_k=-1$ if and only if $a = 
-\lambda_k$ and \eqref{eq:beta-nu-j-k} holds.
Moreover, for $(j, k)$ satisfying \eqref{eq:beta-nu-j-k}, it again follows from 
$\bnu^{-1}$ being upper-triangular that $\bth^{(j, k)}$ is given as in \eqref{eq:bth-jk-0}.
In particular,
every $\bth^{(j, k)} \in \calS(\pi_0)$ is negatively bordered on the right. 
\end{proof}

\begin{example}\label{ex:111}
{\rm
Let $\TT=(\CC^\times)^2$, acting on $\CC^4$ via 
$\beta_1 = (-1,0)^t$, $\beta_2=(0,1)^t$, $\beta_3=\beta_4=(-1, -1)^t$, and
$h_1 = (-1, 0), h_2 = (1, -1), h_3 = (-1, 0)$ and $h_4 = (0, -1)$.
The log-canonical Poisson structure $\pi_0$ defined by $(\bbeta, \bfh)$ is then the one given in 
\autoref{ex:11}. Computing $\calS(\pi_0)$ using \autoref{pr:Spi0-action}, one gets the results in \autoref{ex:11}. 
\hfill $\diamond$
}
\end{example}

\begin{remark}\label{rk:new-class}
{\rm
By \autoref{pr:Spi0-action},  \autoref{thm:4.8-W0} applies to every $\pi_0$ of $\TT$-action type. 
We have thus constructed   a large class of
$\TT$-log-symplectic log-canonical Poisson structures on $\mX = \CC^n$
which admit families of $\TT$-invariant algebraic Poisson deformations into
iterated Poisson-Ore extensions. Examples  include the restriction of the standard Poisson structure $\pi$ on $Z$ to any of the $2^n$ affine coordinate charts of $Z$, where 
$Z$ is any $n$-dimensional Bott-Samelson
variety of a complex semi-simple Lie group \cite[Theorem 5.12]{EL:BS}). We will 
treat these examples in more detail elsewhere.
\hfill $\diamond$
}
\end{remark}

\begin{example}\label{ex:11111}
{\rm
Let $\TT = \CC^\times$ and we identify $X^*(\TT) \cong \ZZ$ and $\t \cong \CC$.
Consider the $5$-dimensional $\TT$-action datum $(\bbeta, {\bf h})$ with
$\bbeta = (1, 1, 1, 1, 1)$ and ${\bf h} = (1, 1, 1, -1, -1)$. We then have 
\[
\bnu = \left(\begin{array}{ccccc} 1 & 2 & 2 & 0 & 0\\ 0 & 1 & 2 & 0 & 0\\
0 & 0 & 1 & 0 & 0\\
0 & 0 & 0 & -1 & -2\\
0 & 0 & 0 & 0 & -1\end{array}\right) \hs \mbox{and} \hs
\bnu^{-1} = \left(\begin{array}{ccccc} 1 & -2 & 2 & 0 & 0\\ 0 & 1 & -2 & 0 & 0\\
0 & 0 & 1 & 0 & 0\\
0 & 0 & 0 & -1 & 2\\
0 & 0 & 0 & 0 & -1\end{array}\right).
\]
By \autoref{pr:Spi0-action}, we have $\calE^+(\pi_0) = \{(1, 3), (1, 5), (2, 4)\}$, and 
$\calS(\pi_0) = \{\bth^{(1, 3)}, \bth^{(1, 5)}, \bth^{(2, 4)}\}$ with
\[
\bth^{(1, 3)} = (-1, 2, -1, 0, 0)^t, \hs \bth^{(1, 5)} = (-1, 0, 0, 2, -1)^t, \hs
\bth^{(2, 4)} = (2, -1, 0 , -1, 0)^t.
\]
(see \autoref{fig:SmDiag11111}).

\begin{figure}[h]
    \centering
\pgfdeclarelayer{background layer}
\pgfdeclarelayer{foreground layer}
\pgfsetlayers{background layer,main,foreground layer}
\begin{tikzpicture}[baseline=-1ex,
rotate=360/2/5,
scale=0.9,line join = round 
    ] 
    
    \begin{pgfonlayer}{main}
\foreach \n in {1,...,5}
{
    \coordinate (v\n) at ({-90+((\n-1)*360/5)}:1.5);
}
\node[below right] at (v1) {$1$};
\node[right] at (v2) {$2$};
\node[above] at (v3) {$3$};
\node[left] at (v4) {$4$};
\node[below left] at (v5) {$5$};

\foreach \n in {1,...,5}
{
	\foreach \m in {1,...,5}
	\draw (v\n) -- (v\m);
}
\end{pgfonlayer}

\draw[thick,red,-] (v2) ++({0.4*cos(108)},{0.4*sin(108)})  arc[start angle=108,end angle=216,radius=0.4];
\draw[thick,red,-] (v2) ++({0.55*cos(108)},{0.55*sin(108)})  arc[start angle=108,end angle=216,radius=0.55];
\draw[thick,red,-] (v4) ++({0.4*cos(252)},{0.4*sin(252)})  arc[start angle=252,end angle=286,radius=0.4];
\draw[thick,red,-] (v4) ++({0.55*cos(252)},{0.55*sin(252)})  arc[start angle=252,end angle=286,radius=0.55];
\draw[thick,red,-] (v1) ++({0.4*cos(36)},{0.4*sin(36)})  arc[start angle=36,end angle=108,radius=0.4];
\draw[thick,red,-] (v1) ++({0.55*cos(36)},{0.55*sin(36)})  arc[start angle=36,end angle=108,radius=0.55];

\draw[very thick, blue] (v1)--(v3);
\draw[very thick, blue] (v1)--(v5);
\draw[very thick, blue] (v2)--(v4);

\begin{pgfonlayer}{foreground layer}
\foreach \n in {1,...,5}
{
	\draw[fill] (v\n) circle (0.03);
}
\end{pgfonlayer}
 \end{tikzpicture}
    \caption{Smoothing diagram appearing in \autoref{ex:11111}}
    \label{fig:SmDiag11111}
\end{figure}
\hfill $\diamond$
}
\end{example}

Recall from \autoref{defn:n-b} that for an element $\bfw \in \ZZone$ that is negatively bordered on the right,
we have the right border ${\rm rb}(\bfw)\in [2, n]$ of $\bfw$. 

\begin{lemma}\label{lm:j-unique}
Let $\pi_0$ be any log-canonical Poisson structure on $\CC^n$ of $\TT$-action type. Then 
for each $k \in [2, n]$ there is
at most one $j \in [1, k-1]$ such that $(j, k) \in \calE^+(\pi_0)$, and we have the bijective map
\begin{equation}\label{eq:rb-S}
{\rm rb}:\;\; \calS(\pi_0) \longrightarrow {\rm rb}(\calS(\pi_0)), \;\; \bth \longmapsto {\rm rb}(\bth).
\end{equation}
\end{lemma}

\begin{proof}
Let $k \in [2, n]$ and suppose that there are $1\leq j < j' <k$ such that
both $(j, k)$ and $(j', k)$ are in $\calE^+(\pi_0)$. By  \autoref{pr:Spi0-action}, 
$\blam \bth^{(j, k)} = -\lambda_k (e_j-e_k)$ and 
$\blam \bth^{(j', k)} =- \lambda_k (e_{j'}-e_k)$, so
\[
\blam (\bth^{(j, k)}-\bth^{(j', k)}) = -\lambda_k (e_j-e_{j'}).
\]
By 1) of \autoref{lm:wj-wk},  $\bth^{(j, k)}_j-\bth^{(j', k)}_j = \bth^{(j, k)}_{j'}-\bth^{(j', k)}_{j'}$, 
from which we have
\[
\bth^{(j, k)}_{j'} = \bth^{(j, k)}_j+\bth^{(j', k)}_{j'}-\bth^{(j', k)}_j = -2-\bth^{(j', k)}_j\leq -2,
\]
contradicting 
the fact that $\bth^{(j, k)}_{j'} \geq -1$. the 
 map ${\rm rb}: \calS(\pi_0) \rightarrow {\rm rb}(\calS(\pi_0))$ is thus injectve by \autoref{lm:ij-unique}. 
Being surjective by definition,  the map ${\rm rb}: \calS(\pi_0) \rightarrow {\rm rb}(\calS(\pi_0))$ is thus 
bijective. 
\end{proof}

\begin{remark}\label{rk:no-sink}
{\rm 
The bijectivity of the map $\rb$ in \eqref{eq:rb-S}  is equivalent to 
the oriented smoothing graph $\Gamma^+(\pi_0)$ of $\pi_0$ (see \autoref{de:Delta-n}) 
having no bivalent sinks. 
\hfill $\diamond$
}
\end{remark}

\subsection{Log-canonical Poisson structures of symmetric action type}\label{ss:symmetric-type}
Let again $\TT$ be an algebraic torus over $\CC$ with Lie algebra $\t$, and we continue with the notation from 
$\S$\ref{ss:action-data}.

\begin{definition}\label{def:sym-action-data}
{\rm
1) An $n$-dimensional action datum $(\bbeta = (\beta_1, \ldots, \beta_n), \bfh = (h_1, \ldots, h_n))$
is said to be {\it symmetric} if there exists $ (h_1^\prime, \ldots, h_n^\prime) \in \t^n$
such that $\beta_j(h_j^\prime) \neq 0$ for every $j \in [1, n]$ and 
\begin{equation}\label{eq:beta_jk}
\beta_j(h_k) = -\beta_k(h_j^\prime), \hs 1 \leq j < k \leq n.
\end{equation}

2) A log-canonical  Poisson structure $\pi_0$ on $\mX = \CC^n$ 
is said to be of {\it symmetric $\TT$-action type} if it is defined 
as in \eqref{eq:pi0-action} by a symmetric $\TT$-action datum $(\bbeta, {\bf h})$. 
}
\end{definition}

\begin{remark}\label{rmk:symmetric}
{\rm
The symmetric condition on the action datum $(\bbeta, \bfh)$ says that the 
 Poisson structure $\pi_0$ on $\mX =\CC^n$ defined by $(\bbeta, \bfh)$ 
 also takes the form
 \[
\{x_k, x_j\}_{\pi_0} = -\beta_k(h_j^\prime)x_jx_k, \hs 1 \leq j < k \leq n,
\]
so $\pi_0$ is also defined by the action datum
$(\bbeta^\prime = (\beta_n, \ldots, \beta_2, \beta_1), \bfh^\prime = (h_n, \ldots, h_2^\prime, h_1^\prime)$
in the reversely ordered linear coordinate system $(x_n, \ldots, x_2, x_1)$. 
\hfill $\diamond$
}
\end{remark}

\begin{example}\label{ex:00}
{\rm
Let $\TT = \CC^\times$ and consider the $n$-dimensional action datum $(\bbeta, \bfh)$,
where $\bbeta = (1, 1, \ldots, 1)$ and $\bfh = (a_1, a_2, \ldots, a_n)$ with $a_j \in \CC^\times$ for each $j \in [1, n]$. One checks directly from the definition
that $(\bbeta, \bfh)$ is symmetric if and only if $a_2=a_3=\cdots = a_n$.
\hfill $\diamond$
}
\end{example}

For a log-canonical Poisson structure $\pi_0$ of symmetric $\TT$-action type,
defined by a symmetric $\TT$-action datum $(\bbeta, \bfh)$ as in  \autoref{def:sym-action-data},
recall that we have set $\lambda_i = \beta_i(h_i)$ for $i \in [1, n]$. Set also 
\[
\lambda_i^\prime = \beta_i(h_i^\prime), \hs i \in [1, n].
\]
Recall from \autoref{pr:Spi0-action} the description of the set 
$\calE^+(\pi_0)$, the bijection 
$\calE^+(\pi_0) \rightarrow  \calS(\pi_0):  (j, k) \mapsto \bth^{(j, k)}$, and the fact that $\blam 
\bth^{(j, k)} = -\lambda_k (e_j-e_k)$ for $(j, k) \in \calE^+(\pi_0)$.

\begin{proposition}\label{pr:Spi0-action-symmetric}  For a log-canonical Poisson
structure $\pi_0$ on $\mX = \CC^n$ of symmetric action type, and 
for every $(j, k) \in \calE^+(\pi_0)$, one has $\lambda_j^\prime = -\lambda_k$ and
\begin{equation}\label{eq:bth-jk-1}
\bth^{(j, k)} = 
(0, \; \ldots, \; 0,\; -1, \;\bth^{(j, k)}_{j+1},\; \ldots,\;
\bth^{(j, k)}_{k-1}, \; -1,\; 0, \;\ldots, \;0)^t,
\end{equation}
where $\bth^{(j, k)}_i \in \ZZ_{\geq 0}$ for each $i \in [j+1, k-1]$.
In particular, $\calS(\pi_0)$ is negatively bordered on both sides, so every subset $\calS$ of $\calS(\pi_0)$ has
Property \eqref{eq:W1}.
\end{proposition}

\begin{proof} As $\pi_0$ is defined by the $\TT$-action datum 
$(\bbeta =(\beta_1, \ldots, \beta_n), \bfh = (h_1, \ldots, h_n))$ 
in the ordered linear coordinate system $(x_1, \ldots, x_n)$, as well as by the  
$\TT$-action datum $(\bbeta^\prime = (\beta_n, \ldots, \beta_1), \bfh^\prime = (h_n, \ldots, h_1^\prime)$
in the ordered linear coordinate system $(x_n, \ldots, x_1)$,  it follows from 
\autoref{pr:Spi0-action} that $\bth^{(j, k)}$, for every $(j, k) \in \calE^+(\pi_0)$, is 
negatively bordered on both sides as in \eqref{eq:bth-jk-1} and that
\[
\blam \bth^{(j, k)} = -\lambda_k(e_j-e_k) = -\lambda^\prime_j (e_k-e_j),
\]
from which one sees that $\lambda_j^\prime = -\lambda_k$.
\end{proof}

\begin{remark}\label{rk:lambda-jkl}
{\rm
In the setting of \autoref{pr:Spi0-action-symmetric}, if $j, k, l \in [1,n]$ are such that
$(j, k), (k, l) \in \calE^+(\pi_0)$, then, as $\bth^{(j, k)}_l = \bth^{(k, l)}_j = 0$,
it follows from
\autoref{lm:rational}  that 
$\lambda_k = \lambda_l = -\lambda_k^\prime = -\lambda_j^\prime$.
\hfill $\diamond$
}
\end{remark}

\begin{example}\label{ex:1111}
{\rm
The $4$-dimensional $(\CC^\times)^2$-action datum in \autoref{ex:111} is not symmetric, as 
$\bth^{(2, 3)} = (1, -1, -1, 0)^t \in \calS(\pi_0)$ is not negatively bounded on both sides.
\hfill $\diamond$
}
\end{example}

\begin{example}\label{ex:beta-2p-00}
{\rm 
Let $p \in \ZZ\backslash \{0\}$ and consider the log-canonical Poisson structure on $\CC^4$ given by
\begin{align}\nonumber
\pi_0 =&-p x_1x_2\frac{\partial}{\partial x_1} \wedge \frac{\partial}{\partial x_2} + p
x_1x_3\frac{\partial}{\partial x_1}\! \wedge \!\frac{\partial}{\partial x_3}  +
p^2 x_1x_4\frac{\partial}{\partial x_1}\! \wedge \!\frac{\partial}{\partial x_4} +
x_2x_3\frac{\partial}{\partial x_2}\! \wedge \!\frac{\partial}{\partial x_3}\\
\label{eq:pi0-p}
&+ p
x_2x_4\frac{\partial}{\partial x_2}\! \wedge\! \frac{\partial}{\partial x_4} -
p x_3x_4\frac{\partial}{\partial x_3}\! \wedge \!\frac{\partial}{\partial x_4}.
\end{align}
Then $\pi_0$ is of symmetric $\TT$-action type for any non-trivial $\TT$. Indeed, choose any non-zero 
$\beta \in X^*(\TT)$ and $h \in \t$ such that $\beta(h) = 1$, and let
$\bbeta = (p\beta, \beta, -\beta, -p\beta)$ and ${\bf h} = (ph, h, -h, -ph)$. Then $\pi_0$ is defined by the symmetric 
$\TT$-action datum $(\bbeta, {\bf h})$. Computing $\calS(\pi_0)$ using 
\autoref{pr:Spi0-action}, we have the following.

\begin{figure}[h]
    \centering
\pgfdeclarelayer{background layer}
\pgfdeclarelayer{foreground layer}
\pgfsetlayers{background layer,main,foreground layer}
\begin{subfigure}{0.3\textwidth}
\centering
\begin{tikzpicture}[baseline=-1ex,
rotate=360/2/4,
scale=0.9,line join = round 
    ] 
    
    \begin{pgfonlayer}{main}
\foreach \n in {1,...,4}
{
    \coordinate (v\n) at ({-90+((\n-1)*360/4)}:1.5);
}
\node[below right] at (v1) {$1$};
\node[above right] at (v2) {$2$};
\node[above left] at (v3) {$3$};
\node[below left] at (v4) {$4$};

\foreach \n in {1,...,4}
{
	\foreach \m in {1,...,4}
	\draw (v\n) -- (v\m);
}
\end{pgfonlayer}

\draw[very thick, blue] (v1)--(v2);
\draw[very thick, blue] (v2)--(v3);
\draw[very thick, blue] (v3)--(v4);

\begin{pgfonlayer}{foreground layer}
\foreach \n in {1,...,4}
{
	\draw[fill] (v\n) circle (0.03);
}
\end{pgfonlayer}
 \end{tikzpicture}
\caption{Case 1, $p=-1$.}
\end{subfigure} 
\hfill 
\begin{subfigure}{0.3\textwidth}
\centering
\begin{tikzpicture}[baseline=-1ex,
rotate=360/2/4,
scale=0.9,line join = round 
    ] 
    
    \begin{pgfonlayer}{main}
\foreach \n in {1,...,4}
{
    \coordinate (v\n) at ({-90+((\n-1)*360/4)}:1.5);
}
\node[below right] at (v1) {$1$};
\node[above right] at (v2) {$2$};
\node[above left] at (v3) {$3$};
\node[below left] at (v4) {$4$};

\foreach \n in {1,...,4}
{
	\foreach \m in {1,...,4}
	\draw (v\n) -- (v\m);
}
\end{pgfonlayer}

\draw[very thick, blue] (v2)--(v3);

\begin{pgfonlayer}{foreground layer}
\foreach \n in {1,...,4}
{
	\draw[fill] (v\n) circle (0.03);
}
\end{pgfonlayer}
 \end{tikzpicture}
\caption{Case 2, $p<0$ and $p\not=-1$.}
\end{subfigure} 
\hfill 
\begin{subfigure}{0.3\textwidth}
\centering
\begin{tikzpicture}[baseline=-1ex,
rotate=360/2/4,
scale=0.9,line join = round 
    ] 
    
    \begin{pgfonlayer}{main}
\foreach \n in {1,...,4}
{
    \coordinate (v\n) at ({-90+((\n-1)*360/4)}:1.5);
}
\node[below right] at (v1) {$1$};
\node[above right] at (v2) {$2$};
\node[above left] at (v3) {$3$};
\node[below left] at (v4) {$4$};

\foreach \n in {1,...,4}
{
	\foreach \m in {1,...,4}
	\draw (v\n) -- (v\m);
}
\end{pgfonlayer}

\draw[thick,red,-] (v2) ++({0.3*cos(180)},{0.3*sin(180)})  arc[start angle=180,end angle=225,radius=0.3];
\draw[fill,red] (v2) ++ ({0.4*cos(202)},{0.4*sin(202)}) circle (0.01);
\draw[fill,red] (v2) ++ ({0.5*cos(202)},{0.5*sin(202)}) circle (0.01);
\draw[fill,red] (v2) ++ ({0.6*cos(202)},{0.6*sin(202)}) circle (0.01);
\draw[thick,red,-] (v2) ++({0.7*cos(180)},{0.7*sin(180)})  arc[start angle=180,end angle=225,radius=0.7];
\draw[thick,red,-] (v3) ++({0.3*cos(225)},{0.3*sin(225)})  arc[start angle=225,end angle=270,radius=0.3];
\draw[fill,red] (v3) ++ ({0.4*cos(248)},{0.4*sin(248)}) circle (0.01);
\draw[fill,red] (v3) ++ ({0.5*cos(248)},{0.5*sin(248)}) circle (0.01);
\draw[fill,red] (v3) ++ ({0.6*cos(248)},{0.6*sin(248)}) circle (0.01);
\draw[thick,red,-] (v3) ++({0.7*cos(225)},{0.7*sin(225)})  arc[start angle=225,end angle=270,radius=0.7];

\coordinate (u1) at ($(v2) +({0.25*cos(235)},{0.25*sin(235)})$);
\coordinate (u2) at  ($(v2) +({0.75*cos(230)},{0.75*sin(230)})$);

\draw[decorate,decoration={brace,amplitude=5pt}] (u1) -- (u2) node[midway, right, xshift = 4pt]{\textcolor{red}{$\scriptstyle 2p$}};

\coordinate (u3) at ($(v3) +({0.25*cos(215)},{0.25*sin(215)})$);
\coordinate (u4) at  ($(v3) +({0.75*cos(220)},{0.75*sin(220)})$);

\draw[decorate,decoration={brace,amplitude=5pt}] (u4) -- (u3) node[midway, left, xshift = -4pt]{\textcolor{red}{$\scriptstyle 2p$}};

\draw[very thick, blue] (v1)--(v4);
\draw[very thick, blue] (v2)--(v3);

\begin{pgfonlayer}{foreground layer}
\foreach \n in {1,...,4}
{
	\draw[fill] (v\n) circle (0.03);
}
\end{pgfonlayer}
 \end{tikzpicture}
 \caption{Case 3, $p>0$.}
 \end{subfigure} 
    \caption{Smoothing diagrams appearing in \autoref{ex:beta-2p-00}}
    \label{fig:SmDiagbeta-2p-00}
\end{figure}

{\bf Case 1,} $p =-1$: Then $\calE^+(\pi_0) = \{(1, 2), (2, 3), (3, 4)\}$ and 
$\calS(\pi_0) = \{\bth^{(1, 2)}, \bth^{(2, 3)}, \bth^{(3, 4)}\}$, where
\[
\bth^{(1, 2)} = (-1, -1,0, 0)^t, \hs  \bth^{(2, 3)} = (0, -1, -1, 0)^t,\hs  \bth^{(3, 4)}=(0, 0, -1, -1)^t;
\]

{\bf Case 2,} $p < 0$ and $p \neq -1$. Then $\calE^+(\pi_0) = \{(2, 3)\}$ and  
$\calS(\pi_0) = \{\bth^{(2, 3)}\}$, where
\[
\bth^{(2, 3)} = (0, -1, -1, 0)^t;
\]

{\bf Case 3,} $p > 0$: Then $\calE^+(\pi_0) = \{(1, 4), (2, 3)\}$ and 
$\calS(\pi_0) = \{\bth^{(1,4)}, \bth^{(2,3)}\}$, where
\[
\bth^{(1,4)} = (-1, 2p, 2p, -1)^t, \hs \bth^{(2,3)} = (0, -1, -1, 0)^t.
\]

This example will be continued in \autoref{ex:beta-2p-01}. 
\hfill $\diamond$
}
\end{example}

\begin{lemma}\label{lm:j-k-unique}
Let $\pi_0$ be a log-canonical Poisson structure on $\mX = \CC^n$ of symmetric $\TT$-action type. 

1)  For each $k \in [2, n]$ there is
at most one $j \in [1, k-1]$ such that $(j, k) \in \calE^+(\pi_0)$;

2) For each $j \in [1, n-1]$ there is at most one
$k \in [j+1, n]$ such that $(j, k) \in \calE^+(\pi_0)$. 

\noindent
In other words, the oriented smoothing graph $\Gamma^+(\pi_0)$ of $\pi_0$
has neither bivalent sinks nor bivalent sources.
\end{lemma}

\begin{proof}
1) is \autoref{lm:j-unique}, and 2) follows by applying  \autoref{lm:j-unique} to $\pi_0$ being of action type
in the coordinates $(x_n, \ldots, x_1)$.
\end{proof}

Recall from \autoref{defn:n-b} that for an element $\bfw \in \ZZone$ that is negatively bordered on both sizes,
we have the left border ${\rm lb}(\bfw) \in [1, n-1]$ and the right border ${\rm rb}(\bfw)\in [2, n]$ of $\bfw$. 
As a consequence of \autoref{lm:j-k-unique},
for $\pi_0$ of symmetric action type we have the following improved version of \autoref{lm:j-unique} and 
consequence of \autoref{pr:Spi0-action-symmetric}.

\begin{lemma}\label{lm:lb-rb-S}
For a log-canonical Poisson structure $\pi_0$ of symmetric action type, one has bijections 
\begin{equation}\label{eq:lb-rb-S}
{\rm lb}: \; \calS(\pi_0) \longrightarrow {\rm lb}(\calS(\pi_0)), \;\; \bth \longmapsto {\rm lb}(\bth),\hs \mbox{and} \hs 
{\rm rb}:\;\; \calS(\pi_0) \longrightarrow {\rm rb}(\calS(\pi_0)), \;\; \bth \longmapsto {\rm rb}(\bth).
\end{equation}
Moreover, for every 
$\bth \in \calS(\pi_0)$ one has $-\lambda_{{\rm rb}(\bth)} = \lambda_{{\rm lb}(\bth)}^\prime$, and 
\begin{equation}\label{eq:blam-bth}
\blam \bth = -\lambda_{{\rm rb}(\bth)} (e_{{\rm lb}(\bth)}-e_{{\rm rb}(\bth)})
=\lambda_{{\rm lb}(\bth)}^\prime (e_{{\rm lb}(\bth)}-e_{{\rm rb}(\bth)}).
\end{equation}
\end{lemma}

\subsection{Cartan integers associated to $\pi_0$ of symmetric action type}\label{ss:Cartan-integers}
Assume that $\pi_0$ is a log-canonical Poisson structure on $\CC^n$ of symmetric 
$\TT$-action type, and let the notation be as in $\S$\ref{ss:symmetric-type}. In this section, we introduce certain 
non-positive integers, 
called {\it Cartan integers associated to $\pi_0$}, and we show that all entries of elements in 
$\calS(\pi_0)$ are either $0$, or $-1$, or multiples of these Cartan integers by  $-1$. 

Recall from  \autoref{de:Delta-n} the smoothing 
graph $\Gamma^+(\pi_0)$ of $\pi_0$ with vertex set $[1, n]$ and edge set $\calE^+(\pi_0)$, where
 one has an edge $j \rightarrow k$ for $1\leq j < k\leq n$ if and only if there is 
$\bth^{(j, k)} \in \calS(\pi_0)$ with entry $-1$ at $j$ and $k$, and that 
$\calE^+(\pi_0) \ni (j, k) \mapsto \bth^{(j, k)} \in \calS(\pi_0)$ is a bijection. 
As  $0 \notin \calS(\pi_0)_{\geq 1}$,
$\Gamma^+(\pi_0)$ has no cycles by \autoref{ld:indep}. 

\begin{lemma-definition}\label{ld:levels}
For any log-canonical Poisson structure $\pi_0$ on $\CC^n$ of symmetric $\TT$-action type,
every connected component of 
$\Gamma^+(\pi_0)$ is of the form
\[
j_1  \longrightarrow \cdots \longrightarrow j_p
\]
for some $p\geq 1$. 
We call the vertex set of a 
connected component of $\Gamma^+(\pi_0)$ a {\it level set in $[1, n]$ defined by $\pi_0$}. For $j \in [1, n]$ 
we denote by $L(j)$ the 
level set containing $j$.
\end{lemma-definition}

\begin{proof} 
Let $L$ be the vertex set of a connected component  of $\Gamma^+(\pi_0)$ with $|L| =p \geq 2$, and let $j_1 = {\rm min}(L)$. By the injectivity of 
the map ${\rm lb}$ in \eqref{eq:lb-rb-S}, there is a unique $j_2>j_1$ such that
$(j_1, j_2) \in \calE^+(\pi_0)$. If $p = 2$, we are done. Suppose that $p >2$. Then by the 
injectivity of 
the map ${\rm rb}$ in \eqref{eq:lb-rb-S}, there is a unique $j_3>j_2$ such that
$(j_2, j_3) \in \calE^+(\pi_0)$. Continuing this way, we see that the connected component  is as described. 
\end{proof}

\begin{lemma}\label{lm:L-L}
Let $\pi_0$ be any log-canonical Poisson structure on $\CC^n$ of symmetric $\TT$-action type.

1) For every  level set $L\subset [1, n]$ defined by $\pi_0$ with $|L| \geq 2$ there exists $a_{\sL} \in \CC^\times$ such that 
\begin{equation}\label{eq:a-L-0}
\blambda \bth^{(j, k)} = a_{\sL} (e_j-e_k), \quad \forall\; 
(j, k) \in \calE^+(\pi_0) \;\;\mbox{such that}\;\; j, k \in L.
\end{equation}
Specifically, 
$a_\sL = \lambda_{{\rm min}(L)}^\prime
=-\lambda_{{\rm max}(L)}$.

2) For every $(j, k) \in \calE^+(\pi_0)$ and $i \in [j+1, k-1]$, the $i^{\rm th}$ entry $\bth^{(j, k)}_i \in
\ZZ_{\geq 0}$ of $\bth^{(j, k)}$ depends only on the level sets $L(j)=L(k)\neq L(i)$.
\end{lemma}

\begin{proof}
1) follows directly from \autoref{pr:Spi0-action} and \autoref{rk:lambda-jkl}. 

To prove 2), 
let $L = L(j) = L(k)$ and $L' = L(i)$. By  \autoref{ld:levels}, $L\neq L'$.  If $L' = \{i\}$ is a singleton, then
$(j, k)$ is the only element in $\calE^+(\pi_0)$ with $j, k \in L$ such that $i \in [j+1, k-1]$, so there is nothing to prove. Assume thus that $|L'|\geq 2$. Let 
\begin{equation}\label{eq:jj-ii}
j_1  \longrightarrow \cdots \longrightarrow j_p \hs \mbox{and} \hs i_1  \longrightarrow \cdots \longrightarrow i_q
\end{equation}
be any two sub-graphs of $\Gamma^+(\pi_0)$ with $p \geq 2$, $q \geq 2$,  
$\{j_1, \ldots, j_p\} \subset L$, and $\{i_1,\ldots, i_q\} \subset L'$, and set
\begin{equation}\label{eq:bth-jj}
\bth^{[j_1, j_p]} =  \bth^{(j_1, j_2)} + \cdots + \bth^{(j_{p-1}, j_p)} 
\hs \mbox{and} \hs 
\bth^{[i_1, i_q]} =  \bth^{(i_1, i_2)} + \cdots + \bth^{(i_{q-1}, i_q)}, 
\end{equation}
and write $\bth^{[j_1, j_p]}=(\bth^{[j_1, j_p]}_1, \ldots, \bth^{[j_1, j_p]}_n)^t$ and
$\bth^{[i_1, i_q]}=(\bth^{[i_1, i_q]}_1,\ldots, \bth^{[i_1, i_q]}_n)^t$.
 By \autoref{lm:if-bordered-both}, we have
\[
\blambda \bth^{[j_1, j_p]} = a_{\sL} (e_{j_1}-e_{j_p}) \hs \mbox{and} \hs 
\blambda \bth^{[i_1, i_q]} = a_{\sL'} (e_{i_1}-e_{i_q}).
\]
It then follows from \autoref{lm:wj-wk} that 
\begin{equation}\label{eq:aL-jj}
 a_{\sL'}\left(\bth^{[j_1, j_p]}_{i_1}-\bth^{[j_1, j_p]}_{i_q}\right) = -
 a_{\sL}\left(\bth^{[i_1, i_q]}_{j_1}-\bth^{[i_1, i_q]}_{j_p}\right).
\end{equation}
Assume now that  $j_1 < i_1 < j_2$, and $j_{{p-1}} < i_q < j_p$. Then 
$\bth^{[i_1, i_q]}_{j_1}=\bth^{[i_1, i_q]}_{j_p}=0$, so
$\bth^{[j_1, j_p]}_{i_1}=\bth^{[j_1, j_p]}_{i_q}$. Thus 
$\bth^{(j_1, j_2)}_{i_1} = \bth^{[j_1, j_p]}_{i_1} = \bth^{[j_1, j_p]}_{i_q} = \bth^{(j_{p-1}, j_p)}_{i_q}$.
This finishes the proof of \autoref{lm:L-L}.
\end{proof}

\begin{definition}\label{def:a-LL}
{\rm Let $\pi_0$ be any log-canonical Poisson structure on $\CC^n$ of symmetric $\TT$-action type.
For two distinct level sets $L$ and $L'$ defined by $\pi_0$ such that 
$|L|\geq 2$ and $L' \cap [{\rm min}(L), {\rm max}(L)] \neq \emptyset$, we set
\[
a_{\sL', \sL} = -\bth^{(j, k)}_i\in \ZZ_{\leq 0}
\]
for any $(j, k) \in \calE^+(\pi_0)$ and $i \in [j+1, k-1]$ such that $j, k \in L$ and $i \in L'$, and 
we call the non-positive integer 
$a_{\sL', \sL}$, which is well-defined by \autoref{lm:L-L}, a {\it Cartan integer} associated to $\pi_0$.
}
\end{definition}

The following fact now follows from the definitions.

\begin{lemma}\label{lm:bth-Cartan-integers}
Let $\pi_0$ be any log-canonical Poisson structure on $\CC^n$ of symmetric $\TT$-action type. For any 
$(j, k) \in \calE^+(\pi_0)$, one has
\begin{equation}\label{eq:theta-jk-a}
\bth^{(j, k)} = (0, \ldots, 0, -1, -a_{\sL(j+1), \sL(j)}, \; \ldots, \; -a_{\sL(k-1), \sL(j)}, \, -1, 0, \ldots, 0)^t
\in \calS(\pi_0).
\end{equation}
\end{lemma}

We say two distinct level sets $L$ and $L'$ {\it interlace each other} if $|L| \geq 2$, $|L'| \geq 2$, 
\[
L' \cap [{\rm min}(L), {\rm max}(L)] \neq \emptyset \quad \mbox{and} \quad
L \cap [{\rm min}(L'), {\rm max}(L')] \neq \emptyset.
\]
In such a case we also have, by \autoref{ld:levels}, 
 the well-defined 
$a_{\sL} \in \CC^\times$ and $a_{\sL'} \in \CC^\times$.

\begin{lemma}\label{lm:interlace}
Let $\pi_0$ be any log-canonical Poisson structure on $\CC^n$ of symmetric $\TT$-action type.
For any two distinct level sets $L$ and $L'$ in $[1, n]$ defined by $\pi_0$ that interlace each other, one has
\[
a_{\sL} a_{\sL, \sL'} = a_{\sL'} a_{\sL', \sL}.
\]
\end{lemma}

\begin{proof}
Let $j_1 = {\rm min}(L)$ and $i_1 = {\rm min}(L')$. As $L \cap L'=\emptyset$, we may assume that $j_1 < i_1$. 
As $L$ and $L'$ interlace each other, we may assume that there are two sub-graphs of $\Gamma^+(\pi_0)$ 
as in \eqref{eq:jj-ii}, with $p \geq 2$ and $q \geq 2$,  such that
$\{j_1, \ldots, j_p\} \subset L$,  $\{i_1,\ldots, i_q\} \subset L'$, and 
$j_1 < i_1 < \cdots < j_p < i_q$.
Let $\bth^{[j_1, j_p]}$ and $\bth^{[i_1, i_q]}$ be defined as in \eqref{eq:bth-jj}.  It now follows from 
\eqref{eq:aL-jj} and 
\[
\bth^{[j_1, j_p]}_{i_1} = -a_{\sL, \sL'}, \hs \bth^{[j_1, j_p]}_{i_q}=0, \hs 
 \bth^{[i_1, i_q]}_{j_1}=0, \hs \bth^{[i_1, i_q]}_{j_p} = -a_{L, L'}
\]
that $a_{\sL} a_{\sL, \sL'} = a_{\sL'} a_{\sL', \sL}$.
\end{proof}

\subsection{Classification of symmetric Poisson CGL extensions}\label{ss:sym-CGL-classification}
 Let again $\TT$ be a complex algebraic torus with Lie algebra $\t$
 and character lattice $X^*(\TT) \subset \t^*$.
 We now recall the notion of symmetric $\TT$-Poisson CGL extensions 
 by K. Goodearl and M. Yakimov
 \cite{GY:Poi-CGL}.

\begin{definition}\label{def:Poi-CGL}\cite[Definition 6.1]{GY:Poi-CGL} A  {\it
symmetric $\TT$-Poisson CGL extension} of dimension $n$ is the polynomial 
algebra $R=\CC[x_1, \ldots, x_n]$ together with a Poisson bracket $\{\,, \, \}_\pi$ 
and a rational $\TT$-action preserving
$\{\,, \, \}_\pi$ such that

1) each $x_j$, for $j \in [1, n]$, is a $\TT$-weight vector with $\TT$-weight 
$\beta_j \in X^*(\TT) \subset \t^*$;

2) there exist $h_1, \ldots, h_n, h_1^\prime, \ldots,  h_n^\prime \in \t$ such that 
$\beta_j(h_j) \neq 0$ and $\beta_j(h_j^\prime) \neq 0$ for each $j \in [1, n]$ and 
$\beta_j(h_k) = -\beta_k(h_j')$ for all $1 \leq j < k \leq n$; 

3) the Poisson bracket $\{\,, \, \}_\pi$ has the form
\[
\{x_j, x_k\}_{\pi} = -\beta_j(h_k) x_jx_k + \varphi_{j, k}, \hs \hs 1 \leq j < k \leq n,
\]
where $\varphi_{j, k} \in \CC[x_{j+1}, \ldots, x_{k-1}]$ for each pair $1 \leq j< k\leq n$. 
The variables $x_1, \ldots, x_n$ are called the {\it CGL generators} of the Poisson CGL extension
$R$.
\end{definition}

\begin{definition}\label{def:pi-CGL}
Given the linear coordinates $(x_1, \ldots, x_n)$ on $\CC^n$, if $\pi$ is a
$\TT$-invariant  Poisson structure on $\CC^n$ 
such that $(\CC[x_1, \ldots, x_n], \{\,, \, \}_\pi)$ is a 
symmetric $\TT$-Poisson CGL extension, we also say that  $(\CC^n, \pi)$, or simply $\pi$, is a 
symmetric $\TT$-Poisson CGL extension in the coordinates $(x_1, \ldots, x_n)$.
\end{definition}

Recalling from \eqref{eq:mW-pi} the definition of 
$\mW(\pi)\subset \calW_2\backslash \{0\}$ for an arbitrary 
algebraic Poisson structure $\pi$ on $\mX = \CC^n$,
a $\TT$-symmetric Poisson CGL extension of dimension $n$ is thus nothing but a $\TT$-invariant 
algebraic Poisson structure $\pi$ on $\mX =\CC^n$ such that

1) the log-canonical term of $\pi$ is of symmetric $\TT$-action type;

2) the set $\mW(\pi)$ is negatively bordered on both sides. 

\noindent
We now have the following classification theorem on symmetric $\TT$-Poisson CGL extensions.

\begin{theorem}\label{thm:sym-CGL} 
1) For any log-canonical Poisson structure $\pi_0$ on $\CC^n$ of 
symmetric $\TT$-action type and any
$\calS \subset \calS(\pi_0)$, every $\calS$-admissible algebraic Poisson deformation of
$\pi_0$ (see \autoref{def:calS-admi}) is a symmetric $\TT$-Poisson CGL extension of dimension $n$;

2) Every symmetric $\TT$-Poisson CGL extension $\pi$ of dimension $n$ is the unique  $\calS_\pi$-admissible algebraic Poisson deformation of
its log-canonical term 
$\pi_0$ along $\pi_1 \in \mfX^2(\mX)^{\calS_\pi}$, where
\[
\calS_\pi = \mW(\pi)\backslash \mW(\pi)_{2} \subset \calS(\pi_0), 
\]
and $\pi_1$ is the $\mfX^2(\mX)^{\calS_\pi}$-component of $\pi$ with respect to the
$\CCsn$-weight space decomposition of $\mfX^2(\mX)$. 
\end{theorem}

\begin{proof}
1) Let $\pi_0$ be as given in 1) and let $\calS \subset \calS(\pi_0)$. By 
\autoref{pr:Spi0-action-symmetric}, $\calS$ is negatively bordered on both sides, and 
so is $\calS_{\geq 1}$ by \autoref{lm:nbr-W0}. It now follows from \autoref{def:Poi-CGL} and
\autoref{def:calS-admi}
that all the  $\calS$-admissible algebraic Poisson deformation of
$\pi_0$, in particular the maximal ones, are symmetric $\TT$-Poisson CGL extensions of dimension $n$;

2) Given a symmetric $\TT$-Poisson CGL extension $\pi$ of dimension $n$,   by 
\autoref{def:Poi-CGL} the log-canonical term $\pi_0$ of $\pi$ is
of symmetric $\TT$-action type and $\pi$ has Property \eqref{eq:W1}. It then 
follows from \autoref{thm:4.9-W1} that $\pi$ is as described.
\end{proof}

By \autoref{thm:max-family-W1} and \autoref{defn:max-W1}, for a log-canonical Poisson structure $\pi_0$ of symmetric $\TT$-action type we have the maximal family $\pi^{\calS(\pi_0)}(c)$ of algebraic Poisson
deformations of $\pi_0$, where $c \in \CC^{\calS(\pi_0)}$.

\begin{definition-remark}\label{defn-rk:CGL-maximal}
{\rm
We say a symmetric $\TT$-Poisson CGL extension $\pi$ with log-canonical term $\pi_0$ is  {\it maximal}
if $\pi = \pi^{\calS(\pi_0)}(c)$ for some $c \in (\CC^\times)^{\calS(\pi_0)}$.
By \autoref{cor:member-summand}, 
every symmetric $\TT$-Poisson CGL extension 
is both a direct summand of a maximal one  and the 
limit of a family of maximal ones.
\hfill $\diamond$
}
\end{definition-remark}

By \autoref{def:Poi-CGL}, any rescaling of the CGL generators 
of a symmetric Poisson CGL extension is again a symmetric Poisson CGL extension with the same
log-canonical term.  It is shown in 
\cite[$\S$9.2]{GY:Poi-CGL} that 
a symmetric Poisson CGL extension needs to be normalized, via rescaling the CGL generators,
for the construction of cluster structures.
The following definition is motivated by \cite[$\S$9.2]{GY:Poi-CGL}.

\begin{definition}\label{defn:normalized}
{\rm
For a log-canonical Poisson structure $\pi_0$ of symmetric $\TT$-action type, we also set  
\[
a_\bth = -\lambda_{{\rm lb}(\bth)}= \lambda^\prime_{{\rm lb}(\bth)}\in \CC^\times,
\]
for $\bth \in \calS(\pi_0)$, so that $\blam \bth = a_\bth (e_{{\rm lb}(\bth)}-e_{{\rm rb}(\bth)})$.
For $\calS \subset \calS(\pi_0)$, we call $\pi^\calS((a_\bth)_{\bth \in \calS})$
the {\it normalized} $\calS$-admissible algebraic Poisson deformation of $\pi_0$.
We call $\pi^{\calS(\pi_0)}((a_\bth)_{\bth \in \calS(\pi_0)})$
the {\it maximal normalized} admissible algebraic Poisson deformation of $\pi_0$.
}
\end{definition}

We now have the following consequence of \autoref{pr:scaling-to-c} and 
\autoref{thm:sym-CGL}.

\begin{corollary}\label{cor:classification-CGL}
Given a log-canonical Poisson structure $\pi_0$ on $\CC^n$ of symmetric $\TT$-action type, 
symmetric Poisson CGL extensions $\pi$  with log-canonical term
$\pi_0$ are, up to rescaling of CGL generators, in bijection with subsets of $\calS(\pi_0)$ via the correspondences
\begin{equation}\label{eq:bf-1}
\pi \longmapsto \calS_\pi  = \mW(\pi)/\mW(\pi)_{2} \subset \calS(\pi_0) \hs \mbox{and} \hs
\calS(\pi_0) \supset \calS \longmapsto \pi^\calS((a_\bth)_{\bth \in \calS}).
\end{equation}
\end{corollary}

\begin{remark}\label{rk:bf-1}
{\rm
In place of $(a_\bth)_{\bth \in \calS}\in (\CC^\times)^\calS$, another choice of the bijection in 
\eqref{eq:bf-1} is 
\[
\calS(\pi_0) \supset \calS \longmapsto \pi^\calS({\bf 1}_\calS),
\]
where ${\bf 1}_\calS \in (\CC^\times)^\calS$ has all entries $1$.
\hfill $\diamond$
}
\end{remark}

\begin{example}\label{ex:beta-2p-01}
{\rm 
Continuing with {\bf Case 3} of \autoref{ex:beta-2p-00}, where $p \in \ZZ_{>0}$ and $\pi_0$ is given in \eqref{eq:pi0-p}, we have
$\calS(\pi_0) = \{\bth^{(1,4)}, \bth^{(2,3)}\}$, where $\bth^{(1,4)} = (-1, 2p, 2p, -1)^t$ and
$\bth^{(2,3)} = (0, -1, -1, 0)^t$, and 
\[
V_{\bth^{(1,4)}} = x_2^{2p}x_3^{2p} \frac{\partial}{\partial x_1}\! \wedge \!\frac{\partial}{\partial x_4} 
\hs \mbox{and} \hs V_{\bth^{(2,3)}} =  \frac{\partial}{\partial x_2}\! \wedge \!\frac{\partial}{\partial x_3}.
\]
Let $\pi_1 = V_{\bth^{(1,4)}} + V_{\bth^{(2,3)}}$,  and let $\pi(1, 1)$ be the $\calS(\pi_0)$-admissible algebraic Poisson 
deformation of $\pi_0$ along $\pi_1$. Note that 
$\calS_{\geq 2} \cap \calW_2 =\{\bth^{(1,4)}+k \bth^{(2,3)}: k \in [1, 2p]\}$. By \autoref{rmk:in-practice}, 
\[
\pi(1, 1) = \pi_0 + V_{\bth^{(1,4)}} + V_{\bth^{(2,3)}} + \sum_{k=1}^{2p} \mu_k 
(x_2x_3)^{2p-k} \frac{\partial}{\partial x_1}\! \wedge \!\frac{\partial}{\partial x_4}
\]
for some $\mu_1, \ldots, \mu_{2p} \in \CC$. Solving for $\mu_1, \ldots, \mu_{2p}$ from $[\pi, \pi]_{\rm Sch} = 0$, 
we get 
\[
\pi(1, 1) = \pi_0 + \frac{\partial}{\partial x_2}\! \wedge \!\frac{\partial}{\partial x_3} +(1+x_2x_3)^{2p}\frac{\partial}{\partial x_1}\! \wedge \!\frac{\partial}{\partial x_4}.
\]
We have $-\lambda_{{\rm rb}(\bth^{(1, 4)})} = -p^2$ and $-\lambda_{{\rm rb}(\bth^{(2, 3)})} = -1$. 
The $\calS(\pi_0)$-admissible deformation of $\pi_0$ along $\pi_1 =-p^2V_{\bth^{(1,4)}} - V_{\bth^{(2,3)}}$ is given by
\[
\pi(-p^2, -1) = \pi_0 - \frac{\partial}{\partial x_2}\! \wedge \!\frac{\partial}{\partial x_3} -p^2(x_2x_3-1)^{2p}\frac{\partial}{\partial x_1}\! \wedge \!\frac{\partial}{\partial x_4}.
\]
The rescaling of the coordinates $(x_, x_2, x_3, x_4) \mapsto (-p^{-2}x_1, x_2, -x_3, x_4)$
brings $\pi(1, 1)$ to $\pi(-p^2, -1)$.
\hfill $\diamond$
}
\end{example}

\subsection{The predecessor
and successor maps in terms of  $\calS_\pi$}
Fix a 
symmetric $\TT$-Poisson CGL extension 
$(\CC[x_1, \ldots, x_n], \{\, ,\, \}_\pi)$ with log-canonical term of $\pi_0$. In this subsection, 
we first give an alternative description of the subset 
$\calS_\pi = \mW(\pi)\backslash \mW(\pi)_{2}$ of $\calS(\pi_0)$
in terms of the Poisson 
bracket $\{\, ,\, \}_\pi$ on $\CC[x_1, \ldots, x_n]$. We then show how the  {\it predecessor
and successor maps} for $\pi$ defined by Goodearl and Yakimov in \cite{GY:Poi-CGL} can be expressed in terms of elements in $\calS_\pi$. 
To this end, 
write again
$\pi = \pi_0 + \pi_{\rm tail}$ and correspondingly 
\[
\{x_j, x_k\}_{\pi} = 
\{x_j, x_k\}_{\pi_0} + \{x_j, x_k\}_{\pi_{\rm tail}},\hs 1 \leq j < k \leq n,
\]
where $\{x_j, x_k\}_{\pi_{\rm tail}} \in \CC[x_{j+1}, \ldots, x_{k-1}]$. 
Consider also, as in \cite[$\S$5.2]{GY:Poi-CGL}, the 
{\it reverse lexicographic order} $\prec$ on $\ZZ^n$, i.e., for
$f = (f_1, \ldots, f_n)^t$ and $g = (g_1, \ldots, g_n)^t \in \ZZ^n$, one has 
$f \prec g$ if and only if there exists $i \in [1, n]$ such that $f_j = g_j$ 
for all $j \in [i+1, n]$ and $f_i < g_i$. Every non-zero $\phi \in \CC[x_1, \ldots, x_n]$ then has a 
well-defined {\it $x$-degree}, denoted as  $\deg_x\phi \in (\ZZ_{\geq 0})^n$, and an {\it $x$-leading term}, 
denoted as
${\rm lt}_x \phi = \xi x^{\deg_x\phi}$ for some 
$\xi \in \CC^\times$, such that 
\[
\phi = {\rm lt}_x \phi + \sum_{g \prec \deg_x\phi} \xi_g x^g
\]
as a finite sum and with $\xi_g \in \CC^\times$ for each $g \in (\ZZ_{\geq 0})^n$  appearing 
in the sum on the right-hand side.  
Recall from \autoref{lm:lb-rb-S} that we have  the bijection 
${\rm rb}: \calS_\pi \rightarrow {\rm rb}(\calS_\pi)\subset [2, n],  \bth \mapsto {\rm rb}(\bth)$.
For $k \in [2, n]$ let 
\[
I_\pi(k) = \{i \in [1, k-1]:\;  \{x_i, x_k\}_{\pi_{\rm tail}} \neq 0\}.
\]

\begin{proposition}\label{pr:p-pi}
For any  symmetric $\TT$-Poisson CGL extension $\pi$ of dimension  $n$, one has
\[
{\rm rb}(\calS_\pi) = \{k \in [2, n]: I_\pi(k) \neq \emptyset\}. 
\]
Setting  $p_\pi(k) = {\rm max} (I_\pi(k))$ for $k \in {\rm rb}(\calS_\pi)$,  one has 
$\calS_\pi = \{\bth^{(p_\pi(k), k)}: k \in {\rm rb}(\calS_\pi)\}$, where for $k \in {\rm rb}(\calS_\pi)$,
\begin{equation}\label{eq:p-pi-Poi}
\bth^{(p_\pi(k), k)}= -e_{p_\pi(k)}-e_k +\deg_x \left(\{x_{p_\pi(k)}, \, x_k\}_{\pi_{\rm tail}}\right)
\hs \mbox{and} \hs
\blam \bth^{(p_\pi(k), k)} = -\lambda_k (e_{p_\pi(k)}-e_k).
\end{equation}
\end{proposition}

\begin{proof}
We first prove the following claims:

Claim (a): for any  $k \in [1, n]$, there exists $\bfw \in \mW(\pi)$ with $\rb(\bfw) = k$ if and only if 
$k \in {\rm rb}(\calS_\pi)$;

Claim (b): if $k \in {\rm rb}(\calS_\pi)$, then for every 
$\bfw \in \mW(\pi)$ with $\rb(\bfw) = k$ one has 
$\lb(\bfw) \leq \lb(\bth)$, where $\bth$ is the unique element in $\calS_\pi$ with $\rb(\bth)=k$;

Claim (c): if $\bfw \in \mW(\pi)$ such that $\rb(\bfw) = k$, then either $\bfw = \bth$, or $\bfw \prec \bth$ and $\bfw = \bfw'+\bth$, where $\bfw' \in (\calS_\pi)_{\geq 1}$
does not have $\bth$ in its decomposition as a sum of elements in $\calS_\pi$.

To prove the claims,  for $\bfw \in (\calS_\pi)_{\geq 1}$ let $\calS_\pi(\bfw)$ be the set of all elements in 
$\calS_\pi$ that appear in $\bfw$ when expressed as a (unique finite) sum of elements in $\calS_\pi$. Suppose that 
$\bfw \in \mW(\pi)$ and $\rb(\bfw) = k$. Let $\bth \in \calS_\pi(\bfw)$ be such that $\rb(\bth) = {\rm max}
\{\rb(\widetilde{\bth}): \widetilde{\bth} \in \calS_\pi(\bfw)\}$. By 
\autoref{lm:j-k-unique}, such a $\bth \in \calS_\pi(\bfw)$ is unique. As $\bfw \in \ZZone$, the element
$\bth$  appears exactly once in $\bfw$ when written as a sum of elements in $\calS_\pi(\bfw)$, and $\rb(\bth) = k$. This proves that $k\in \rb(\calS_\pi)$, and Claim (a) follows. Let  
$\bth'$ be the unique element in $\calS_\pi(\bfw)$ such that 
$\lb(\bth') = {\rm min}\{\lb(\widetilde{\bth}): \widetilde{\bth}\in \calS_\pi(\bfw)\}$, then
$\lb(\bfw) = \lb(\bth')$. 
As $\bth \in \calS_\pi(\bfw)$, we have $\lb(\bfw) = \lb(\bth') \leq \lb(\bth)$. This proves Claim (b). Suppose, in addition, that $\bfw \neq \bth$. Then $\bfw = \bfw' + \bth$ with
$\bfw' \in (\calS_\pi)_{\geq 1}$ and $\bth \notin \calS_\pi(\bfw')$. As 
$\widetilde{\bth} \prec {\bf 0}$ for every $\widetilde{\bth} \in \calS_\pi(\bfw')$, where ${\bf 0} = (0, \ldots, 0)^t \in \ZZ^n$,
we have $\bfw' \prec {\bf 0}$, 
Thus $\bfw =\bfw' + \bth \prec \bth$, proving Claim (c).

Fix now $k \in [1, n]$. Assume first that $k \in [2, n]$ and $I_\pi(k) \neq 0$. 
For
$i \in I_\pi(k)$, let 
\[
 \pi_{{\rm tail}, i,k} :=\{x_i, \, x_k\}_{\pi_{\rm tail}} \frac{\partial}{\partial x_i}\wedge 
\frac{\partial}{\partial x_k},
\]
and let $\mW(\pi_{{\rm tail}, i,k})$ be the non-empty set of all  
$\CCsn$-weights of the monomial bi-vector fields appearing in  $\pi_{{\rm tail}, i,k}$. 
Then ${\rm rb}(\bfw) = k$
for any $\bfw  \in \mW(\pi_{{\rm tail}, i,k})$ and any  $i \in I_\pi(k)$, so 
$k \in {\rm rb}(\calS_\pi)$
by Claim (a).

Conversely, suppose that $k \in {\rm rb}(\calS_\pi)$ and let 
$\bth$ be the (necessarily unique) element in 
$\calS_\pi$ such that 
$\rb(\bth) = k$, and let $j = \lb(\bth)$. Then 
$\pi_{\rm tail}$ contains the monomial bi-vector field
\[
\xi V_\bth = \xi \prod_{i \neq j, k} x_i^{\bth_i}\frac{\partial}{\partial x_j}\wedge 
\frac{\partial}{\partial x_k}
\]
for some $\xi \in \CC^\times$.
In particular, $\{x_j, x_k\}_{\pi_{\rm tail}} \neq 0$, so $j \in I_\pi(k)$ and thus $I_\pi(k) \neq \emptyset$. 
For any $i \in I_\pi(k)$, applying Claim (b) to any  $\bfw \in \mW(\pi_{{\rm tail}, i,k})$ we see  that  
$i ={\rm lb}(\bfw) \leq \lb(\bth)= j$. Thus 
\[
j = {\rm max}\{i \in I_\pi(k)\}=p_\pi(k).
\]
Moreover, for any  $\bfw \in \mW(\pi_{{\rm tail}, j, k})$ and $\bfw \neq \bth$, as $\rb(\bfw) = k$, we have 
$\bfw \prec \bth$ by Claim (c). Thus  
$\bfw + e_j + e_k  \prec \bth + e_j + e_k$, from which it 
follows that $\bth + e_j+e_k = {\rm deg}_x \left(\{x_j, x_k\}_{\pi_{\rm tail}}\right)$.
This proves the first identity in \eqref{eq:p-pi-Poi}.
The second identity in \eqref{eq:p-pi-Poi} is proved in \autoref{pr:Spi0-action}.
\end{proof}

Similarly, recall from \autoref{lm:lb-rb-S} that we have  the bijection 
${\rm lb}: \calS_\pi \rightarrow {\rm lb}(\calS_\pi),  \bth \mapsto {\rm lb}(\bth) \in  [1, n-1]$.
For $j \in [1, n-1]$ such that
$\{x_j, x_k\}_{\pi_{\rm tail}} \neq 0$ for some $k \in [j+1, n]$, set
\begin{equation}\label{eq:s-pi-j}
s_\pi(j) = {\rm min} \{k \in [j+1, n]: \{x_j, x_k\}_{\pi_{\rm tail}} \neq 0\}.
\end{equation}
The following analog of \autoref{pr:p-pi}  is proved similarly.

\begin{proposition}\label{pr:s-pi}
For any  symmetric $\TT$-Poisson CGL extension $\pi$ of dimension $n$, one has
\[
{\rm lb}(\calS_\pi) = \{j \in [1, n-1]: \;\exists \; k \in [j+1, n]\; \mbox{such that}\; 
\{x_j, x_k\}_{\pi_{\rm tail}} \neq 0\},
\]
and 
$\calS_\pi = \{\bth^{(j, s_\pi(j))}: j \in {\rm lb}(\calS_\pi)\}$, where for $j \in {\rm lb}(\calS_\pi)$
(recall from \autoref{pr:Spi0-action-symmetric} that $\lambda_j^\prime = -\lambda_{s_\pi(j)}$), 
\begin{equation}\label{eq:s-pi-Poi}
\bth^{(j, s_\pi(j)}= -e_{j}-e_{s_\pi(j)} +\deg_x \left(\{x_{j}, \, x_{s_\pi(j)}\}_{\pi_{\rm tail}}\right)
\hs \mbox{and} \hs \blam \bth^{(j, s_\pi(j)} = \lambda_{j}^\prime(e_j-e_{s_\pi(j)}).
\end{equation}
\end{proposition}

\begin{remark}\label{rk:s-p-pi}
{\rm
By \autoref{pr:p-pi} and \autoref{pr:s-pi}, we can regard $p_\pi$ and $s_\pi$ as two maps 
\begin{align*}
&p_\pi:\;\; {\rm rb}(\calS_\pi) \longrightarrow {\rm lb}(\calS_\pi), \;\;\; {\rm rb}(\bth) \longmapsto {\rm lb}(\bth),\\
&s_\pi:\;\; {\rm lb}(\calS_\pi) \longrightarrow {\rm rb}(\calS_\pi), \;\;\; {\rm lb}(\bth) \longmapsto {\rm rb}(\bth),
\end{align*}
that are inverses of each other. 
\hfill $\diamond$
}
\end{remark}

We now extend the definition of $p_\pi$ in \autoref{pr:p-pi} to a map
\[
p_\pi: \; [1, n] \longrightarrow \{-\infty\} \sqcup [1, n-1]
\]
by setting, for all $k \in [1, n]$, 
\begin{align}\nonumber
p_\pi(k) &= \begin{cases} -\infty, & \hs \mbox{if} \; \;
\{x_i, x_k\}_{\pi_{\rm tail}} = 0 \; \forall\; i \in [1, k-1],\\
{\rm max}\{i \in [1, k-1]: \; \{x_i, x_k\}_{\pi_{\rm tail}} \neq 0\}, &\hs  \mbox{otherwise},\end{cases}\\
\label{eq:p-pi-bth}& = \begin{cases} {\rm lb}(\bth), & \hs \mbox{if}\;\;k = {\rm rb}(\bth) \;\;
\mbox{for (a necessarily unique)} \;\;  
\bth\in \calS_\pi,\\
-\infty, & \hs \mbox{otherwise}, 
\end{cases}
\end{align}
Similarly, define $s_\pi: [1, n] \to [2, n] \sqcup \{+\infty\}$ by setting, for $j \in [1, n]$, 
\begin{align}\nonumber
s_\pi(j)& = \begin{cases} +\infty, & \hs \mbox{if} \; 
\{x_j, x_k\}_{\pi_{\rm tail}} = 0 \; \forall\; k \in [j+1, n],\\
{\rm min}\{k \in [j+1, n]: \; \{x_j, x_k\}_{\pi_{\rm tail}} \neq 0\}, &\hs  \mbox{otherwise},\end{cases}\\
\label{eq:s-pi-bth}&=\begin{cases} {\rm rb}(\bth), & \hs \mbox{if}\;j = {\rm lb}(\bth) \;
\mbox{for (a necessarily unique)} \; 
\bth\in \calS_\pi,\\
+\infty, &\hs \mbox{otherwise}.
\end{cases}
\end{align}
On the other hand, 
Goodearl and Yakimov defined in \cite[$\S$5.2]{GY:Poi-CGL} two maps
\[
p: \;\; [1, n]\longrightarrow \{-\infty\}\cup [1, n-1] \hs \mbox{and} \hs 
s: \;\;[1, n] \longrightarrow  [2, n] \cup \{+\infty\},
 \]
respectively called the  {\it predecessor map} and the {\it successor map} of  the symmetric Poisson CGL extension $\pi$ (with respective to the ordered CGL generators $(x_1, \ldots, x_n)$).

\begin{lemma}\label{lm:s-p-same}
One has  $p_\pi = p: [1, n]\rightarrow \{-\infty\}\cup [1, n-1]$ and 
and $s_\pi=s: [1, n] \longrightarrow  [2, n] \cup \{+\infty\}$. 
\end{lemma}

\begin{proof}
Let $k \in [1, n]$. By \cite[Theorem 5.5]{GY:Poi-CGL}, 
$p(k) = -\infty$
if and only if $\{x_i, x_k\}_{\pi_{\rm tail}} = 0$ for every $i \in [1, k-1]$.
Assume that $p(k)  \neq -\infty$. By \cite[Corollary 8.11]{GY:Poi-CGL}
(see also the paragraph from \cite[Page 64]{GY:Poi-CGL} 
before \cite[Example 8.12]{GY:Poi-CGL}), one has 
$\{x_{p(k)}, x_k\}_{\pi_{\rm tail}} \neq  0$. Thus $p(k) = p_\pi(k)$ if
$p(k) = k-1$. Suppose that $p(k) < k-1$. By \cite[Corollary 8.7]{GY:Poi-CGL}
and applying \cite[Theorem 5.5]{GY:Poi-CGL} to the {\it interval Poisson subalgebra}
(see \cite[Page 60]{GY:Poi-CGL}) $\CC[x_{p(k)+1}, \ldots, x_k]$, one has 
$\{x_{i}, x_k\}_{\pi_{\rm tail}} =  0$ for all $i \in [p(k)+1, k-1]$. Thus 
$p(k) = {\rm max}\{i \in [1, k-1]: \{x_i, x_k\}_{\pi_{\rm tail}} \neq  0\} = p_\pi(k)$.

 The {\it successor map} $s: [1, n] \to  [2, n] \cup \{+\infty\}$ is defined in 
\cite[$\S$5.2]{GY:Poi-CGL} by $s(j) = +\infty$ if $j \neq p_(k)$ for any $k \in [j+1, n]$
and otherwise $s(j)$ is  (the necessarily unique) $k \in [j+1, n]$ such that $p(k) = j$. It now
follows from the definition of $s_\pi$ that $s = s_\pi$. 
\end{proof}

\begin{remark}\label{rk:normal}
{\rm
It follows from \eqref{eq:s-pi-Poi} and \autoref{lm:s-p-same} that the symmetric Poisson CGL extension $\pi$ is normal
in the sense of \cite[$\S$9.2]{GY:Poi-CGL} if and only if $\pi$ is the normalized 
$\calS_\pi$-admissible algebraic Poisson deformation of $\pi_0$ as in \autoref{defn:normalized}.
\hfill $\diamond$
}
\end{remark}

Following \cite[$\S$5.2]{GY:Poi-CGL} and \cite[$\S$8.3]{GY:Poi-CGL}, 
for $k \in [1, n]$ define 
\[
L_\pi(k) = \{p_\pi^{o_-(k)}(k), \;\cdots, \;p_\pi(k),\; k, \;s_\pi(k), \;\ldots, \;s_\pi^{o_+(k)}(k)\}
\subset [1, n]
\]
as  {\it level set of $k$ defined by $\pi$}, 
where 
\[
o_-(k) = {\rm max}\{m \in \ZZ_{\geq 0}: p_\pi^m(k) \neq -\infty\},\quad
o_+(k) = {\rm max}\{m \in \ZZ_{\geq 0}: s_\pi^m(k) \neq +\infty\}.
\]
On the other hand, recall from \autoref{ld:levels}
that we have defined the level set $L(k) \subset [1, n]$ for each $k \in [1, n]$ using the entire 
$\calS(\pi_0)$. For $k \in [1, n]$, one then has 
$L_\pi(k) \subset L(k)$, and 
$|L_\pi(k)| \geq 2$ if and only if $k \in {\rm lb}(\calS_\pi)\cup {\rm rb}(\calS_\pi)$. 
Moreover, if $k \in [1, n]$ is such that 
$L_\pi(k) = \{k_0, k_1, \ldots, k_l\}$, where
$k_0 < k_1 < \cdots < k_l$ and $\l \geq 1$, then by \eqref{eq:p-pi-bth} and \eqref{eq:s-pi-bth} one has the
 well-defined (see \autoref{lm:L-L})
\begin{equation}\label{eq:lambda-level}
\lambda_{k_0}^\prime = \lambda_{k_1}^\prime =\cdots =  \lambda^\prime_{k_{l-1}} = -\lambda_{k_l} = 
 -\lambda_{k_{l-1}}\cdots  = -\lambda_{k_1} = a_{\sL(k)}.
\end{equation}
The identities in \eqref{eq:lambda-level} are first proved in \cite[Proposition 8.8]{GY:Poi-CGL}.
 For  $j \in 
{\rm lb}(\calS_\pi)\cup {\rm rb}(\calS_\pi)$, set 
\begin{equation}\label{eq:a-SL-SL-0}
a_{\sL_\pi(j)} = a_{\sL(j)}
\end{equation} 
and for  $i \in [1, n]$ such that $L_\pi(i) 
\cap [{\rm min}(L_\pi(j)), {\rm max}(L_\pi(j))] \neq \emptyset$, set
\begin{equation}\label{eq:a-SL-SL-1}
a_{\sL_\pi(i), \sL_\pi(j)} = a_{\sL(i), \sL(j)},
\end{equation}
where recall from \autoref{def:a-LL} that $a_{\sL(i), \sL(j)}$ is 
a Cartan integer associated to $\pi_0$.

\begin{remark}\label{rk:GY-also}
{\rm
It follows from \autoref{pr:s-pi} and \autoref{lm:bth-Cartan-integers} that
for every $j \in {\rm lb}(\calS_\pi)$ and
$i \in [j+1, s_\pi(j)-1]$, the $i^{\rm th}$ entry of ${\rm deg}_x (\{x_j, x_{s_\pi(j)}\})_{\pi_{\rm tail}})$ 
depends only on the level set $L_\pi(i)$. This fact is also proved in \cite[Corollary 8.11]{GY:Poi-CGL}.
\hfill $\diamond$
}
\end{remark}

\subsection{The Goodearl-Yakimov initial mutation matrix in terms of  $\calS_\pi$}\label{ss:mut-matrix}
Recall \cite{FZ:I} that a seed in $\CC(x_1, \ldots, x_n)$ is 
a pair $({\bf y}, M)$, where ${\bf y} = \{y_1, \ldots, y_n\}$ is an ordered set of free transcendental generators of
$\CC(x_1, \ldots, x_n)$ over $\CC$, and $M$ is an integral matrix with rows indexed by $j \in [1, n]$
and columns indexed by $k \in {\rm ex}$ for some subset ${\rm ex}$ of $[1, n]$, called the {\it exchange set}, 
such that the 
${\rm ex} \times {\rm ex}$ sub-matrix of $M$ is skew-symmetrizable.  The set ${\bf y}$ is called 
the {\it extended cluster} and $M$ the {\it mutation matrix} of the seed $({\bf y}, M)$. Seeds in 
$\CC(x_1, \ldots, x_n)$ {\it mutate} according to certain mutation rules, and the {\it mutation equivalence
class}  of a seed $({\bf y}, M)$ defines a sub-algebra of $\CC(x_1, \ldots, x_n)$ (with the frozen variables not inverted), called 
the {\it cluster algebra} associated to  $({\bf y}, M)$. We refer to \cite{FZ:I} 
for the basics of cluster algebras. 

Associated to any  dimension $n$ symmetric $\TT$-Poisson CGL extension $\pi$, 
Goodearl and Yakimov constructed in \cite{GY:Poi-CGL} 
 a seed $({\bf y}, M)$ in $\CC(x_1, \ldots, x_n)$  and showed that 
the cluster algebra defined by  $({\bf y}, M)$
coincides with the polynomial ring $\CC[x_1, \ldots, x_n]$. In this section, we prove an explicit formula
expressing the mutation matrix $M$ in terms of the subset $\calS_\pi$ of $\calS(\pi_0)$, where
$\pi_0$ is the log-canonical term of $\pi$.

Fix thus a dimension $n$ symmetric $\TT$-Poisson CGL extension $\pi$ with log-canonical term
$\pi_0$ in the linear coordinates $(x_1, \ldots, x_n)$ on $\CC^n$, 
and let the notation be as in $\S$\ref{ss:sym-CGL-classification}. 
Let $R_k=\CC[x_1, \ldots, x_k]$ for  $k \in [1, n]$
(and we set $R_0 = \CC$), and note that
$R_k$ is a $\TT$-invariant Poisson subalgebra of $(R = R_n, \{\, , \, \}_\pi)$. For $k \in [1, n]$, 
an element $y \in R_k$ is called a {\it $\TT$-Poisson prime} if $y$ is a $\TT$-weight vector 
and a prime element of $R_k$ such that
 $\{y, R_k\}_\pi
\subset y R_k$. 
By \cite[Theorem 5.5]{GY:Poi-CGL}, there is a unique sequence 
\begin{equation}\label{eq:by}
{\bf y} = (y_1, y_2, \ldots, y_n),
\end{equation}
where $y_k$ for each $k \in [1, n]$ is the unique $\TT$-Poisson prime element in $R_k$ of the form
\begin{equation}\label{eq:y-kk}
y_k = x_k y_{p_\pi(k)} +c_k
\end{equation}
for some  $c_k \in \RR_{k-1}$ (here we set $y_{-\infty} = 1$). It follows from \eqref{eq:y-kk} 
that ${\bf y}$ is a  free generating set of $\CC(x_1, \ldots, x_n)$. To recall 
the mutation matrix $M$ from \cite{GY:Poi-CGL}  for the Goodearl-Yakimov seed $({\bf y}, M)$, 
note first that as $\{y_k, R_k\}_\pi
\subset y_k R_k$ for each $k \in [1, n]$,  the elements in ${\bf y}$ have log-canonical Poisson brackets
with respect to $\{\, , \, \}_\pi$. Consider the skew-symmetric $n \times n$ matrix 
\[
{\bf q} = (q_{j, k})_{j, k \in [1, n]}
\]
where 
$\{y_j, y_k\}_{\pi} = q_{j, k} y_jy_k$ for $j, k \in [1, n]$.
Let $\chi_{y_k} \in X^*(\TT)$ be the $\TT$-weight of $y_k$, and set
\[
\chi_{\bf y} = (\chi_{y_1}, \chi_{y_2}, \ldots, \chi_{y_n}) \in X^*(\TT)^n. 
\]
For a subset $J$ of $[1, n]$ and for $A = \ZZ$ or $\CC$, 
denote by ${\calM}_{n \times J}(A)$ 
the set of all matrices with entries in $A$ whose rows are indexed by $i \in [1, n]$ and whose columns are indexed 
by $j \in J$. 
Set now
\begin{equation}\label{eq:def-ex}
{\rm ex}  ={\rm lb}(\calS_\pi)=\{j \in [1, n]: s_\pi(j) \neq +\infty\} \subset [1, n],
\end{equation}
and let  $\Lambda^\prime \in \calM_{n \times {\rm ex}}(\CC)$ whose  $(j, j)$-entry for $j \in {\rm ex}$ is
$\lambda_j^\prime = -\lambda_{s_\pi(j)} = a_{\sL_\pi(j)}$ and all  other entries $0$.

\begin{theorem}\label{thm:GY-main}\cite[Theorem 11.1]{GY:Poi-CGL}
For any symmetric $\TT$-Poisson CGL extension $\pi$ of dimension $n$ and with notation as above, 
there is a unique $M \in \calM_{n \times {\rm ex}}(\ZZ)$
satisfying
\begin{equation}\label{eq:bfq-M}
{\bf q} M = \Lambda^\prime \hs \mbox{and} \hs \chi_{\bf y} M = 0.
\end{equation}
If, in addition, there exist positive integers $d_{\sL_\pi(j)}$ for each $j \in {\rm ex}$ level set $L$ such that
\[
\frac{a_{\sL_\pi(j)}}{d_{L_\pi(j)}} = \frac{a_{\sL_\pi(k)}}{d_{L_\pi(k)}}, \hs \forall \; j, k \in {\rm ex},
\]
then $M$ is skew-symmetrizable, and the cluster sub-algebra of $\CC(x_1, \ldots, x_n)$ defined by the (equivalence class of the) seed $({\bf y}, M)$ (with no frozen variable inverted) is the polynomial ring $\CC[x_1, \ldots, x_n]$. 
\end{theorem}

\begin{remark}\label{rk:other-GY}
{\rm \cite[Theorem 11.1]{GY:Poi-CGL} also contains statements on  
upper-cluster algebras and cluster algebras defined by $({\bf y}, M)$ with subsets of the set of 
frozen variables inverted.
It also describes a family of seeds in the equivalence class of $({\bf y}, M)$. We do not review these facts here
and we refer to \cite{GY:Poi-CGL} for details.   We  point out that the two conditions in \eqref{eq:bfq-M} guarantee  \cite{GSV:book} that every extended cluster ${\bf y}'$
mutation equivalent to ${\bf y}$ is 
 $\TT$-Poisson in the sense that all the elements in ${\bf y}'$ are $\TT$-weight vectors and have 
log-canonical Poisson brackets with respect to $\pi$. 
\hfill $\diamond$
}
\end{remark}

We are now ready to state our main result on the mutation matrix $M \in \calM_{n \times {\rm ex}}(\ZZ)$. 

Introduce the matrix $\calH_\pi  \in \calM_{n \times {\rm ex}}(\ZZ)$ whose columns are
 the elements in $\calS_\pi \subset \calS(\pi_0)  \subset \ZZone$.
  More precisely, define (see \autoref{pr:s-pi})
\begin{equation}\label{eq:calH-pi}
\calH_\pi = (\bth^{(j, s_\pi(j))})_{j \in {\rm ex}} \in \calM_{n \times {\rm ex}}(\ZZ).
\end{equation}
Recalling again that $\{e_1, \ldots, e_n\}$ is the standard basis of $\ZZ^n$ consisting of column vectors,
introduce also the $n \times n$ lower-triangular
matrix (we set $e_{+\infty} = 0$)
\begin{equation}\label{eq:E-pi}
E_\pi = (e_1 -e_{s_\pi(1)}, \; e_2 -e_{s_\pi(2)}, \; \ldots,\; e_n -e_{s_\pi(n)}), 
\end{equation}
Let $E_\pi^t = (E_\pi)^t$, the transpose of $E_t$. Recall from \autoref{def:a-LL} and \eqref{eq:a-SL-SL-1} the Cartan 
integers $a_{\sL', \sL}$ associated to interlacing level sets in $[1, n]$ defined by
the log-canonical part  $\pi_0$ of $\pi$.

\begin{theorem}\label{thm:main-M}
For any symmetric $\TT$-Poisson CGL extension $\pi$ of dimension $n$, 
the unique matrix $M \in \calM_{n \times {\rm ex}}(\ZZ)$ satisfying \eqref{eq:bfq-M} is given by
\[
M = E_\pi^t \calH_\pi.
\]
Writing $M = (m_{i, j})_{i \in [1, n], j \in {\rm ex}}$, for $i \in [1, n]$ and $j \in {\rm ex}$, we also have
\[
m_{i, j}  = \begin{cases} 1, & \hs i = p_\pi(j), \\
-1, & \hs i = s_\pi(j), \\
a_{\sL(i), \sL(j)}, & \hs i < j < s_\pi(i) < s_\pi(j), \\
-a_{\sL(i), \sL(j)}, & \hs j < i < s_\pi(j) < s_\pi(i) \;\; (\mbox{including when} \; s_\pi(i) = +\infty),\\
0, & \hs \mbox{otherwise}.\end{cases}
\]
\end{theorem}

\begin{proof} Recall from \eqref{eq:blam} the Poisson coefficient matrix 
$\blam$ of the coordinates $(x_1, \ldots, x_n)$ with respect to the log-canonical Poisson structure $\pi_0$. 
Recall also that $\bbeta = (\beta_1, \ldots, \beta_n) \in X^*(\TT)^n$, where $\beta_j$ is the $\TT$-weight of $x_j$
for $j \in [1, n]$. 
Let $F_\pi = (E_\pi)^{-1}$, and let $F_\pi^t = (F_\pi)^t$ be the transpose of $F_\pi$. 
It is proved in \cite[Proposition 5.8]{GY:Poi-CGL} that
${\bf q} = F_\pi \blam F_\pi^t$ and $\chi_{\bf y} = \bbeta F_\pi^t$.
Thus the two equation in \eqref{eq:bfq-M} become
\[
\blam F_\pi^t M = E_\pi \Lambda^\prime \hs \mbox{and} \hs \bbeta F_\pi^t M = 0.
\]
Set $M' = E_\pi^t\calH_\pi$ so that $F_\pi^t M' = \calH_\pi$. Using the second identity in
\eqref{eq:s-pi-Poi}, for each $j \in {\rm ex}$ we have
\[
\blam F_\pi^t M' e_j = \blam \calH_\pi e_j = \blam \bth^{(j, s_\pi(j))} = \lambda^\prime_{j}(e_j-e_{s_\pi(j)})
= E_\pi \Lambda^\prime e_j.
\]
Thus $\blam F_\pi^t M' = E_\pi \Lambda^\prime$. As $\calS_\pi \subset \ker \bbeta$, we also have
$\bbeta F_\pi^t M = \bbeta \calH_\pi = 0$. By the uniqueness of the solution $M$ to 
\eqref{eq:bfq-M}, we have $M = M' = E_\pi^t\calH_\pi$. As
\[
E_\pi^t = (e_1-e_{p_\pi(1)}, \, e_2-e_{p_\pi(2)}, \,\ldots,\, -e_{p_\pi(n)}+e_n),
\]
where we set $e_{-\infty} = 0$, the description of the entries of $M$ is now a direct consequence 
of \autoref{lm:bth-Cartan-integers}.
\end{proof}

\begin{example}\label{ex:beta-2p-02}
{\rm
Let $\pi = \pi(-p^2, -1)$ in  \autoref{ex:beta-2p-01}. We then have ${\rm ex}  = \{1, 2\}$, and
\[
M = E_\pi^t \calH_\pi = \left(\begin{array}{cccc} 1 & 0 & 0 & -1\\ 0 & 1 & -1 & 0\\
0 & 0 & 1 & 0\\ 
0 & 0 & 0 & 1\end{array}\right) \left(\begin{array}{cc} -1 & 0 \\ 2p & -1\\ 2p & -1 \\ -1 & 0\end{array}
\right) = \left(\begin{array}{cc} 0 & 0 \\ 0 & 0 \\ 2p & -1\\ -1 & 0\end{array}\right).
\]
On the other hand, one checks directly that the sequence ${\bf y}$ in \eqref{eq:by} is 
\[
{\bf y} = (y_1, y_2, y_3, y_4) = (x_1, \;x_2, \;x_2x_3-1, \;x_1x_4-(x_2x_3-1)^{2p}).
\]
The seed $({\bf y}, M)$ is thus of finite type $A_1 \times A_1$, with 
$x_3$ and $x_4$ as the only two new cluster variables other than those in the initial extended cluster ${\bf y}$.
\hfill $\diamond$
}
\end{example}

\section{Strongly symmetric  Poisson CGL extensions}\label{s:strongly-CGL}
Let again $\TT$ be a complex algebraic torus with Lie algebra $\t$ and character lattice 
$X^*(\TT) \subset \t^*$, and let 
$\TT$ act on $\CC^n$ via $\bbeta = (\beta_1, \ldots, \beta_n) \in X^*(\TT)^n \subset (\t^*)^n$. 
In this section, we show that if $\pi_0$ is a log-canonical Poisson structure on $\CC^n$ of {\it strongly
symmetric $\TT$-action type}, the elements in $\calS(\pi_0)$ can be  read off directly from 
a sequence of elements in $\t^*$ computed using 
reflection operators defined by the $\beta_j$'s. This class of examples are important, as we will see
in $\S$\ref{s:Cartan} that they contain 
the Lie theoretical examples as special cases.

\subsection{Strongly symmetric  Poisson CGL extensions and action pairs}\label{ss:strongly}

\begin{definition}\label{def:strong-sym-action-data}
{\rm
1) An $n$-dimensional $\TT$-action datum 
$(\bbeta = (\beta_1, \ldots, \beta_n), \bfh = (h_1, \ldots, h_n))$
is said to be {\it strongly symmetric} if 
$\beta_j(h_k) = \beta_k(h_j)$ for all $j, k \in [1, n]$, i.e., if it is 
of symmetric action type as in \autoref{def:sym-action-data}, where 
$h_j^\prime$ for each $j \in [1, n]$ can be chosen to be $-h_j$;

2) A log-canonical Poisson structure $\pi_0$ on $\CC^n$ is said to be 
of {\it strongly symmetric $\TT$-action type} if it is defined by a strongly
symmetric $\TT$-action datum;

3) A symmetric $\TT$-Poisson CGL as in \autoref{def:Poi-CGL} is said to be {\it strongly symmetric} if
its log-canonical term $\pi_0$ is of strongly symmetric $\TT$-action type.
}
\end{definition}

In view of the classification results in \autoref{thm:sym-CGL} and \autoref{cor:classification-CGL} on 
symmetric $\TT$-Poisson CGL extensions, $n$-dimensional 
strongly symmetric $\TT$-Poisson CGL extensions are classified, up to rescaling of CGL generators,
by log-canonical Poisson structures
$\pi_0$ on $\CC^n$ of strongly symmetric $\TT$-action type and by subsets of the set
$\calS(\pi_0)$ of all  non-zero $\CCsn$-weights in $\mH^2_{\pi_0}(\CC^n)^\TT$. 

\begin{lemma}\label{lm:lara}
The following are equivalent for a $\TT$-action datum $(\bbeta = (\beta_1, \ldots, \beta_n), \bfh = (h_1, \ldots, h_n))$:

1) $(\bbeta, {\bf h})$ is strongly symmetric;

2) there exists a bilinear symmetric form $\lara$ on $\t^*$ such that 
$\beta_j(h_k) = \la \beta_j, \beta_k\ra$ for all $j, k \in [1, n]$.
\end{lemma}

\begin{proof} 
Clearly 2) implies 1). Assuming 1), we now prove 2). 
Without loss of generality, assume that $\{\beta_1, \ldots, \beta_l\}$ is a  linearly independent subset of 
$\{\beta_1, \ldots, \beta_n\}$, and let $\bbeta^\prime = (\beta_1^\prime, \ldots, \beta_r^\prime)$ be any basis of $\t^*$ with $\beta_j^\prime = \beta_j$
for $j \in [1, l]$. One then has
\[
\bbeta = \bbeta^\prime \left(\begin{array}{cc} I_l & A \\ 0 & 0 \end{array}\right),
\]
where $I_l$ is the $l \times l$ identity matrix and $A$ is the $l \times (n-l)$ matrix such that
\begin{equation}\label{eq:beta-beta}
(\beta_{l+1}, \ldots, \beta_n) = (\beta_1, \ldots, \beta_l) A.
\end{equation}
Let $B = (\beta_j(h_k))_{j, k \in [1, l]}$, which is symmetric, and let $\lara$ be the symmetric bilinear form on $\t^*$ given by
\[
\langle \bbeta^\prime, \bbeta^\prime \rangle = \left(\begin{array}{cc} B & 0 \\ 0 & 0 \end{array}\right),
\]
where $\langle \bbeta^\prime, \bbeta^\prime \rangle 
= (\langle \beta_j^\prime, \beta_k^\prime\rangle)_{j, k \in [1, n]}$. Writing 
$\langle \bbeta, \bbeta \rangle = (\langle \beta_j, \beta_k\rangle)_{j, k \in [1, n]}$, one then has
\[
\langle \bbeta, \bbeta\rangle =\left(\begin{array}{cc} I_l & A \\ 0 & 0 \end{array}\right)^t
\left(\begin{array}{cc} B & 0 \\ 0 & 0 \end{array}\right)
\left(\begin{array}{cc} I_l & A \\ 0 & 0 \end{array}\right) =
\left(\begin{array}{cc} B & BA \\ A^tB &A^tBA \end{array}\right).
\]
Evaluating both sides of \eqref{eq:beta-beta} at every $h_j$ for $j \in [1, l]$, one has
\[
BA = \left(\begin{array}{ccc} \beta_{l+1}(h_1) & \cdots & \beta_n(h_1)\\
\cdots & \cdots & \cdots \\
\beta_{l+1}(h_l) & \cdots & \beta_n(h_l)\end{array}\right) =
\left(\begin{array}{ccc} \beta_{1}(h_{l+1}) & \cdots & \beta_1(h_{n+1})\\
\cdots & \cdots & \cdots \\
\beta_{l}(h_{l+1}) & \cdots & \beta_l(h_n)\end{array}\right).
\]
Similarly, $A^tB = (\beta_j(h_k))_{j \in [l+1, n], k \in [1, l]}$ and 
$A^tBA = (\beta_j(h_k))_{j, k \in [l+1, n]}$. Thus $\langle \beta_j, \beta_k\rangle = \beta_j(h_k)$ for all
$j, k \in [1, n]$.
\end{proof}

\begin{definition}\label{de:action-pair}
{\rm By an {\it $n$-dimensional $\TT$-action pair} we mean a pair $(\lara, \bbeta)$, where 
$\lara$ is a symmetric bilinear form on
$\t^*$ and $\bbeta = (\beta_1, \ldots, \beta_n) \in X^*(\TT)^n \subset (\t^*)^n$ is such that
$\langle \beta_j, \; \beta_j \rangle \neq 0$ for every $j \in [1, n]$.
For such a pair $(\lara, \bbeta)$, and with again $\TT$ acting on $\CC^n$ as in \eqref{eq:T-CCsn-1}, we call 
\begin{equation}\label{eq:pi0-beta}
\pi_0 = -\sum_{1\leq j < k \leq n} \langle \beta_j, \, \beta_k\rangle x_jx_k\frac{\partial}{\partial x_j} \wedge \frac{\partial}{\partial x_k}
\end{equation}
the ($\TT$-log-symplectic) log-canonical Poisson structure on $\mX = \CC^n$ defined by $(\lara, \bbeta)$. 
}
\end{definition}

\begin{remark}\label{rk:action-pair}
{\rm
In view of \autoref{lm:lara},
a log-canonical Poisson structure $\pi_0$ on $\CC^n$ is of  strongly symmetric $\TT$-action type 
(see \autoref{def:Poi-CGL}) if and only if 
$\pi_0$ is defined  by an $n$-dimensional  $\TT$-action pair. 
Consequently, a symmetric $\TT$-Poisson CGL  is strongly symmetric (see 
again \autoref{def:Poi-CGL}) if and only if 
its log-canonical term is defined by a $\TT$-action pair.
\hfill $\diamond$
}
\end{remark}

\subsection{The set $\calS(\pi_0)$ for $\pi_0$ defined by $\TT$-action pairs}\label{ss:lara-beta}
Let again $\TT$ be a complex  algebraic torus with Lie algebra $\t$ and character lattice $X^*(\TT)\subset \t^*$.
In $\S$\ref{ss:lara-beta}, we fix a $\TT$-action pair $(\lara, \bbeta=(\beta_1, \ldots, \beta_n))$ and 
let $\pi_0$ be defined by $(\lara, \bbeta)$
as  in \eqref{eq:pi0-beta}. Our main result of $\S$\ref{ss:lara-beta} is \autoref{pr:lara-beta}, which describes 
the set $\calS(\pi_0)$ of all non-zero $\CCsn$-weights in $\mH^2_{\pi_0}(\CC^n)^\TT$ 
in terms of a sequence $\bgamma = (\gamma_1, \ldots, \gamma_n) \in (\t^*)^n$ that is computed from $(\lara, \bbeta)$ using reflection operators on $\t^*$. 
To this end, for $\beta \in \t^*$ such that $\la \beta, \beta\ra \neq 0$, introduce the {\it Cartan number}
\begin{equation}\label{eq:Cartan-number}
a_{\beta, \xi}=\frac{2\langle \beta, \xi\rangle}{\langle \beta, \beta\rangle}\in \CC
\end{equation} 
for $\xi \in \t^*$, and define the reflection operator
$s_{\beta}$ on $\t^*$  by 
\begin{equation}\label{eq:s-beta}
s_{\beta}(\xi) = \xi - \frac{2\langle \beta, \xi \rangle}{\langle \beta, \beta\rangle} \beta
=\xi -a_{\beta, \xi} \beta, \hs \xi \in \t^*.
\end{equation}
One checks directly that $\langle s_{\beta}(\xi), s_{\beta}(\eta)\rangle = \la \xi, \eta\ra$ for all $\xi, \eta\in \t^*$. Define 
\begin{equation}\label{eq:gamma-j}
\bgamma = (\gamma_1, \, \gamma_2, \, \ldots, \,\gamma_n)\in (\t^*)^n,\hs \mbox{where}\hs 
\gamma_j = s_{\beta_1}\cdots s_{\beta_{j-1}} \beta_j \hs \mbox{for}\;\; j \in [1, n].
\end{equation}
Then $\la \gamma_j, \gamma_j\ra =\langle \beta_j, \beta_j\rangle  \neq 0$ for every $j \in [1, n]$, and, as $s_{\gamma_j} = 
s_{\beta_1} \cdots s_{\beta_{j-1}} s_{\beta_j} s_{\beta_{j-1}} \cdots s_{\beta_1}$, one has
\begin{equation}\label{eq:beta-j-gamma-j}
\beta_j = s_{\gamma_1} \cdots s_{\gamma_{j-1}} \gamma_j, \hs j \in [1, n].
\end{equation}

\begin{lemma}\label{lm:P-Q}
One has $\bbeta = \bgamma P$ and $\bgamma = \bbeta Q$,  where $Q = P^{-1}$ and 
\begin{equation}\label{eq:P-Q}
P = \left(\begin{array}{ccccc}
1 & a_{{}_{\beta_1, \beta_2}} & \cdots & a_{{}_{\beta_1, \beta_{n-1}}}  & a_{{}_{\beta_1, \beta_n}}\\
0 & 1 &\cdots & a_{{}_{\beta_2, \beta_{n-1}}}  & a_{{}_{\beta_2, \beta_n}}\\
\cdots & \cdots & \cdots & \cdots & \cdots \\
0 & 0 & \cdots & 1 & a_{{}_{\beta_{n-1}, \beta_n}}\\
0 & 0 & \cdots & 0 & 1\end{array}\right),\;\; \; \;
Q = \left(\begin{array}{ccccc} 1 & a_{\gamma_1, \gamma_2} & \cdots & a_{\gamma_1, \gamma_{n-1}}  & a_{\gamma_1, \gamma_n}\\
0 & 1 &\cdots & a_{\gamma_2, \gamma_{n-1}}  & a_{\gamma_2, \gamma_n}\\
\cdots & \cdots & \cdots & \cdots & \cdots \\
0 & 0 & \cdots & 1 & a_{\gamma_{n-1}, \gamma_n}\\
0 & 0 & \cdots & 0 & 1\end{array}\right).
\end{equation}
\end{lemma}

\begin{proof}
Clearly $\bbeta = \bgamma P$ holds for $n =1$. Assume that $n \geq 2$ and the statement holds for $n-1$.
Then 
$\beta_k = \gamma_k + \sum_{j=1}^{k-1}a_{\beta_j, \beta_k} \gamma_j$ for $k \in [1, n-1]$.
Setting 
\[
\gamma_{n-1}^\prime = s_{\gamma_{n-1}}\gamma_n = \gamma_n -a_{{}_{\gamma_{n-1}, \gamma_n}} \gamma_{n-1} =\gamma_n + a_{\beta_{n-1}, \beta_n} \gamma_{n-1}
\]
and noting that $\beta_n=s_{\gamma_1}\cdots s_{\gamma_{n-2}} \gamma_{n-1}^\prime$, 
we apply the induction assumption to 
the sequences $(\gamma_1, \ldots, \gamma_{n-2}, \gamma_{n-1}^\prime)$ and $(\beta_1, \ldots, \beta_{n-2}, \beta_n)$ and we get
\[
\beta_n = \gamma_{n-1}^\prime + \sum_{j=1}^{n-2} a_{\beta_j, \beta_n} \gamma_j=
\gamma_n + \sum_{j=1}^{n-1} a_{\beta_j, \beta_n}  \gamma_j.
\]
Thus $\bbeta =\bgamma P$. Switching $\bbeta$ and $\bgamma$, we have $\bgamma = \bbeta Q$ with $Q = P^{-1}$
given in \eqref{eq:P-Q}.
\end{proof}

For $\pi_0$ defined by $(\lara, \bbeta)$, 
we now turn to the set $\calS(\pi_0) \subset \calW_2\backslash \{0\}$ of all non-zero $\CCsn$-weights in 
$\mH^2_{\pi_0}(\CC^n)^\TT$. 
Introduce the map $[1, n] \mapsto [2, n] \cup \{+\infty\}$ by 
\begin{equation}\label{eq:j-plus-def}
j \longmapsto j^+ = \begin{cases} +\infty, & \;\;  \mbox{if} \;\; \{j' \in [j+1, n]: \gamma_{j'} = \gamma_j\} = \emptyset,\\
{\rm min}\{j' \in [j+1, n]: \gamma_{j'} = \gamma_j\}, & \;\; \mbox{otherwise}.\end{cases}
\end{equation}
Setting $e_{+\infty} = 0$, we  also define the lower triangular $n \times n$ matrix 
\begin{equation}\label{eq:Emax-def}
E_+ = (e_{1}-e_{1^+}, \; e_2-e_{2^+}, \; \ldots, e_n-e_{n^+}).
\end{equation}
For future use, note that for $j \in [1, n]$, the $j^{\rm th}$-row of 
 the inverse of  $E_+$ is given by
\begin{equation}\label{eq:F-row}
e_j^t E_+^{-1} = \sum_{i \in [1, j], \gamma_i = \gamma_j} e_i^t,
\end{equation} 
Consider now the matrix $QE_+$, where $Q = P^{-1}$ is given in \eqref{eq:P-Q}.
Set 
\begin{equation}\label{eq:calJ}
{\calJ} = \{j \in [1, n]: \; j^+\neq +\infty\}.
\end{equation}
Note then that for $j \in \calJ$, the 
$j^{\rm th}$ column of  the matrix $QE_+$ is given by 
\begin{equation}\label{eq:QE-j}
QE_+e_j = Qe_j-Qe_{j^+} 
= (0, \, \ldots, \, 0, \, -1, \, -a_{\gamma_{j+1}, \gamma_j}, \,\ldots,\,
-a_{\gamma_{j^+-1}, \gamma_j}, \, -1, \, 0, \, \ldots, \, 0)^t.
\end{equation}
Introduce 
\begin{equation}\label{eq:calJ-plus}
\calJ_{{\rm int}} = \{j \in {\calJ}: \;
-a_{\gamma_i, \gamma_j}  \in \ZZ_{\geq 0}\;\;\forall\;\;
i \in [j+1, \;j^+-1]\} \subset {\calJ}.
\end{equation}

\begin{lemma}\label{lm:J-0-int}
For $j \in [1, n]$, one has $j \in \calJ_{{\rm int}}$ if and only if the $j^{\rm th}$-column
 of the matrix  $QE_+$ is in $\calW_2$ and is negatively bordered on both sides.
\end{lemma}

\begin{proof}
Let  $j \in [1, n]$. If $j \notin \calJ$, then $QE_+e_j= Qe_j$ has $1$ as the last non-zero entry, so it is not negatively bordered on both sides. Suppose that $j \in \calJ$. By \eqref{eq:QE-j}, 
$QE_+e_j \in \calW_2$ and is negatively bordered on both sides if and only if $j \in \calJ_{{\rm int}}$. 
\end{proof}

Recall
from \eqref{eq:E-plus-S} the bijection 
$\calE^+(\pi_0) \ni (j, k) \mapsto \bth^{(j, k)} \in \calS(\pi_0)$.

\begin{notation}\label{nota:Col-J-M}
{\rm In what follows, for an $n \times n$ matrix $M$ and $J \subset [1, n]$, 
denote by 
\begin{equation}\label{eq:Col-J-M}
{\rm Col}_J(M) = \{Me_j: j \in J\}, 
\end{equation}
the set of columns of $M$ indexed by $j \in J$, and call ${\rm Col}_J(M)$ the {\it $J$-column set of $M$}.
\hfill $\diamond$
}
\end{notation}

\begin{proposition}\label{pr:lara-beta}
For any $n$-dimensional $\TT$-action pair  $(\lara, \bbeta)$ and the log-canonical Poisson structure
$\pi_0$ defined by $(\lara, \bbeta)$ in \eqref{eq:pi0-beta}, one has
$\calE^+(\pi_0) = \{(j,j^+): j \in \calJ_{{\rm int}}\}$, and 
\[
\calS(\pi_0) = {\rm Col}_{\calJ_{{\rm int}}}(Q E_+),
\]
where for $j \in \calJ_{{\rm int}}$,  $\bth^{(j, j^+)} = QE_+e_j= Qe_j-Qe_{j^+} \in \calS(\pi_0)$ is given in 
\eqref{eq:QE-j}. 
\end{proposition}

\begin{proof} For $\beta \in \t^*$, let $\beta^\# \in \t$ be such that 
$\beta'(\beta^\#) =\langle \beta', \beta\rangle$ for all $\beta' \in \t^*$. Setting $h_j = \beta_j^\#$ 
and ${\bf h} = (h_1,\ldots, h_n)$, the log-canonical Poisson structure $\pi_0$ is then defined by the
strongly symmetric $\TT$-action datum $(\bbeta, {\bf h})$. The corresponding 
matrix $\bnu$  in \eqref{eq:bnu} is then given by
\[
\bnu =  \left(\begin{array}{cccc} 
\langle \beta_1, \beta_1\rangle & 2\langle \beta_1, \beta_2\rangle & \cdots &2\langle \beta_1, \beta_n\rangle \\
0 & \langle \beta_2, \beta_2\rangle & \cdots &2\langle \beta_2, \beta_n\rangle \\
\cdots & \cdots & \cdots & \cdots \\
0& 0 & \cdots &\langle \beta_n, \beta_n\rangle 
\end{array}\right) = \left(\begin{array}{cccc} 
\lambda_1 & 2\langle \beta_1, \beta_2\rangle & \cdots &2\langle \beta_1, \beta_n\rangle \\
0 & \lambda_2 & \cdots &2\langle \beta_2, \beta_n\rangle \\
\cdots & \cdots & \cdots & \cdots \\
0& 0 & \cdots &\lambda_n
\end{array}\right).
\]
Let $\Lambda$ be the diagonal matrix whose $(j, j)$-entry for 
$j \in [1, n]$ is $\lambda_j =\la \beta_j, \beta_j\ra=\la \gamma_j, \gamma_j\ra$.
Then, by definition, $\bnu = \Lambda P$ with $P$ given in \eqref{eq:P-Q}. Thus
$\bmu^{-1} = Q\Lambda^{-1}$. 
By \autoref{lm:P-Q}, 
\[
\bbeta \bnu^{-1} = \bbeta Q \Lambda^{-1}=\bgamma \Lambda^{-1}.
\]
By \autoref{pr:Spi0-action}, the set 
$\calE^+(\pi_0)$ consists of all pairs $(j, k)$ of 
$1\leq j < k \leq n$ such that 
$\bgamma \Lambda^{-1} e_j = \bgamma \Lambda^{-1} e_k$ and 
$\lambda_k \left(\bnu^{-1}e_j - \bnu^{-1}e_k\right) \in \calW_2$, or,
equivalently,
\begin{equation}\label{eq:E-plus-strong}
{\lambda_j}^{-1}{\gamma_j} = {\lambda_k}^{-1}{\gamma_k} \hs \mbox{and} \hs 
{\lambda_j}^{-1}{\lambda_k} Qe_j - Qe_k \in \calW_2.
\end{equation}
Note that ${\lambda_j}^{-1}{\gamma_j} = {\lambda_k}^{-1}{\gamma_k}$ if and only if
$\gamma_j = \gamma_k$. By the formula for $Q$ in \eqref{eq:P-Q} and the fact 
that $a_{\gamma_i, \gamma_i} = 2$ for every $i \in [1, n]$, one sees that \eqref{eq:E-plus-strong} 
is equivalent to 
\begin{equation}\label{eq:gamma-jk}
\gamma_j = \gamma_k \hs \mbox{and} \hs -a_{\gamma_i, \gamma_j}  \in \ZZ_{\geq 0},\;\; 
\forall\;\; i \in [j+1, k-1].
\end{equation}
Again by the fact that $a_{\gamma_i, \gamma_i} = 2$ for every $i \in [1, n]$, 
\eqref{eq:gamma-jk} is equivalent to  $k = j^+$ and $j \in \calJ_{{\rm int}}$. 
Thus $\calE^+(\pi_0) = \{(j, j^+): j \in \calJ_{{\rm int}}\}$. Moreover, 
for $j \in \calJ_{{\rm int}}$, as $\gamma_j = \gamma_{j^+}$, 
one has $\lambda_j = \lambda_{j^+}$, so by 
\autoref{pr:Spi0-action} again, one has
\begin{align*}
\bth^{(j, j^+)} &= \lambda_{j^+} (\bnu^{-1}e_j - \bnu^{-1}e_{j^+}) =
\lambda_{j} \bnu^{-1}e_j - \lambda_{j^+}\bnu^{-1}e_{j^+}=\bnu^{-1}\Lambda e_j -\bnu^{-1}
\Lambda e_{j^+}\\
&= Qe_j - Qe_{j^+} = QE_+ e_j.
\end{align*}
This finishes the proof of \autoref{pr:lara-beta}.
\end{proof}

\begin{remark}\label{rk:a-LL-J}
{\rm
By \autoref{pr:lara-beta}, the Cartan integers associated to  $\pi_0$ as defined in \autoref{def:a-LL}  are precisely the integers $a_{\gamma_i, \gamma_j}$,
for $j \in \calJ_{{\rm int}}$ and $i \in [j+1, j^+-1]$. As $\calJ_{{\rm int}} = {\rm lb}(\calS(\pi_0))$, 
by \eqref{eq:def-ex} the set $\calJ_{{\rm int}}$ is also the exchange set for the Goodearl-Yakimov seed
$({\bf y}, M)$ associated to maximal normalized admissible algebraic Poisson deformation of $\pi_0$
regarded as a symmetric Poisson CGL extension (see \autoref{defn:normalized}). 
\hfill $\diamond$
}
\end{remark}

\begin{example}\label{ex:beta-2p-03}
{\rm 
For every $p \in \ZZ\backslash\{0\}$, the log-canonical Poisson structure $\pi_0$ in \eqref{eq:pi0-p} in 
\autoref{ex:beta-2p-00} is defined by the $\TT$-action pair $(\lara, \bbeta = (p\beta, \beta, -\beta, -p\beta))$, where $\lara$ is any  symmetric bilinear form 
on $\t^*$ such that $\langle \beta, \beta\rangle = 1$.  We have
$\bgamma = (p\beta, -\beta, -\beta, p\beta)$ and 
\[
Q = \!\left(\!\begin{array}{cccc} 1 & -\frac{2}{p} & -\frac{2}{p} & 2\\ 0 & 1 & 2 & -2p\\ 0 & 0 & 1 & -2p \\ 0 & 0 & 0 & 1\end{array}\!\right).
\]
For $p = -1$,  we have $\calJ_{{\rm int}} = {\calJ} = \{1, 2, 3\}$ with $1^+ = 2$, $2^+ = 3$, and $3^+ = 4$, so  
\[
QE_+ =Q \left(\begin{array}{cccc} 1 & 0 & 0 & 0\\ -1 & 1 & 0 & 0 \\ 0 & -1 & 1 & 0\\ 0& 0 & -1 & 1\end{array}\right)=\left(\!\begin{array}{cccc} -1 & 0 & 0 & 2\\ -1 & -1 & 0 & 2 \\ 0 & -1 & -1 & 2\\ 0& 0 & -1 & 1\end{array}\right);
\]
Assume that  $p \neq -1$. We then have ${\calJ} = \{1, 2\}$ with
 $1^+ = 4$ and $2^+ = 3$, and 
\[
Q E_+ = Q\left(\begin{array}{cccc} 1 & 0 & 0 & 0\\ 0 & 1 & 0 & 0 \\ 0 & -1 & 1 & 0\\ -1& 0 & 0 & 1\end{array}\right)
=\left(\begin{array}{cccc} -1 & 0 & -\frac{2}{p} & 2\\
2p & -1 & 2 & -2p\\ 2 p & -1 & 1 & -2p\\ -1 & 0 & 0 &1\end{array}\right).
\]
Since $-a_{{}_{-\beta, p\beta}} = 2p$,
if  $p <0$ and $p \neq -1$,  then $\calJ_{{\rm int}} = \{2\}$, and if  $p > 0$, then $\calJ_{{\rm int}} = {\calJ} =\{1, 2\}$.
In all the cases, the conclusions on $\calS(\pi_0)$ via \autoref{pr:lara-beta} coincide with that stated in 
\autoref{ex:beta-2p-00}.
\hfill $\diamond$
}
\end{example}

\subsection{Distinguished $\TT$-action pairs}\label{ss:distinguished}
Continuing with the setting and notation  in $\S$\ref{ss:strongly},
for an $n$-dimensional $\TT$-action pair $(\lara, \bbeta = (\beta_1, \ldots, \beta_n))$, recall 
that $\ker \bbeta$ denotes the kernel of the 
linear map 
\[
\CC^n \longrightarrow \t^*, \;\; \bfw =(\bfw_1, \ldots, \bfw_n)^t \longmapsto \bbeta \bfw = \bfw_1\beta_1+ \cdots + 
\bfw_n \beta_n.
\]
Let again $\bgamma = (\gamma_1, \ldots, \gamma_n)$ be given in \eqref{eq:gamma-j}.
With $Q = P^{-1}$ and $E_+$ given respectively in \eqref{eq:P-Q} and \eqref{eq:Emax-def} 
and ${\calJ} \subset [1, n]$  in \eqref{eq:calJ}, 
for every $j \in {\calJ}$ we have 
by  \autoref{pr:lara-beta} and \autoref{lm:P-Q} that
\[
\bbeta QE_+e_j = \bgamma E_+e_j = \bgamma(e_j-e_{j^+}) = \gamma_j-\gamma_{j^+} = 0. 
\]
Let $\pi_0$ be defined in \eqref{eq:pi0-beta}. By \autoref{pr:lara-beta} we have 
\begin{equation}\label{eq:three-sets}
\calS(\pi_0) = {\rm Col}_{\calJ_{{\rm int}}}(Q E_+) \subset  {\rm Col}_{{\calJ}}(Q E_+) 
\subset \ker \bbeta.
\end{equation}
As  $QE_+$ is invertible, its column set ${\rm Col}_{\calJ}(QE_+)$ is a $\CC$-linearly independent
subset of $\ker \bbeta$.
We now give a necessary and sufficient condition on $(\lara, \bbeta)$ for ${\rm Col}_{\calJ}(QE_+)$ to be
a $\CC$-basis of $\ker \bbeta$. Let
\begin{equation}\label{eq:supp-gamma}
{\rm Supp}(\bgamma) = \{\gamma_j: j \in [1, n]\}, 
\end{equation}
i.e, the set of the distinct elements in $\t^*$ that appear in the sequence $\bgamma =(\gamma_1, \ldots, \gamma_n)$.

\begin{definition}\label{def:disting}
{\rm An $n$-dimensional  $\TT$-action pair $(\lara, \bbeta)$ is said to be {\it distinguished} if 
${\rm Supp}(\bgamma)$ is a $\CC$-linearly independent subset of $\t^*$.
}
\end{definition}

\begin{example}\label{ex:beta-2p-04}
{\rm
The $\TT$-action pair $(\lara, \bbeta =(p\beta, \beta, -\beta, -p\beta))$ in \autoref{ex:beta-2p-03} is  distinguished if and only if $p = -1$. 
\hfill $\diamond$
}
\end{example}

In what follows, 
for a $p \times q$ matrix $M$ and a subset $J \subset [1, q]$ 
we denote by $M_{p \times J}$ the sub-matrix
of $M$ formed by the columns of $M$ indexed by $j \in J$. Similarly, for $K \subset [1, p]$, 
denote by $M_{K \times q}$ the 
sub-matrix of $M$ formed by the rows of $M$ indexed by $k \in K$.

\begin{lemma}\label{lm:disting-ker-beta}
For an $n$-dimensional
$\TT$-action pair $(\lara, \bbeta)$, the set
${\rm Col}_{\calJ}(QE_+)$ 
is a $\CC$-basis of $\ker \bbeta$ if and only if  $(\lara, \bbeta)$ is distinguished, 
and in such a case, for every $\bfw \in \ker \bbeta$ one has 
\begin{equation}\label{eq:bfw-bth}
\bfw =(QE_+)_{n \times {\calJ}} \left((E_+^{-1}P)_{{\calJ} \times n}\bfw\right).
\end{equation}
\end{lemma}

\begin{proof} We already know  that ${\rm Col}_{\calJ}(QE_+)$ is a  
$\CC$-linearly independent subset of $\ker \bbeta$ of cardinality $|{\calJ}|$.

Let $\calJ^c = \{j \in [1, n]: j^+=+\infty\}= [1, n]\backslash {\calJ}$.
As $I_n  = (QE_+) (E_+^{-1}P)$, where  
$I_n$ is the $n \times n$ identity matrix, for  every $\bfw \in \CC^n$ we have
\begin{equation}\label{eq:bfw-QQ}
\bfw =(QE_+)_{n \times {\calJ}} (E_+^{-1}P)_{{\calJ} \times n}\bfw  
+ (QE_+)_{n \times \calJ^c} (E_+^{-1}P)_{\calJ^c \times n} \bfw.
\end{equation}
Note that $\bgamma$, regarded as a $1 \times n$ matrix with entries in $\t^*$, has the sub-matrix
$\bgamma_{1 \times {\calJ}^c}$ whose entries are precisely the elements in ${\rm Supp}(\bgamma)$. By the definition of $\calJ^c$,  we have
\begin{equation}\label{eq:QEJw}
\bbeta (QE_+)_{n \times \calJ^c} = \bbeta Q (E_+)_{n \times \calJ^c} = \bgamma
(E_+)_{n \times \calJ^c} = \bgamma_{1 \times \calJ^c}.
\end{equation}
It thus follows from \eqref{eq:bfw-QQ} that for every $\bfw \in \CC^n$, we have
\[
\bbeta \bfw = \bbeta (QE_+)_{n \times \calJ^c} (E_+^{-1}P)_{\calJ^c \times n} \bfw = 
\bgamma_{1 \times \calJ^c}(E_+^{-1}P)_{\calJ^c \times n} \bfw.
\]

Assume first that $(\lara, \bbeta)$ is distinguished. Then for every $\bfw \in \CC^n$,
\[
\bbeta \bfw =0\hs \Longleftrightarrow \hs 
(E_+^{-1}P)_{\calJ^c \times n} \bfw = 0.
\]
As the rank of the matrix 
$(E_+^{-1}P)_{\calJ^c \times n}$ is $|\calJ^c| = n-|{\calJ}|$,  
we have $\dim \ker \bbeta = |{\calJ}|$.
Thus ${\rm Col}_{\calJ}(QE_+)$ 
is a $\CC$-basis of $\ker \bbeta$. It also follows from \eqref{eq:bfw-QQ}
that \eqref{eq:bfw-bth} holds for $\bfw \in \ker \bbeta$.

Conversely, assume that ${\rm Col}_{\calJ}(QE_+)$ 
is a $\CC$-basis of $\ker \bbeta$. Let ${\bf u}$ be a column vector with entries in $\CC$ and 
indexed by $j \in \calJ^c$ and suppose that  $\bgamma_{1 \times \calJ^c} {\bf u} = 0$. Then 
$\bbeta (QE_+)_{n \times \calJ^c} {\bf u} = 0$ by \eqref{eq:QEJw}, so 
$(QE_+)_{n \times \calJ^c} {\bf u}$ is a $\CC$-linear combination of 
columns of $(QE_+)_{n \times {\calJ}}$, possible only if ${\bf u} = 0$. Hence 
${\rm Supp}(\bgamma)$ is a linearly independent subset of $\t^*$, i.e., the $\TT$-action pair
$(\lara, \bbeta)$ is distinguished. 
\end{proof}

For a distinguished $\TT$-action pair $(\lara, \bbeta)$, we now give another formula for
elements in $\ker \bbeta$ as linear combinations of columns of $(QE_+)_{n \times {\calJ}}$.
To this end, for $\eta \in \sum_{\gamma \in {\rm Supp}(\bgamma)} \CC \gamma \subset \t^*$, write 
\[
\eta = \sum_{\gamma \in {\rm Supp}(\bgamma)} [\eta:\gamma].
\]
For $\bfw =(\bfw_1, \ldots, \bfw_n)^t\in \CC^n$, set $\eta^{(\bfw)}_0 = 0$ and 
define $\eta^{(\bfw)} = (\eta^{(\bfw)}_1, \ldots, \eta^{(\bfw)}_n) \in (\t^*)^n$ by
\begin{equation}\label{eq:eta-j-inductive}
\eta^{(\bfw)}_j= s_{\gamma_j} \eta_{j-1}-\bfw_j \gamma_j, \hs j \in [1, n]. 
\end{equation}
An induction argument shows that we also have
\begin{equation}\label{eq:eta-all}
\eta^{(\bfw)}_j = s_{\gamma_j}\cdots s_{\gamma_1} (\bfw_1\beta_1 + \cdots + \bfw_j \beta_j), \hs j \in [1, n].
\end{equation}

\begin{lemma}\label{lm:etas}
For a distinguished $\TT$-action pair $(\lara, \bbeta)$ and for any $\bfw \in \ker \bbeta$, one has
\[
\bfw = \sum_{j \in {\calJ}} [\eta^{(\bfw)}_j:\gamma_j](QE_+e_j).
\]
\end{lemma}

\begin{proof}
By \autoref{lm:disting-ker-beta}, it suffices to 
 prove the following identity for every $\bfw \in \ker \bbeta$: 
\begin{equation}\label{eq:EPw}
E_+^{-1} P \bfw = \left([\eta^{(\bfw)}_1:\gamma_1], \, \ldots, \, [\eta^{(\bfw)}_n:\gamma_n]\right)^t.
\end{equation}
First let $i \in [1, n]$ be arbitrary. Using \eqref{eq:P-Q} and 
the facts that $\bfw \in \ker \bbeta$ and $\langle \beta_i, \beta_i \rangle = \langle \gamma_i, \gamma_i \rangle$, one sees that the 
$i^{\rm th}$-row of (the column vector) $P\bfw$ is given by
\begin{align}\nonumber
e_i^t P\bfw &= \bfw_i + \sum_{l=i+1}^n a_{\beta_i, \beta_l} \bfw_l =\bfw_i +
\frac{2}{\langle \beta_i, \beta_i\rangle}\left\langle  \beta_i, \, \sum_{l=i+1}^n \bfw_l \beta_l\right\rangle\\
\nonumber& = -\bfw_i + 
\frac{2}{\langle \beta_i, \beta_i\rangle}\left\langle  \beta_i, \; \sum_{l=i}^n \bfw_l \beta_l\right\rangle
=-\bfw_i- 
\frac{2}{\langle \beta_i, \beta_i\rangle}\left\langle  \beta_i, \; \sum_{l=1}^{i-1} \bfw_l \beta_l\right\rangle\\
\nonumber& =-\bfw_i- 
\frac{2}{\langle \gamma_i, \gamma_i\rangle}\left\langle  \beta_i, \; \sum_{l=1}^{i-1} \bfw_l \beta_l\right\rangle
= -\bfw_i-
\frac{2}{\langle \gamma_i, \gamma_i\rangle}\left\langle s_{\gamma_1}\cdots s_{\gamma_{i-1}} \gamma_i, \; 
\sum_{l=1}^{i-1} \bfw_l \beta_l\right\rangle\\
\nonumber& = -\bfw_i -\frac{2}{\langle \gamma_i, \gamma_i\rangle}\left\langle \gamma_i, \; s_{\gamma_{i-1}}\cdots s_{\gamma_1} 
\left(\sum_{l=1}^{i-1} \bfw_l \beta_l\right)\right\rangle=
-\bfw_i -\frac{2}{\langle \gamma_i, \gamma_i\rangle}\left\langle \gamma_i, \; \eta^{(\bfw)}_{i-1}\right\rangle\\
\label{eq:iPw} & =-\bfw_i - a_{\gamma_i, \eta^{(\bfw)}_{i-1}}.
\end{align}
By \eqref{eq:F-row} again, \eqref{eq:EPw} is equivalent to 
\begin{equation}\label{eq:EPw-j}
\sum_{i \in [1, j], \, \gamma_i = \gamma_j}\left(-\bfw_i - a_{\gamma_j, \eta^{(\bfw)}_{i-1}}\right) = 
[\eta^{(\bfw)}_j:\gamma_j], \hs \forall \; j \in [1, n].
\end{equation}
Let $j \in [1, n]$.
By the definition of $\eta^{(\bfw)}_j$ in \eqref{eq:eta-j-inductive}, 
\[
[\eta^{(\bfw)}_j:\gamma_j] = \left[s_{\gamma_j} \eta^{(\bfw)}_{j-1}-\bfw_j \gamma_j:\gamma_j\right]
=\left[\eta^{(\bfw)}_{j-1}-a_{\gamma_j, \eta^{(\bfw)}_{j-1}}\gamma_j-\bfw_j \gamma_j:\gamma_j\right]
=\left[\eta^{(\bfw)}_{j-1}:\gamma_j\right]
-\bfw_j-a_{\gamma_j, \eta^{(\bfw)}_{j-1}}.
\]
If $\gamma_i \neq \gamma_j$ for all $i \in [1, j-1]$, then $\left[\eta^{(\bfw)}_{j-1}:\gamma_j\right]=0$, so
\eqref{eq:EPw-j} holds for $j$. Otherwise, let 
\begin{equation}\label{eq:j-minus}
j^- = {\rm max}\{i \in [1, j-1], \, \gamma_i = \gamma_j\}.
\end{equation}
Then $\left[\eta^{(\bfw)}_{j-1}:\gamma_j\right] = \left[\eta^{(\bfw)}_{j^-}:\gamma_j\right] =
 \left[\eta^{(\bfw)}_{j^-}:\gamma_{j^-}\right]$, so
\begin{equation}\label{eq:eta-j-ind}
[\eta^{(\bfw)}_j:\gamma_j] = \left[\eta^{(\bfw)}_{j^-}:\gamma_{j^-}\right] 
-\bfw_j-a_{\gamma_j, \eta^{(\bfw)}_{j-1}}.
\end{equation}
By induction, 
 \eqref{eq:EPw-j} holds for $j$.
This finishes the proof of \eqref{eq:EPw} and thus of that of \autoref{lm:disting-ker-beta}.
\end{proof}

\begin{example}\label{ex:hideous} 
{\rm
Let 
$\bfw = (0, \ldots, 0, -1, \bfw_{j+1}, \ldots, \bfw_{k-1}, -1, 0, \ldots, 0)^t\in \CC^n$, 
where the two $-1$ entries are at the $j^{\rm th}$ and the $k^{\rm th}$ places. Then 
$\bfw \in \ker \bbeta$ if and only if
\begin{equation}\label{eq:bff-bth}
{\bfw} = \sum_{l \in [j, k-1], \,l^+\leq k} [\eta^{(\bfw)}_l:\gamma_l]\, (Q E_+ e_l).
\end{equation}
Indeed, by \eqref{eq:eta-j-inductive} we have
$\eta^{(\bfw)} = \left(0, \;\ldots, \;0, \;\eta^{(\bfw)}_j, \; \eta^{(\bfw)}_{j+1}, \; \ldots, \; 
\eta^{(\bfw)}_{k-1}, \;\eta^{(\bfw)}_k,\; \ldots, \;\eta^{(\bfw)}_n\right)$,
where $\eta^{(\bfw)}_j = \gamma_j$, and $\eta^{(\bfw)}_l = s_{\gamma_l} \eta^{(\bfw)}_{l-1} - \bfw_l \gamma_l$ 
for $l \in [j+1,n]$. By \eqref{eq:eta-all}, $\bfw \in \ker \bbeta$ if and only if 
$\eta^{(\bfw)}_{k-1} = \gamma_k$,  and in such a 
case, $\eta^{(\bfw)}_k = \eta^{(\bfw)}_{k+1}=\cdots =\eta^{(\bfw)}_n = 0$.  
Assume that $\bfw \in \ker \bbeta$. Then by \autoref{lm:etas},
\[
{\bfw} =\sum_{l \in [j, k-1] \cap {\calJ}} [\eta^{(\bfw)}_l:\gamma_l]\,(Q E_+ e_l).
\]
If $l \in [j, k-1]\cap {\calJ}$ is such that $l^+ \geq k+1$, then 
$\left[\eta^{(\bfw)}_l:\gamma_l\right] = \left[\eta^{(\bfw)}_{l^+}:\gamma_{l^+}\right]+\bfw_{l^+} +
a_{\gamma_l,\eta^{(\bfw)}_{l^+-1}} = 0$ by \eqref{eq:eta-j-ind}, so we have \eqref{eq:bff-bth}.
Assuming 
\eqref{eq:bff-bth}, it follows from $QE_+e_l \in \ker \bbeta$ for every $l \in \calJ$ that 
$\bfw \in \ker \bbeta$.
\hfill $\diamond$
}
\end{example}

\begin{definition}\label{def:int}
{\rm A  $\TT$-action pair $(\lara, \bbeta)$ is said to be {\it weakly integral} if $\calJ_{{\rm int}} = {\calJ}$, i.e., if 
\[
-a_{\gamma_i, \gamma_j} \in \ZZ_{\geq 0} \hs \forall\; j\in {\calJ} \; \mbox{and}\; i \in [j+1, \;j^+-1].
\]
}
\end{definition}

For $\pi_0$ defined by 
a weakly integral $\TT$-action pair, it follows from 
the \autoref{pr:lara-beta} that 
\[
\calS(\pi_0) = {\rm Col}_{\calJ}(QE_+).
\]

\begin{proposition}\label{pr:int-basis} For $\pi_0$ defined by 
a  distinguished and weakly integral $\TT$-action pair $(\lara, \bbeta)$ as in \eqref{eq:pi0-beta}, the set 
$\calS(\pi_0)$ is both a $\CC$-basis of $\ker \bbeta \subset \CC^n$ and a $\ZZ$-basis of 
$(\ker \bbeta)\cap \ZZ^n$.
\end{proposition}

\begin{proof} As $(\lara, \bbeta)$ is weakly integral, 
one has ${\calJ} = \calJ_{{\rm int}}$.
By \autoref{pr:lara-beta}, $\calS(\pi_0)$ coincides with  of the $\calJ$-column set of 
the matrix $QE_+\in SL(n, \ZZ)$, so by \autoref{lm:disting-ker-beta} $\calS(\pi_0)$
is both a $\CC$-basis of $\ker \bbeta \subset \CC^n$ and a $\ZZ$-basis of 
$(\ker \bbeta)\cap \ZZ^n$.
\end{proof}

\begin{definition}\label{def:integral}
{\rm A $\TT$-action pair $(\lara, \bbeta)$ is said to be {\it integral} if 
$\langle \gamma, \gamma \rangle \in \QQ_{>0}$ for all $\gamma \in {\rm Supp}(\bgamma)$ and if 
$-a_{\gamma, \gamma'}\in \ZZ_{\geq 0}$ for all 
$\gamma, \gamma' \in {\rm Supp}(\bgamma)$ such that 
$\gamma \neq \gamma'$.
}
\end{definition}

We will show in $\S$\ref{ss:pair-Cartan} that an action pair is distinguished and integral if
and only if it {\it factors through} (see \autoref{defn:factor-through}) an action pair 
associated to a symmetrizable generalized Cartan matrix. 

\section{Symmetric Poisson CGL extensions of Cartan type}\label{s:Cartan}
\subsection{Action pairs of Cartan type}\label{ss:pair-Cartan}
Recall that a {\it symmetrizable generalized Cartan matrix of size $r$} is an integral  $r \times r$ matrix $A = (a_{i,j})$ such that

(1) $a_{i,i} = 2$ for every $i \in [1, r]$;

(2) $a_{i, j} \leq 0$ for all $i, j \in [1, r]$ with $i \neq j$;

(3) there exist $d_i \in \QQ_{>0}$ for each $i \in [1, n]$ such that 
$d_i a_{i, j} = d_j a_{j, i}$ for all $i, j \in [1, r]$.

\noindent
Fix now a symmetrizable generalized Cartan matrix $A = (a_{i, j})$ of size $r$ and 
a symmetrizer $(d_i)_{i \in [1, r]}$. Fix also 
a realization \cite[Definition 1.1.2]{Kumar:Kac-Moody} of $A$ , i.e., a vector space $\h$ of dimension $2r - {\rm rank}(A)$ together with linearly independent subsets
\[
\{\alpha_1^\vee, \ldots, \alpha_r^\vee\} \subset \h \hs \mbox{and} \hs \{\alpha_1, \ldots, \alpha_r\}\subset \h^*, 
\]
such that 
$(\alpha_i^\vee, \alpha_j) = a_{i, j}$ for all $i, j \in [1, r]$, where $(\, , \, )$ is the $\CC$-linear
pairing between $\h$ and $\h^*$. 
Consider the  linear map $\underline{\alpha} = (\alpha_1, \ldots, \alpha_r): \h \to \CC^r$ and 
$\ker \underline{\alpha} = \bigcap_{i=1}^r \ker \alpha_i \subset \h$. Since $\underline{\alpha}$ is surjective and 
since $\dim \h = 2r - {\rm rank}(A)$, we have 
$\dim \ker \underline{\alpha} =
r-{\rm rank}(A)$ which is also the dimension of the kernel of the restriction of $\underline{\alpha}$ to 
$\sum_{i=1}^r \CC \alpha_i^\vee$. 
It follows that
$\ker \underline{\alpha} \subset {\sum_{i=1}^r \CC \alpha_i^\vee}$. Set 
\[
\t_A = \h /\ker \underline{\alpha}
\]
and note that $\dim \t_A = r$. For notational
simplicity, the image of each $\alpha_i^\vee$ in $\t_A$ will still be denoted by $\alpha_i^\vee$. 
Note that the pairing 
$(\, , \,)$ between $\h$ and $\h^*$  induces a non-degenerate pairing, still denoted by $(\, , \, )$,
between $\t_A$ and the vector subspace
$\sum_{i = 1} ^r \CC \alpha_i$ of $\h^*$, so we write $\t_A^* =  \sum_{i = 1} ^r \CC \alpha_i\subset \h^*$.  

Let $\lara_A$ be the 
unique symmetric bilinear form on $\t_A^*$ such that
\begin{equation}\label{eq:lara-A}
\la \alpha_i, \, \alpha_j\ra_A = d_i a_{i, j}, \hs \hs i, j \in [1, r].
\end{equation}
In particular, we have $\la \alpha_i, \alpha_i\ra_A = 2d_i$ for each $i \in [1, r]$. Comparing with the definition in 
\eqref{eq:Cartan-number} of  the Cartan numbers 
$a_{\beta, \xi}\in \t^*$ for  $\beta, \xi \in \t_A^*$ with
$\la \beta, \beta \ra_A \neq 0$, we have 
\[
a_{\alpha_i, \alpha_j} = a_{i, j}, \hs \forall \;\; i, j \in [1, r].
\]
Note that the definition of $\lara_A$ on $\t^*$, and thus that of $\pi_0$, depends on the choice of the symmetrizer 
$(d_i)_{i \in [1, r]}$ $A$ but we suppress the dependence in the notation $\lara_A$.

As in \eqref{eq:s-beta}, for each 
$\beta \in \t_A^*$ such that $\la \beta, \beta\ra_A \neq 0$, we have the reflection operator $s_\beta$ on $\t_A^*$.
In particular we have $s_i :=s_{\alpha_i}$ for each $i \in [1, r]$.
The sub-group $W$ of ${\rm GL}(\t^*)$ generated by $\{s_i: i \in [1, r]\}$ is called the Weyl group
associated to the generalized Cartan matrix $A$. Let $\calQ = \sum_{i=1}^r \ZZ \alpha_i \subset \t_A^*$.
Elements in $\{\alpha_1, \ldots, \alpha_r\}$ are called the simple roots, and elements in
\[
\Phi := \{w\alpha_i: w \in W, \,i \in [1, r]\}\subset \calQ
\]
are called {\it real roots}. 
If $\beta = w\alpha_i$ for some $w \in W$ and $i \in [1, r]$, one readily checks that $s_\beta = w s_i w^{-1} \in W$.
Let $\TT_A$ be the complex algebraic torus
with character lattice $\calQ$, and we 
identify the Lie algebra of $\TT_A$ with $\t_A= \h/\ker \underline{\alpha}$.
For any sequence ${\bf i} = (i_1, \ldots, i_n)$ in $[1, r]$, set $\bbeta({\bf i}) = (\beta_1({\bf i}), \ldots, \beta_n(\bf i)))
\in \calQ^n$,
where
\begin{equation}\label{eq:beta-alpha-j}
\beta_j({\bf i}) = s_{i_1}\cdots s_{i_{j-1}}\alpha_{i_j}, \hs \hs j \in [1, n].
\end{equation}
As $\langle \beta_j({\bf i}), \beta_j({\bf i})\rangle_A = \langle \alpha_{i_j}, \alpha_{i_j}\rangle
=2d_{i_j} \neq 0$ for every 
$j \in [1, n]$, $(\lara_A, \bbeta({\bf i}))$ is a $\TT_A$-action pair.

\begin{lemma}\label{lm:Cartan}
For any symmetrizable generalized Cartan matrix $A$ of size $r$ and 
any sequence ${\bf i} = (i_1, \ldots, i_n)$ in $[1, r]$, the sequence 
$\bgamma$ associated to $(\lara_A, \bbeta({\bf i}))$ defined in \eqref{eq:gamma-j} is
$(\alpha_{i_1}, \ldots, \alpha_{i_n})$, and $(\lara_A, \bbeta({\bf i}))$ is a distinguished 
(\autoref{def:disting}) and integral (\autoref{def:integral})
$\TT_A$-action pair.
\end{lemma}

\begin{proof}
Writing $\bbeta({\bf i}) = (\beta_1, \ldots, \beta_n)$ for simplicity, by \eqref{eq:beta-alpha-j} we  have
$\alpha_{i_j} = s_{\beta_1} \cdots s_{\beta_{j-1}} \beta_j$ for every $j \in [1, n]$. 
Comparing with \eqref{eq:gamma-j}, we have $\gamma_j = \alpha_{i_j}$ for every $j \in [1, n]$.
As the set of all simple roots is a basis of $\t^*$,  the pair $(\lara_A, \bbeta({\bf i}))$ is distinguished.
As for all $i, i' \in [1, n]$ and $i \neq i'$, the Cartan integers $a_{i, i'}$ 
are non-negative and $\langle \alpha_i, \alpha_i\rangle/\langle \alpha_{i'}, \alpha_{i'}\rangle=d_i/d_{i'}
\in \QQ_{>0}$, the pair $(\lara_A, \bbeta({\bf i}))$ is 
 integral. 
\end{proof}

\begin{definition}\label{defn:pair-Cartain-type}
{\rm
The $\TT_A$-action pair $(\lara_A, \bbeta({\bf i}))$, where $A$ is any $r \times r$ 
symmetrizable generalized Cartan matrix and ${\bf i}$ any finite sequence in $[1, r]$, 
is called {\it an action pair of  Cartan type}.
}
\end{definition}

Let  now $\TT$ and $\TT'$ be two complex algebraic tori with respective Lie algebras $\t$ and $\t^\prime$
and character lattices $X^*(\TT) \subset \t^*$ and $X^*(\TT^\prime) \subset (\t^\prime)^*$.

\begin{definition}\label{defn:factor-through}
{\rm
An $n$-dimensional $\TT$-action pair $(\lara, \bbeta = (\beta_1, \ldots, \beta_n))$ is said to 
{\it factor through} an $n$-dimensional $\TT^\prime$-action pair 
$(\lara^\prime, \bbeta^\prime = (\beta_1^\prime, \ldots, \beta_n^\prime))$ if
there exists a surjective morphism $\rho: \TT \to \TT^\prime$ of algebraic tori, inducing a linear map
$\rho^*: (\t^\prime)^* \to \t^*$, such that $\beta_j = \rho^*(\beta_j^\prime)$ for all
$j \in [1, n]$, and 
\begin{equation}\label{eq:rho-lara}
\langle \rho^* \xi, \,\, \rho^* \eta\rangle = \langle \xi, \, \,\eta\rangle^\prime \;\; \mbox{for all}\;\; 
\xi, \, \eta \in (\t^\prime)^*.
\end{equation}
}
\end{definition}

\begin{remark}\label{rk:factor-through}
{\rm 
If a $\TT$-action pair $(\lara, \bbeta)$ factors through 
a $\TT^\prime$-action pair $(\lara^\prime, \bbeta^\prime)$, then 
 $(\lara, \bbeta)$ and
$(\lara^\prime, \bbeta^\prime)$ define the same log-canonical Poisson structure
$\pi_0$,   and the set $\calS(\pi_0)$ associated to the $\TT$-action is the same
as that associated to the $\TT^\prime$-action. Moreover, an algebraic Poisson structure on $\CC^n$ is $\TT$-invariant if and only if it is $\TT^\prime$-invariant. 
\hfill $\diamond$
}
\end{remark}

\begin{lemma}\label{lm:factor-through-dist-int}
Suppose that a $\TT$-action pair $(\lara, \bbeta)$ factors through 
a $\TT^\prime$-action pair $(\lara^\prime, \bbeta^\prime)$. Then 
$(\lara^\prime, \bbeta^\prime)$ is distinguished and integral if and only if  $(\lara, \bbeta)$ is 
distinguished and integral.
\end{lemma}

\begin{proof}
Suppose that a $\TT$-action pair $(\lara, \bbeta)$ factors through 
a $\TT^\prime$-action pair $(\lara^\prime, \bbeta^\prime)$ via $\rho: \TT \to \TT'$ 
as in \autoref{defn:factor-through}, and let
$\bgamma = (\gamma_1, \ldots, \gamma_n) \in (\t^*)^n$ and 
$\bgamma' = (\gamma_1^\prime, \ldots, \gamma_n^\prime) \in ((\t^\prime)^*)^n$ be the corresponding
sequences defined in \eqref{eq:gamma-j}. It then follows from \eqref{eq:rho-lara} that 
$\gamma_j = \rho^* \gamma_j^\prime$ for every $j \in [1, n]$, and that 
$a_{\gamma_j, \gamma_k} = a_{\gamma_j^\prime, \gamma_k^\prime}$. Since the linear map
$\rho^*: (\t')^* \to \t^*$ is injective, 
$(\lara^\prime, \bbeta^\prime)$ is distinguished and integral if and only if  $(\lara, \bbeta)$ is 
distinguished and integral.
\end{proof}

\begin{proposition}\label{pr:factor-Cartan}
Let $\TT$ be any complex algebraic torus. A $\TT$-action pair factors through 
an action pair of Cartan type if and only if it
is distinguished and integral.
\end{proposition}

\begin{proof} Let $(\lara, \bbeta = (\beta_1, \ldots, \beta_n)$ be a $\TT$-action pair. If 
$(\lara, \bbeta)$ factors through an action pair of Cartan type, then
$(\lara, \bbeta)$ is distinguished and integral by \autoref{lm:factor-through-dist-int}.

Assume now that 
$(\lara, \bbeta)$ is distinguished and integral. 
Fix any listing of the elements in ${\rm Supp}(\bgamma) \subset X^*(\TT)$ as 
$\widetilde{\alpha}_1, \ldots, \widetilde{\alpha}_r$, so that $(\gamma_1, \ldots, \gamma_n) = (\widetilde{\alpha}_{i_1}, \ldots, \widetilde{\alpha}_{i_n})$ 
for some sequence ${\bf i} = (i_1, \ldots, i_n)$ in $[1, r]$. As $(\lara, \bbeta)$ is
integral, we have a well-defined $r \times r$ 
integral matrix $A=(a_{i, i'})$, where $a_{i, i} = 2$ for $i \in [1, r]$, and $a_{i, i'}$ 
for  $i, i' \in [1, r]$ and $i \neq i'$ is defined as
\begin{equation}\label{eq:aii}
-a_{i, i'} =  -a_{\gamma, \gamma'} = -\frac{2\langle \gamma, \gamma'\rangle}{\langle \gamma, \gamma\rangle}
\in \ZZ_{\geq 0}
\end{equation}
if $\gamma, \gamma'$ are the unique elements in  ${\rm Supp}(\bgamma)$ such that
$\gamma =\widetilde{\alpha}_i$ and $\gamma' = \widetilde{\alpha}_{i'}$. Setting 
$d_i = \frac{\langle \gamma, \gamma\rangle}{2} \in \QQ_{>0}$ if $\widetilde{\alpha}_i = \gamma \in {\rm Supp}(\bgamma)$,
we then have $d_i a_{i, i'} = d_{i'} a_{i', i}$ for all $i, i' \in [1, r]$.
Thus $A$ is a symmetrizable generalized Cartan matrix.

Let $\lara_A$ be the symmetric bilinear form on $\t_A^*$ in \eqref{eq:lara-A} and consider the
$\TT_A$-action pair $(\lara_A, \bbeta({\bf i}))$. Let $\rho: \TT \to \TT_A$ be the surjective morphism of algebraic tori 
induced by the lattice embedding
\[
X^*(\TT_A) = \sum_{i=1}^r \ZZ \alpha_i \;\;\longrightarrow  \sum_{\gamma \in {\rm Supp}(\bgamma)} \ZZ \gamma \subset X^*(\TT),\hs \alpha_i \longmapsto \widetilde{\alpha}_i,\hs i \in [1, r],
\]
and let $\rho^*$ be the corresponding injective linear map $\t_A^* \to \t^*$. It then follows from the definition of
$\lara_A$ and from \eqref{eq:aii} that 
$\langle \rho^* \alpha,  \rho^* \alpha'\rangle = \langle \alpha, \alpha'\rangle_A$ for all 
$\alpha,  \alpha' \in \t_A^*$. Moreover, if $i, j \in [1, n]$ are such that $\widetilde{\alpha}_i = \gamma$ 
and $\widetilde{\alpha}_j = \gamma'$ for $\gamma, \gamma' \in {\rm Supp}(\bgamma)$, then 
(see notation from $\S$\ref{ss:lara-beta})
\[
\rho^*(s_i(\alpha_{i'})) = \rho^*(\alpha_{i'} - a_{i, i'} \alpha_i) = \gamma' - a_{\gamma, \gamma'}\gamma
= s_{\gamma}(\gamma').
\]
It follows from \eqref{eq:beta-j-gamma-j} that $\rho^* (\bbeta_j({\bf i})) = \beta_j$ for $j \in [1, n]$. 
Thus the $\TT$-action pair $(\lara, \bbeta)$ factors through the $\TT_A$-pair $(\lara_A, \bbeta({\bf i}))$
via $\rho: \TT \to \TT_A$.
\end{proof}

\subsection{Symmetric Poisson CGL extensions of Cartan type}\label{ss:Poi-Cartan}
Let again $A = (a_{i, j})$ be an $r \times r$  symmetrizable generalized Cartan matrix with a given
symmetrizer $(d_i)_{i \in [1, r]}$, and let ${\bf i} = (i_1, \ldots, i_n)$ be a sequence in $[1, r]$.
Then with $\lara_A$ given in \eqref{eq:lara-A} and 
\[
\beta_j({\bf i}) = s_{i_1}\cdots s_{i_{j-1}}\alpha_{i_j}, \hs j \in [1, n],
\]
the log-canonical Poisson structure on $\CC^n$ defined by the $\TT_A$-pair $(\lara_A, \bbeta({\bf i}))$ is then given by
\begin{equation}\label{eq:pi0-alpha}
\pi^{(A,{\bf i})}_0 = -\sum_{1\leq j < k \leq n} \langle \beta_j({\bf i}), \, \beta_k({\bf i})\rangle_A 
\frac{\partial}{\partial x_j} \wedge \frac{\partial}{\partial x_k}
= -\sum_{1\leq j < k \leq n} \langle s_{i_1}\cdots s_{i_{j-1}}\alpha_{i_j}, \; s_{i_1}\cdots s_{i_{k-1}}\alpha_{i_k}\rangle_A 
\frac{\partial}{\partial x_j} \wedge \frac{\partial}{\partial x_k}.
\end{equation}
Note again that the definition of $\lara_A$, and thus that of $\pi_0^{(A,{\bf i})}$, depends on the choice of the symmetrizer 
$(d_i)_{i \in [1, r]}$ of the Cartan matrix $A$.
The matrix $Q = P^{-1}$ in \eqref{eq:P-Q} associated to $\bbeta({\bf i})$ is now
given by
\begin{equation}\label{eq:Q-Cartan}
Q = \left(\begin{array}{ccccc} 1 & a_{i_1, i_2} & \cdots & a_{i_1, i_{n-1}}  & a_{i_1,i_n}\\
0 & 1 &\cdots & a_{i_2, i_{n-1}}  & a_{i_2, i_n}\\
\cdots & \cdots & \cdots & \cdots & \cdots \\
0 & 0 & \cdots & 1 & a_{i_{n-1}, i_n}\\
0 & 0 & \cdots & 0 & 1\end{array}\right),
\end{equation}
and map $[1, n] \to [2, n] \sqcup \{+\infty\}, \,j \mapsto j^+$ in \eqref{eq:j-plus-def} is now  given by
\[
j^+ = \begin{cases} +\infty, & \;\;  \mbox{if} \;\; \{j' >j: \, i_{j'} = i_j\} = \emptyset,\\
{\rm min}\{j' >j:\, i_{j'} = i_j\}, & \;\; \mbox{if} \;\; \{j' >j: \, i_{j'} = i_j\} \neq \emptyset\end{cases}.
\]
Correspondingly, we have the lower triangular matrix 
\begin{equation}\label{eq:E-plus}
E_+ = (e_1 -e_{1^+}, \; \ldots, \; e_n-e_{n^+})
\end{equation}
and the set 
\begin{equation}\label{eq:J-disting}
\calJ = \calJ_{{\rm int}}  = \{j \in [1, n]: j^+ \leq n\}.
\end{equation}
By \autoref{pr:lara-beta}, the set $\calS(\pi^{(A,{\bf i})}_0)$ of all non-zero $(\CC^\times)^n$-weights in $\mH^2_{\pi^{(A,{\bf i})}_0}(\mX)^{\TT_A}$ is now
given by 
$\calS(\pi^{(A,{\bf i})}_0) = {\rm Col}_{\calJ} (QE_+) = \{\bth^{(j, j^+)}: j \in {\calJ}\}$, 
where for $j \in {\calJ}$, 
\[
\bth^{(j, j^+)} = Qe_j-Q_{e_{j^+}} = (0, \,\ldots,\, 0, \,-1,\, -a_{i_{j+1}, i_j}, \,\ldots,\, -a_{i_{j^+-1}, i_j},
\, -1, \, 0, \,\ldots,\, 0)^t.
\]
We can thus identify $\calJ$ with $\calS(\pi^{(A,{\bf i})}_0)$ by $j \mapsto \bth^{(j, j^+)}$ and identify 
$\CC^{\calS(\pi^{(A,{\bf i})}_0)}=\CC^{\calJ} = \{c = (c_j\in \CC)_{j \in \calJ}\}$.
Correspondingly, we denote the
maximal family of algebraic Poisson deformations of $\pi^{(A,{\bf i})}_0$ given in \autoref{thm:max-family-W1} by
$\pi_1^{(A, {\bf i})}(c)$, where $c \in \CC^{\calJ}$. 
For 
$c =  (c_j\in \CC)_{j \in \calJ} \in \CC^{\calJ}$, let
\begin{equation}\label{eq:pi1-c}
\pi_1^{(A, {\bf i})}(c) = \sum_{j \in \calJ} c_j \left(\prod_{j<k<j^+} x_k^{-a_{i_k, i_j}}\right)
\frac{\partial}{\partial x_j} \wedge \frac{\partial}{\partial x_{j^+}},
\end{equation}
where  $\prod_{j<k<j^+} x_k^{-a_{i_k, i_j}} =1$ if $j^+ = j+1$. 
Recall that $\TT_A$ act on $\CC^n$ 
by $t\cdot x_j = t^{\beta_j({\bf i})}x_j$ for $j \in [1, n]$. The following summary 
 on the family
$\pi_1^{(A, {\bf i})}(c)$  follows directly from 
\autoref{thm:max-family-W1} and \autoref{cor:member-summand}.

\begin{theorem}\label{thm:Cartan}
Let $A = (a_{i,j})$ be any $r \times r$ symmetrizable generalized Cartan matrix with symmetrizer $(d_i)_{i \in [1, r]}$, and let
${\bf i} = (i_1, \ldots, i_n)$ be any sequence
in $[1, r]$ with $n \geq 2$. 
Then for  every
$c \in \CC^{\calJ}$ there is a unique $\TT_A$-invariant
algebraic Poisson structure $\pi^{(A,{\bf i})}(c)$ on $\mX = \CC^n$ of the form
\[
\pi^{(A,{\bf i})}(c) = \pi^{(A,{\bf i})}_0 + \pi_1^{(A, {\bf i})}(c) + \cdots + \pi_{N}^{(A, {\bf i})}(c),
\]
where $N \geq 1$, 
$\pi^{(A,{\bf i})}_0$ is given in \eqref{eq:pi0-alpha}, $\pi_1^{(A, {\bf i})}(c)$ is given in \eqref{eq:pi1-c}, and 
for $m \in [2, N]$, 
$\pi_{m}^{(A, {\bf i})}(c) \in \mfX^2(\mX)^{\calS(\pi_0)_m}$ and is homogeneous polynomial in
$c$ with homogeneous degree $m$. Moreover,  for every $c \in \CC^{\calJ}$, 
$\pi^{(A, {\bf i})}(c)$ 
is a strongly symmetric $\TT_A$-Poisson CGL extension in the coordinates $(x_1, \ldots, x_n)$,
and these are all 
symmetric $\TT_A$-Poisson CGL extensions with log-canonical term $\pi^{(A,{\bf i})}_0$. 
\end{theorem}

\begin{notation}\label{nota:pi-bfi}
In the setting of \autoref{thm:Cartan}, for a given sequence ${\bf i} = (i_1, \ldots, i_n)$ in $[1, r]$, we 
set
\begin{equation}\label{eq:pi-bfi}
\pi^{(A, {\bf i})} = \pi^{(A, {\bf i})}(c=(-2\langle \alpha_{i_j}, \alpha_{i_j}\rangle_A)_{j \in \calJ})
\end{equation}
(we suppress the dependence on the symmetrizer in the notation).
In the terminology of \autoref{defn:normalized}, $\pi^{(A, {\bf i})}$
is the  maximal normalized admissible algebraic Poisson deformation of $\pi^{(A, {\bf i})}_0$.
\end{notation}

Consider now the Goodearl-Yakimov seed $({\bf y}, M)$ associated to 
$\pi = \pi^{(A, {\bf i})}$ as a symmetric $T_A$-Poisson CGL extension (see $\S$\ref{ss:mut-matrix}).  
The exchange set is then $\calJ$ in \eqref{eq:J-disting}, and the successor map 
$s_\pi$ is given by $s_\pi(j) = j^+$ for $j \in \calJ$ and $s_\pi(j) = +\infty$ otherwise. Moreover, 
$E_+$ coincides with the matrix 
$E_{\pi}$ defined in \eqref{eq:E-pi}. As a special case of \autoref{thm:main-M}, 
we have the following description of the mutation
matrix $M$. 

\begin{proposition}\label{pr:BS-M}
The mutation matrix  in the Goodearl-Yakimov seed $({\bf y}, M)$ 
associated to 
the symmetric $T_A$-Poisson CGL extension $\pi^{(A, {\bf i})}$ is given by
\[
M = (E_+^tQE_+)_{n \times \calJ},
\]
where the $n\times n$ matrices $E_+$ and $Q$ are respectively given in  \eqref{eq:E-plus}
and \eqref{eq:Q-Cartan}, and the set $\calJ \subset [1, n]$
is given in \eqref{eq:J-disting}. 
Writing $M = (m_{j', j})_{j'\in [1, n], j\in \calJ}$, one also has
\[
m_{j', j} = \begin{cases} 1, & \hs  (j')^+=j, \\
-1, & \hs j' = j^+,\\
a_{i_{j'}, i_j}, &\hs  j' < j < (j')^+ < j^+, \\
-a_{i_{j'}, i_j}, &\hs  j < j' < j^+ < (j')^+ \;\; (\mbox{including when} \; (j')^+ = +\infty),\\
0, & \hs \mbox{otherwise}.\end{cases}
\]
\end{proposition}

Given a symmetrizable generalized Cartan matrix $A$, Shen and Weng introduced in  \cite{ShenWeng:DBS} 
(several versions of) {\it double Bott-Samelson cells}
for Kac-Moody groups associated to $A$ and studied cluster structures on them via cluster ensembles. 
In particular, associated to a sequence ${\bf i} = ({i_1}, \ldots, {i_n})$ in $[1, r]$ one has the (single)
Bott-Samelson cell ${\rm Conf}_{{\bf i}}^e(\calC_{\rm sc})$
defined in \cite{GSW:aug} which is isomorphic to a principal Zariski open subset of $\CC^n$, where sc stands for ``simply connected". An initial mutation matrix
for the cluster structure on ${\rm Conf}_{{\bf i}}^e(\calC_{\rm sc})$, which we denote by
$M_{SW}$, is described on \cite[Page 51]{ShenWeng:DBS}.  

\begin{corollary}\label{cor:M-Shen-Weng}
For any $r \times r$ symmetrizable generalized Cartan matrix $A$ and any sequence ${\bf i} =(i_1, \ldots, i_n)$ in 
$[1, r]$, the Goodearl-Yakimov initial mutation matrix $M$  
associated to $\pi^{(A, {\bf i})}$ as 
a symmetric $T_A$-Poisson CGL extension  coincides with the initial mutation matrix $M_{{\rm SW}}$. 
\end{corollary}

\begin{proof}
The statement follows from a direct comparison between the mutation matrix $M$ in \autoref{pr:BS-M} and 
the description of $M_{{\rm SW}}$ on \cite[Page 51]{ShenWeng:DBS}.
\end{proof}

Let again $\TT$ be an arbitrary algebraic torus over $\CC$ acting on $\CC^n$ via 
by $\bbeta = (\beta_1, \ldots, \beta_n) \in X^*(\TT)^n$.

\begin{definition}\label{defn:Poi-CGL-Cartan}
{\rm
A symmetric $\TT$-Poisson CGL extension $(\CC^n, \pi)$ is said to be {\it of Cartan type} if 
the action of $\TT$ on $\CC^n$ factors through a surjective morphism $\rho: \TT \to \TT_A$ of algebraic tori
for some $r \times r$ symmetrizable generalized Cartan matrix $A$, and if 
\[
\pi = \pi^{(A, {\bf i})}(c)
\]
as given in \autoref{thm:Cartan} for some sequence ${\bf i}=
(i_1, \ldots, i_n)$ in $[1, r]$ and some $c \in \CC^{\calJ}$. 
}
\end{definition}

\begin{proposition}\label{pr:Cartan-type}
A symmetric $\TT$-Poisson CGL extension $(\CC^n, \pi)$ is  of Cartan type if and only if the log-canonical
term $\pi_0$ of $\pi$ is defined by a distinguished and integral $\TT$-action pair.
\end{proposition}

\begin{proof}
Let $(\CC^n, \pi)$ be a symmetric $\TT$-Poisson CGL extension with log-canonical term $\pi_0$.
Suppose first that $\pi$ is of Cartan type as in \autoref{defn:Poi-CGL-Cartan}. Then
$\pi_0 = \pi^{(A, {\bf i})}_0$ given in \eqref{eq:pi0-alpha}, and 
$\beta_j = \rho^* \beta_j({\bf i})$ for every $j \in [1, n]$. As $\rho^*: \t_A^* \to \t^*$ is injective, 
there exists a symmetric bilinear form $\lara$ on $\t^*$ such that  the 
$\TT$-action pair $(\lara,\bbeta)$ factors through the $\TT_A$-action pair $(\lara_A, \bbeta({\bf i)})$
and $\pi_0$ is defined by $(\lara,\bbeta)$. By \autoref{pr:factor-Cartan}, $(\lara,\bbeta)$ is 
distinguished and integral.

Conversely, suppose that $\lara$ is a symmetric bilinear form on $\t^*$ such that 
$\pi_0$ is defined by $(\lara,\bbeta)$ and that $(\lara,\bbeta)$
is distinguished and integral. By \autoref{pr:factor-Cartan}, there exist
some $r \times r$ symmetrizable generalized Cartan matrix $A$ and some sequence
${\bf i} = (i_1, \ldots, i_n)$ in $[1, r]$ such that 
$\TT$-action pair $(\lara,\bbeta)$ factors through the $\TT_A$-action pair $(\lara_A, \bbeta({\bf i)})$ by
a surjective morphism $\rho: \TT \to \TT_A$ of algebraic tori. 
Then $\pi_0 = \pi^{(A, {\bf i})}_0$, and $\pi$ is a symmetric $\TT_A$-Poisson CGL extension. 
By \autoref{thm:Cartan}, $\pi =  \pi^{(A, {\bf i})}(c)$ for some $c \in \CC^{\calJ}$.
\end{proof}

\subsection{The standard Poisson structure on generalized Schubert cells}\label{ss:BS-cells}
We keep the notation as in $\S$\ref{ss:Poi-Cartan} but assume now that the generalized Cartan matrix $A$ is of finite type. 
In such a case, for every sequence ${\bf i} = (i_1, \ldots, i_n)$ in $[1, r]$,
the Bott-Samelson cell ${\rm Conf}_{{\bf i}}^e(\calC_{\rm sc})$ is 
isomorphic \cite{LMY:groupoid} to the generalized
Bruhat cell $\calO^{(s_{i_1}, \ldots, s_{i_n})}$ {\it of Bott-Samelson type} as defined 
in \cite{LM:T-leaves, EL:BS} and will be reviewed below. Moreover, $\calO^{(s_{i_1}, \ldots, s_{i_n})}$ carries
the so-called {\it standard Poisson structure}, which, when regarded as a Poisson structure on $\CC^n$ via 
the so-called Bott-Samelson parametrization, is shown in \cite{EL:BS} to be a strongly symmetric $T_{\rm ad}$-Poisson CGL extension, where $T_{\rm ad}$ is a maximal 
torus of the semi-simple Lie group associated to $A$ and of adjoint type. 
In this sub-section, after recalling some definitions, we show that the standard Poisson structure on 
$\calO^{(s_{i_1}, \ldots, s_{i_n})} \cong \CC^n$
coincides with $\pi^{(A, {\bf i})}$, the 
 maximal normalized admissible algebraic Poisson deformation of $\pi^{(A, {\bf i})}_0$ given in
 \eqref{eq:pi-bfi}.

Assume thus that the $r \times r$ generalized Cartan matrix $A = (a_{i, j})$ is of finite type, and let $(d_i)_{i \in [1, r]}$ be a symmetrizer of $A$. Let $G$ be a 
connected and simply connected complex semi-simple Lie group, with a maximal torus  $T \subset G$ 
 with Lie algebra $\t$ and a Borel subgroup $B \subset G$ 
containing $T$, such that the set of simple roots $\{\alpha_1, \ldots, \alpha_r\} \subset \t^*$
corresponding to $B$ has $A$ as its Cartan matrix. Let $T_{\rm ad} = T/Z(G)$, where $Z(G)$ is the center of $G$, so that the character lattice of $T_{\rm ad}$ is the 
root lattice $\calQ = \sum_{i=1}^r \ZZ\alpha_i \subset \t^*$. 

Let $\lara_\g$ be the invariant symmetric bilinear form on the Lie algebra $\g$
such that $\la\alpha_i, \alpha_i \ra = 2d_i$ for $i \in [1, r]$, where 
$\lara$ is the non-degenerate symmetric bilinear form on $\t^*$ induced by $\lara_\g$.
It is well-known that the choice of $(T, B, \lara_\g)$ 
gives rise to the so-called {\it standard multiplicative Poisson structure} $\pist$ on  $G$,
and $(G, \pist)$ is called a {\it standard complex semi-simple Poisson Lie group} \cite{dr:Hamil, LM:mixed}. 
Among the many interesting Poisson varieties associated to the Poisson Lie
$(G, \pist)$ are the  twisted products 
\[
F_n = G \times_B  \cdots \times_B G/B
\]
of  the flag variety $G/B$, where $n \in \ZZ_{\geq 1}$ and $F_n$ is
the quotient of $G^n$ by the right action of $B^n$ given by
\[
(g_1, g_2, \ldots, g_n) \cdot (b_1, b_2, \ldots, b_n) = (g_1b_1, \, b_1^{-1}g_2 b_2, \ldots, \, b_{n-1}^{-1}g_nb_n),
\hs g_1, \ldots, g_n \in G, \, b_1, \ldots, b_n \in B.
\]
Denote the projection from $G^n$ to $F_n$ by 
$G^n \ni (g_1, \, g_2, \, \ldots, \, g_n) \mapsto [g_1, \, g_2, \, \ldots, \, g_n] \in F_n$.
It is shown in \cite[$\S$7]{LM:mixed} that the product Poisson structure $(\pist)^n$ on $G^n$ projects to a well-defined
Poisson structure, denoted as $\pi_n$, on $F_n$, and the $T$-action on $F_n$ given by
\[
t\cdot [g_1, \, g_2, \, \ldots, \, g_n] = [tg_1, \, g_2, \, \ldots, \, g_n], \hs t \in T, \, g_1, g_2, \ldots, g_n \in G,
\]
 preserves the Poisson structure $\pi_n$. 
One of the most important properties of $(F_n, \pi_n)$ is its decomposition
\[
F_n = \bigsqcup_{(u_1, \ldots, u_n) \in W^n} \calO^{(u_1, \ldots, u_n)}
\]
 into the so-called 
{\it generalized Schubert cells} $\calO^{(u_1, \ldots, u_n)}$, where 
$\calO^{(u_1, \ldots, u_n)} = (Bu_1B) \times_B \cdots \times_B(Bu_nB)/B$.
Each generalized Schubert cell 
$\calO^{(u_1, \ldots, u_n)}$ is a $T$-invariant Poisson submanifold of $(F_n, \pi_n)$ and is a
finite union of $T$-leaves of  $(F_n, \pi_n)$. 
The restriction of $\pi_n$ to a generalized Schubert cell $\calO^{(u_1, \ldots, u_n)} \subset F_n$ 
is called the {\it standard Poisson
structure on $\calO^{(u_1, \ldots, u_n)}$}. We refer to 
\cite{LM:mixed, LM:T-leaves}, and in particular to \cite[Theorem 1.1]{LM:T-leaves}, for more details on  $(F_n, \pi_n)$
and on the $T$-leaf decompositions of generalized Schubert cells. Note  
that the $T$-action on $F_n$ descends to a $T_{\rm ad}$-action on $F_n$, and the $T_{\rm ad}$-leaves of $(F_n, \pi_n)$ are the same as
the $T$-leaves of $(F_n, \pi_n)$.

A generalized Schubert cell $\calO^{(u_1, \ldots, u_n)}$ is said to be {\it of Bott-Samelson type}
\cite{EL:BS} if every $u_j$ is a simple reflection in $W$.   Following \cite{ShenWeng:DBS}, we also refer to
generalized Schubert cells of Bott-Samelson type as {\it Bott-Samelson cells}. 
For any $(u_1, \ldots, u_n) \in W^n$,  it is shown in 
\cite[$\S$1,4]{EL:BS} that any choice 
of a reduced word of
$u_j$ for each $j \in [1, n]$ gives rise to a $T$-equivariant Poisson isomorphism from the generalized Schubert cell
$\calO^{(u_1, \ldots, u_n)}$ to a Bott-Samelson cell $\calO^{(s_{i_1}, \ldots, s_{i_l})}$, 
where $(s_{i_1}, \ldots, s_{i_l})$ is the 
concatenation of the reduced words of the $u_j$'s. For the study of the standard Poisson structures on 
generalized Schubert cells, it is thus enough to study only that on Bott-Samelson cells.

To recall some results from \cite{EL:BS}
on the standard Poisson structure on Bott-Samelson cells, we fix, for each $i \in [1, r]$, 
root vectors $e_{\alpha_i}$ for $\alpha_i$ and 
$e_{-\alpha_i}$ for $-\alpha_i$ such that 
$[e_{\alpha_i}, e_{-\alpha_i}] = \alpha_i^\vee \in \t$, and 
let 
\[
\overline{s}_i = \exp(-e_{\alpha_i})\exp(e_{-\alpha_i})\exp(-e_{\alpha_i}).
\]
Then $\overline{s}_i$
is a representative of $s_i \in W$ in the normalizer subgroup $N_G(T)$ of $T$ in $G$.

Let ${\bf i} = (i_1, \ldots, i_n)$ be
any sequence in $[1, r]$. One then has the 
{\it Bott-Samelson parametrization} \cite{EL:BS}
\begin{equation}\label{eq:BS-para}
\CC^n \longrightarrow \calO^{(s_{i_1}, \ldots, s_{i_n})}, \;\; 
(x_1, \ldots, x_n) \longmapsto [\exp(x_1e_{\alpha_{i_1}}) \overline{s}_{i_1},  \; \ldots, \; 
\exp(x_ne_{\alpha_{i_n}})\, \overline{s}_{i_n}],
\end{equation}
and the resulting coordinates $(x_1, \ldots, x_n)$ on $\calOs$ are called the {\it Bott-Samelson 
coordinates} on $\calOs$. 
Recall the subset $\calJ \subset [1, n]$ associated to ${\bf i}$ defined in 
\eqref{eq:J-disting}.

\begin{theorem}\label{thm:BS-cells}
Let ${\bf i} = (i_1, \ldots, i_n)$ be any sequence in $[1, r]$. Under the identification  $\calOs \cong \CC^n$ via
\eqref{eq:BS-para}, the standard Poisson structure on the Bott-Samelson cell $\calOs$ 
coincides with  $\pi^{(A,{\bf i})}$ on $\CC^n$ given in \eqref{eq:pi-bfi}.
\end{theorem}

\begin{proof}
Denote by $\pi$ the standard Poisson structure on $\calOs \cong \CC^n$ via the 
 Bott-Samelson parametrization. By \cite[Theorem 5.12]{EL:BS}, $\pi$ is a symmetric 
$\TT$-Poisson CGL extension in the Bott-Samelson coordinates $(x_1, \ldots, x_n)$ and that 
the log-canonical term of $\pi$ is $\pi_0$ given as in \eqref{eq:pi0-alpha}.
The set $\mW(\pi)$ of all the non-zero $\CCsn$-weights of the monmial terms of $\pi$
is explicitly computed in \cite[Theorem 4.14]{EL:BS}. More precisely,
 for $1 \leq j < k \leq n$, and 
\begin{equation}\label{eq:bfw-BS}
\bfw = (0, \ldots, -1, \bfw_{j+1}, \ldots, \bfw_{k-1}, -1, \ldots, 0)^t,
\end{equation}
where the two $-1$ entries are at the $j^{\rm th}$ and the $k^{\rm th}$ places and $\bfw_l \in \ZZ_{\geq 0}$ for $l \in [j+1, k-1]$, let 
\begin{equation}\label{eq:V-bff}
v_\bfw = \left(\prod_{l=j+1}^{k-1} x_l^{\bfw_l}\right)\frac{\partial}{\partial x_j} \wedge \frac{\partial}{\partial x_k}\in \mfX^2(\mX)^{{\bf w}},
\end{equation}
and recall from \autoref{ex:hideous} the sequence 
$\eta^{(\bfw)} = \left(\eta^{(\bfw)}_j=\alpha_{i_j},  \,\eta^{(\bfw)}_{j+1},\, \ldots, \, \eta^{(\bfw)}_{k-1}\right)$ in $\t^*$ associated to
$\bfw$. Let $\Phi_+\subset \t^*$ be the set of positive roots. For $1 \leq j < k \leq n$,
let $\mW_{j, k}$ be the set of all $\bfw$  in \eqref{eq:bfw-BS}   such that
$\eta^{(\bfw)}_l \in \Phi_+$ for every $l \in [j, k-1]$ and $\eta^{(\bfw)}_{k-1} = \alpha_{i_k}$.
By \cite[Theorem 4.10 and Theorem 4.14]{EL:BS}, 
\[
\pi =\pi_0 +\sum_{1\leq j < k \leq n}\; \sum_{\bfw \in \mW_{j, k}} c_{\bfw}\, V_{\bfw},
\]
where $c_{\bfw} \in \CC^\times$ for every $\bfw \in \bigcup_{1 \leq j < k \leq n} \mW_{j, k}$.
In particular, for every $j \in \calJ$ (see \eqref{eq:J-disting}), one has
$\bth^{(j, j+)} \in \mW_{j, j^+}$
with $\eta^{\bth^{(j, j^+)}} = (\alpha_{i_j}, \alpha_{i_j}, \ldots, \alpha_{i_j})$. 
Moreover, by \cite[Theorem 4.14]{EL:BS},  $c_{\bth^{(j, j^+)}}=-\la \alpha_{i_j}, \alpha_{i_j}\ra$ for every $j \in \calJ$. 
Thus $\pi$ is the unique maximal $\calS(\pi_0)$-admissible algebraic Poisson deformation of $\pi_0$ along 
$\pi_1^{(A, {\bf i})}(c) =\sum_{j \in J} c_jV_{\bth^{(j, j^+)}}$ with $c_j=-\la \alpha_{i_j}, \alpha_{i_j}\ra
=-2d_{i_j}$ for $j \in \calJ$.   In other words, $\pi = \pi^{(A, {\bf i})}$ given in \eqref{eq:pi-bfi}.
\end{proof}

\begin{example}\label{ex:G2}
{\rm
Consider $G_2$ with two simple roots $\alpha_1$ and $\alpha_2$ such that $\langle \alpha_2, \alpha_2\rangle
= 3\langle \alpha_1, \alpha_1\rangle = 6$, and let ${\bf i} = (1, 2, 1,2,1,2)$. Then (see 
\cite[Example 5.23]{EL:BS}) the Poisson structure 
$\pi^{(A, {\bf i})}$ on $\CC^n$ is given by 
\begin{align*}
&\{x_{1},x_{2}\} = {-3x_{1}x_{2}}, \hs 
\{x_{1},x_{3}\} = { -x_{1}x_{3}}-2x_{2}, \hs
\{x_{1},x_{4}\} = -6x_{3}^2, \\
&\{x_{1},x_{5}\} = { x_{1}x_{5}}-4x_{3}, \hs
\{x_{1},x_{6}\} = { 3x_{1}x_{6}}-6x_{5}, \hs
\{x_{2},x_{3}\} = { -3x_{2}x_{3}}\\
&\{x_{2},x_{4}\} ={  -3x_{2}x_{4}}-6x_{3}^3, \hs 
\{x_{2},x_{5}\} = -6x_{3}^2, \hs
\{x_{2},x_{6}\} ={  3x_{2}x_{6}}-18x_{3}x_{5}+6x_{4}, \\
&\{x_{3},x_{4}\} ={  -3x_{3}x_{4}}, \hs
\{x_{3},x_{5}\}= { -x_{3}x_{5}}-2x_{4}, \hs 
\{x_{3},x_{6}\} = -6x_{5}^2, \\
&\{x_{4},x_{5}\} = { -3x_{4}x_{5}}, \hs
\{x_{4},x_{6}\} ={ -3x_{4}x_{6}} -6x_{5}^3, \hs
\{x_{5},x_{6}\} = { -3x_{5}x_{6}}.
\end{align*}
}
Write $\pi^{(A, {\bf i})} = \pi_0 + \pi_{\rm tail}$, we have 
 $\pi_{{\rm tail}}  =\pi_1 +  \pi_2 + \pi_3 +\pi_4+\pi_5$, with 
\begin{align*}
 \pi_1 &= -2x_2 \frac{\partial}{\partial x_1} \wedge \frac{\partial}{\partial x_3} -6x_3^3\frac{\partial}{\partial x_2} \wedge \frac{\partial}{\partial x_4}
-2x_4 \frac{\partial}{\partial x_3} \wedge \frac{\partial}{\partial x_5}-6x_5^3 \frac{\partial}{\partial x_4} \wedge \frac{\partial}{\partial x_6},\\
\pi_2 & = -6x_3^2 \frac{\partial}{\partial x_1} \wedge \frac{\partial}{\partial x_4}-6x_3^2 \frac{\partial}{\partial x_2} \wedge \frac{\partial}{\partial x_5}
-6x_5^2 \frac{\partial}{\partial x_3} \wedge \frac{\partial}{\partial x_6},\\
\pi_3 & = -4x_3 \frac{\partial}{\partial x_1} \wedge \frac{\partial}{\partial x_5},\hs
\pi_4  = -18 x_3x_5 \frac{\partial}{\partial x_2} \wedge \frac{\partial}{\partial x_6},\\
\pi_5 & = -6x_5 \frac{\partial}{\partial x_1} \wedge \frac{\partial}{\partial x_6} + 6x_4 \frac{\partial}{\partial x_2} \wedge \frac{\partial}{\partial x_6},
\end{align*}
The four $(\CC^\times)^6$-weights appearing in $\pi_1$ form the set $\calS(\pi_0) = \{\bth^{(1, 3)},
\bth^{(2, 4)}, \bth^{(3, 5)}, \bth^{(4, 6)}\}$ with
\begin{align*}
\bth^{(1, 3)} &= (-1, 1, -1, 0, 0, 0)^t, \hs 
\bth^{(2, 4)} =(0, -1, 3, -1, 0, 0)^t, \\
\bth^{(3, 5)} &=(0, 0, -1, 1, -1, 0)^t, \hs 
\bth^{(4, 6)} =(0, 0, 0, -1, 3, -1)^t.
\end{align*}

These weights are encoded by the smoothing diagram depicted in \autoref{fig:SmDiagG2}.

\begin{figure}[h]
    \centering
\pgfdeclarelayer{background layer}
\pgfdeclarelayer{foreground layer}
\pgfsetlayers{background layer,main,foreground layer}
\begin{tikzpicture}[baseline=-1ex,
rotate=360/2/6,
scale=0.9,line join = round 
    ] 
    
    \begin{pgfonlayer}{main}
\foreach \n in {1,...,6}
{
    \coordinate (v\n) at ({-90+((\n-1)*360/6)}:1.5);
}
\node[below right] at (v1) {$1$};
\node[right] at (v2) {$2$};
\node[above right] at (v3) {$3$};
\node[above left] at (v4) {$4$};
\node[left] at (v5) {$5$};
\node[below left] at (v6) {$6$};

\foreach \n in {1,...,6}
{
	\foreach \m in {1,...,6}
	\draw (v\n) -- (v\m);
}
\end{pgfonlayer}

\draw[thick,red,-] (v2) ++({0.4*cos(90)},{0.4*sin(90)})  arc[start angle=90,end angle=210,radius=0.4];
\draw[thick,red,-] (v3) ++({0.4*cos(150)},{0.4*sin(150)})  arc[start angle=150,end angle=270,radius=0.4];
\draw[thick,red,-] (v3) ++({0.25*cos(150)},{0.25*sin(150)})  arc[start angle=150,end angle=270,radius=0.25];
\draw[thick,red,-] (v3) ++({0.55*cos(150)},{0.55*sin(150)})  arc[start angle=150,end angle=270,radius=0.55];
\draw[thick,red,-] (v4) ++({0.4*cos(210)},{0.4*sin(210)})  arc[start angle=210,end angle=330,radius=0.4];
\draw[thick,red,-] (v5) ++({0.4*cos(270)},{0.4*sin(270)})  arc[start angle=270,end angle=390,radius=0.4];
\draw[thick,red,-] (v5) ++({0.25*cos(270)},{0.25*sin(270)})  arc[start angle=270,end angle=390,radius=0.25];
\draw[thick,red,-] (v5) ++({0.55*cos(270)},{0.55*sin(270)})  arc[start angle=270,end angle=390,radius=0.55];

\draw[very thick, blue] (v1)--(v3);
\draw[very thick, blue] (v2)--(v4);
\draw[very thick, blue] (v3)--(v5);
\draw[very thick, blue] (v4)--(v6);

\begin{pgfonlayer}{foreground layer}
\foreach \n in {1,...,6}
{
	\draw[fill] (v\n) circle (0.03);
}
\end{pgfonlayer}
 \end{tikzpicture}
    \caption{Smoothing diagram appearing in \autoref{ex:G2}}
    \label{fig:SmDiagG2}
\end{figure}

The  $(\CC^\times)^6$-weights of appearing in $\pi_2, \pi_3, \pi_4, \pi_5$ are 
\begin{align*}
&\bth^{(1, 3)}+\bth^{(2, 4)}, \;\;\bth^{(2, 4)}+\bth^{(3, 5)},\;\; \bth^{(3, 5)}+\bth^{(4, 6)},\;\;
\bth^{(1, 3)}+\bth^{(2, 4)}+\bth^{(3, 5)}, \;\; \bth^{(2, 4)} + 2\bth^{(3, 5)}+\bth^{(4, 6)},\\
&\bth^{(1, 3)} + \bth^{(2, 4)} +2\bth^{(3, 5)} + \bth^{(4, 6)},\;\; \bth^{(2, 4)} + 3\bth^{(3, 5)} + \bth^{(4, 6)}.
\end{align*}
\hfill $\diamond$
\end{example}

\bibliographystyle{alpha}
\bibliography{myref-Mykola}

\end{document}